\documentclass[11pt]{article}
 \usepackage{epsfig}
 \usepackage{amssymb,amsmath,amsthm,amscd}
 \usepackage{latexsym}

\newtheorem{thm}{Theorem}[section]

\newtheorem{lem}[thm]{Lemma}

\theoremstyle{definition}
\newtheorem{defn}[thm]{Definition}

\newcommand{\be}{\begin{equation}}
\newcommand{\ee}{\end{equation}}

\newcommand{\vt}{{\widetilde{v}}}
\newcommand{\ut}{{\widetilde{u}}}

 \newcommand{\R}{\mathbb{R}}
\newcommand{\N}{\mathbb{N}}

\def \calf {{  {\mathcal{F}} }}

\def \calftwo {{  {\mathcal{F}}  }}
\def \cala {{  {\mathcal{A}}  }}
\def \calb {{  {\mathcal{B}}  }}
\def \cald {{  {\mathcal{D}}  }}

 \def  \ltwo {{L^2 (\R^n)}}
  \def  \hone {{H^1 (\R^n)}}
  
  \newcommand{\ii}{\int_{\R^n}}

\def \thtwot {{  \theta_{t}  }}

\begin{document}

\begin{titlepage}
\title{\Large\bf  Multivalued Non-Autonomous Random Dynamical Systems
for   
Wave Equations without Uniqueness}
\vspace{7mm}

\author{
Bixiang Wang  
\\
\\
Department of Mathematics, New Mexico Institute of Mining and
Technology \vspace{1mm}\\ Socorro,  NM~87801, USA \vspace{1mm}\\
Email: bwang@nmt.edu}
\date{}
\end{titlepage}

\maketitle

\begin{abstract}
This paper deals with the multivalued non-autonomous random dynamical 
system generated by the non-autonomous stochastic wave equations
on unbounded domains, which  has 
a non-Lipschitz nonlinearity with    critical exponent  in the
three dimensional case.
We introduce the concept of weak upper semicontinuity of multivalued functions
and  use   such continuity   to prove  the measurability of multivalued functions
from a metric space to a separable Banach space.
By this approach,   we  show   the  measurability of pullback attractors
of the  multivalued random dynamical system of the wave equations regardless of
the completeness of the underlying probability space.
The asymptotic compactness of solutions  is proved by the method of energy equations,
and the  difficulty  caused by the non-compactness of Sobolev embeddings 
on $\R^n$ is overcome by the uniform estimates on the tails of solutions.
\end{abstract}



\baselineskip=1.2\baselineskip

 \section{Introduction}
\setcounter{equation}{0}

In this paper, we study the existence and measurability of random attractors
of non-autonomous stochastic wave equations with critical non-Lipschitz
nonlinearity on $\R^n$ $(1\le n\le 3$):
 \begin{equation}
\label{intro1}
u_{tt} + \alpha u_t -\Delta u + \lambda u
+ f(x, u) = g(t, x) +  \varepsilon u \circ  {\frac {dw}{dt}},
\ \  t >\tau,
\end{equation}
  with  initial
 conditions
 \begin{equation}
 \label{intro2}
 u( \tau, x ) =u_0(x), \quad u_t(\tau, x ) = u_1(x),
 \end{equation}
 where
  $  \alpha$,  $\lambda$ and $\varepsilon$
  are     positive  constants,
 $g  $ is a 
 non-autonomous deterministic  external term
 in   $L^2_{loc} (\R, \ltwo)  $, 
$f$ is a nonlinear function which
has critical growth rate and is not 
  Lipschitz continuous,  and
 $w $   is  a  standard real-valued
 Wiener process.
 The symbol $\circ$ indicates that    the stochastic term  in \eqref{intro1}
 is understood in the sense of Stratonovich\rq{}s integration.

 The main assumption of this paper is that   the nonlinearity $f$
 is continuous but not Lipschitz continuous, which leads to the non-uniqueness
 of solutions of \eqref{intro1}-\eqref{intro2}.
 Consequently, the dynamical system associated with \eqref{intro1}-\eqref{intro2}
 becomes  a multivalued (or set-valued) function. This introduces many
 difficulties   for  proving the   measurability  of solutions as well as 
 the measurability  of attractors 
 because the measurability of set-valued functions are much more 
 involved than single-valued functions.
 For instance,  under general conditions,
   the measurability of attractors for stochastic
   equations without uniqueness
 still   remains open  
 in  \cite{car1, car4, car6,  car8, wangy1}
 when  the underlying probability space is not complete.
 Indeed, the measurability of attractors in these cases requires either
 the probability space  be  complete \cite{car1, car4, car8, wangy1}
  or the equation satisfy a  very restrictive
 condition \cite{car1}.
 In the present  paper,  we will solve this problem and prove an abstract
 result on the  measurability  
of  set-valued functions from a metric space to a separable  Banach space.
 Then we apply the abstract result to  further prove the measurability of attractors 
 for set-valued random dynamical  systems when the 
 underlying   probability space is not complete.
 To derive our result,  we introduce the concept of
 weak upper semicontinuity and prove 
 such continuity is sufficient to  
 ensure the measurability of set-valued maps
 (see Definition \ref{wuc} and Theorem \ref{wusc}).
In the end, we will   prove  the  measurability  
of attractors
    of  the  non-autonomous set-valued random dynamical  system associated with
    the stochastic wave equations \eqref{intro1}-\eqref{intro2}
    regardless of  the completeness of the underlying probability space.
    It is worth noticing that our approach can be applied to a wide class of
    stochastic equations without uniqueness,  including those considered
    in \cite{car1, car4, car6,  car8, wangy1}.
  
  In addition to the non-Lipschitz continuity,
  in this paper,  the nonlinearity $f$ is also  allowed to have
  a growth order $\gamma =3$  when  $n=3$, which is referred
  to as the critical exponent in the literature.  In this case,  $f$
  maps 
  $\hone$ to $ \ltwo$ as a continuous function but not compact.
  This introduces another difficulty for proving the asymptotic
  compactness of solutions. The problem caused by the critical
  nonlinearity will be solved by the  method of energy equations,
  which was developed by Ball in \cite{bal1} for the  deterministic 
  wave  equations. Of course, for the  non-autonomous stochastic equation
  \eqref{intro1}, the energy equation is much more  involved
  because it contains deterministic non-autonomous terms as well as
  random terms (see  \eqref{ener1}).
  We will use the ergodicity of the 
  Ornstein-Uhlenbeck process to deal with the random terms and then
  employ the energy equation \eqref{ener1} to  derive the pullback
  asymptotic compactness of solutions in $\hone \times \ltwo$.
  
  The equation \eqref{intro1} is defined on the whole of $\R^n$,
  which introduces extra obstacle to prove the asymptotic compactness
  of solutions since 
  Sobolev embeddings are not compact on $\R^n$. This difficulty will be overcome
  by the uniform  estimates on the tails of solutions as in \cite{wan1, wan9}.

  The concept of pullback attractor
  for autonomous stochastic equations
   was developed  in \cite{cra2, fla1, schm1}, and
   then extensively studied  in  
    \cite{arn1, bat1, bat2, car1, car2, 
   car3, car4, car5, chu2, chu3, cra1, cra2, 
    fla1, gar1, gar2, gar3,  huang1, kloe1, lv1, schm1, shen1,  wan2, wan4}.
     This concept was extended to autonomous stochastic equations without uniqueness
     in \cite{car1, car4, car6, car7,  car8}, 
      to  single-valued non-autonomous 
      random systems   in \cite{dua3, wan5, wan9},
      and to set-valued  non-autonomous 
      equations  in \cite{car6, wangy1}.
      In the present paper, we investigate pullback attractors
      for the non-autonomous stochastic wave equations without uniqueness.
 The reader is referred to 
 \cite{arr1, bab1, bal1, chu1, fei1, fei3, fei4, 
  hal1,  kha1, pri1, pri2, pri3,  sel1, sun1, sun2, 
  tem1} for global attractors of deterministic wave equations.

        In the next section,  we 
        prove  a measurability result
        for set-valued functions from a metric
        space to a separable Banach space which
        can be used to  derive the measurability
        of solutions of stochastic equations without
        uniqueness and the measurability  of pullback attractor
        even if the underlying probability space is not complete.
        In  Section 3, we define a non-autonomous set-valued
        random dynamical system for \eqref{intro1}-\eqref{intro2}, and 
        prove the
          weak and strong continuity of solutions.
        Section 4 is devoted to the pullback asymptotic compactness
        of solutions which is proved  
        by the method  of   energy equations.
        In the last section, we present the existence and uniqueness
        of  $\cald$-pullback attractors    for \eqref{intro1}-\eqref{intro2}.

 In the sequel, 
 we write 
the norm and  inner product
of $\ltwo$ as  
$\| \cdot \|$ and $(\cdot, \cdot)$,  respectively. 
 The norm  of    a  Banach space  $X$
 is  denoted by     $\|\cdot\|_{X}$.

 \section{Multivalued non-autonomous random dynamical systems}
\setcounter{equation}{0}

The main purpose of this section is to prove a general measurability
result  for multivalued functions from a metric space to
a separable Banach space.  This result can be used
to establish the measurability of  random attractors for
multivalued random dynamical systems generated by a wide
class of stochastic differential equations without uniqueness,
which will be demonstrated by the non-autonomous stochastic
wave equations in this paper.
We will also review basic concepts  of multivalued non-autonomous
random  dynamical systems including  definition and existence of
random attractors, see e.g. \cite{car6, wangy1}.
 These concepts are natural extensions of
single-valued non-autonomous random dynamical systems
\cite{dua3, wan5, wan9}
and   multivalued autonomous   cocycles
\cite{car1, car4, car7, car8}.
The details on attractors of single-valued autonomous random
dynamical systems can be found in  
  \cite{bat1,
   cra2,  fla1,
  schm1} and the references therein.

In this section,  we  assume $X$  and $Z$
are  metric spaces, and their  Borel  $\sigma$-algebras are 
written  as $\calb(X)$  and $\calb(Z)$, respectively.
 The collection of all subsets of
  $Z$ is denoted by
 $2^Z$.  Recall that a mapping $G: X \to 2^Z$  
 is measurable  
 with respect to $\calb(X)$
 if the inverse image of any open subset of
$Z$ under $G$ is a Borel subset of $X$; that is,
for every open set $O$ of $Z$, 
 the set  $G^{-1}(O) = \{ x\in X: G(x) \bigcap O  \neq \emptyset  \}$
belongs to $\calb(X)$.
 A mapping   $G: X \to 2^Z$  is 
     said to be  upper semicontinuous    at $x\in X$
    if  for every neighborhood $\mathcal{U}$
     of  $G(x)$ there exists $\delta>0$ such that
     $G(x^\prime) \subseteq\mathcal{U}$
     whenever $d(x^\prime,x) <\delta$. 
   An obvious  sufficient criterion for  upper semicontinuity of such multivalued
   functions is  stated  below.
   
   \begin{lem}
   \label{uscc}
   Suppose  $X$ and $Z$ are  metric spaces.
    Let $G: X \to 2^Z$   be   a multivalued function
    and $x_0 \in X$.
    If  
for any $x_n \to x_0$ in $X$  and $z_n \in G(x_n)$, there
exist $z_0\in G(x_0)$ and a subsequence $\{z_{n_k}\}$
of $\{z_n\}$ such that
$z_{n_k}  \to z_0$ in $Z$, then
$G$ is upper semicontinuous at $x_0$.
     \end{lem}
     
     \begin{proof}
     If  $G$ is not upper semicontinuous at $x_0$, then
     there exist a neighborhood   $\mathcal{U}$ of $G(x_0)$,
     a  sequence $x_n \to x_0$ and $z_n \in G(x_n)$  such that
     $z_n \notin {\mathcal{U}}$.
    However,  by assumption, there exist
      $z_0\in G(x_0)$ and a subsequence $\{z_{n_k}\}$
of $\{z_n\}$ such that
$z_{n_k}  \to z_0$.
Since $z_0 \in G(x_0) $
and $\mathcal{U}$ is a neighborhood of $G(x_0)$, 
we infer that
$z_{n_k} \in {\mathcal{U}}$ for large $k$,
which is in contradiction with $z_{n_k} \notin {\mathcal{U}}$
for all $k\in \N$.
     \end{proof}
   
   Based on Lemma \ref{uscc}, 
   we introduce the  following weak upper semicontinuity of
multivalued functions,  which will be needed 
in   this  paper  when proving measurability of  pullback
random attractors.

\begin{defn}\label{wuc}
Let $X$ be a metric space and $Z$ a Banach space.
A  multivalued function   $G: X \to 2^Z$  is 
said to be 
    weakly  upper semicontinuous  at $x_0 \in X$  if 
for any $x_n \to x_0$ in $X$  and $z_n \in G(x_n)$, there
exist $z_0\in G(x_0)$ and a subsequence $\{z_{n_k}\}$
of $\{z_n\}$ such that
$z_{n_k}  \rightharpoonup z_0$ weakly in $Z$.
If $G$ is weakly upper semicontinuous 
at every $x \in X$, then   we say  $G: X \to 2^Z$  is 
    weakly  upper semicontinuous.
\end{defn}

We now prove a  measurability result  for 
weakly  upper semicontinuous   
  multivalued functions.
  
  \begin{thm}
  \label{wusc}
  Let $X$ be a metric space and $Z$ a separable Banach space.
  If $G: X \to 2^Z$ is  weakly  upper semicontinuous, then
  $G$ is  measurable with respect to $\calb(X)$.
  \end{thm}
  
  \begin{proof}
  Given $r>0$ and $z_0 \in Z$,  let
  ${\overline{B}} 
  =\{z\in Z:  \| z-z_0 \|_Z \le r\}$ be the  closed ball
  in $Z$ with radius $r$ and center $z_0$.
  Then we claim that  the inverse image of 
  $ {\overline{B}} $ under $G$,
   $G^{-1} ( {\overline{B}} )$,
    is  a closed subset
  of $X$. Let $x_n \in G^{-1}( {\overline{B}} )$ and $x\in X$
  such that $x_n \to  x$. We will show $x\in G^{-1}( {\overline{B}} )$.
  Since $x_n \in G^{-1}( {\overline{B}} )$, we have
  $G(x_n)\bigcap {\overline{B}} \neq \emptyset$ and
  thus  there exists $z_n \in  G(x_n)\bigcap {\overline{B}}$
  for each $n\in \N$.
  Since $x_n \to x$  and  $G$ is  weakly  upper semicontinuous, 
  we find that there exist  $z\in G(x)$ and
  a subsequence $\{z_{n_k}\}$ of $\{z_n\}$ such that
 $z_{n_k} \rightharpoonup z $ in $Z$, which implies
  $z_{n_k} -z_0  \rightharpoonup z -z_0 $ and thus
   \be\label{wuscp1}
\liminf_{k \to \infty}
\|z_{n_k} -z_0 \|_Z  \ge \| z-z_0 \|_Z.
  \ee
  Since $z_{n_k} \in {\overline{B}} $,
  by \eqref{wuscp1} we get
  $\|z-z_0\|_Z \le  r$ and hence
  $z\in {\overline{B}}$. On the other hand, we know
  $z\in G(x)$ and thus
  $z\in G(x) \bigcap {\overline{B}} $, 
  which means $x\in G^{-1} ({\overline{B}}  )$ as desired.
  So we have proved that the inverse image of every closed ball
  in $Z$ 
  under $G$ is  a closed set of $X$.
  
  Let $B = \{z\in Z: \|z-\widetilde{z}\|_Z <\widetilde{r} \}$  be the open ball in $Z$
  with radius $\widetilde{r}$ and center $\widetilde{z}$.
  Then we have  $B =\bigcup\limits_{m=1}^\infty
  {\overline{B}}_m$ with
  ${\overline{B}}_m =\{ z\in Z: \| z-\widetilde{z}\|_Z \le \widetilde{r} -{\frac 1m}\}$
  and hence
 \be\label{wuscp2}
 G^{-1} (B) = \bigcup\limits_{m=1}^\infty
  G^{-1}({\overline{B}}_m).
  \ee
  As proved in the above, 
  for each $m \in \N$, 
  $G^{-1}({\overline{B}}_m)$ is closed in $X$ and thus
  $G^{-1}({\overline{B}}_m)\in \calb(X)$.
  This along with \eqref{wuscp2} shows that
  \be\label{wuscp3}
  G^{-1} (B) \in \calb(X) \ \mbox{for every open ball} \   B  \mbox{ in  }  Z.
  \ee
  Note that $Z$ is separable and  hence  there exists
  $M=\{z_m: m \in \N$ \} such that $M$ is dense in $Z$.
  Thus,  for every open set $O$ in $Z$,  there exist
  $M_0 = \{z_{m_p}: p \in \N\} \subseteq M$ 
  and a sequence of rational numbers $\{r_{m_p}\}_{p=1}^\infty$
  such that
  $O= \bigcup\limits_{p=1}^\infty B_p$ where
  $B_p = \{z\in Z: \|z-z_{m_p} \|_Z < r_{m_{p}}\}$
  is the open ball with radius $r_{m_p}$ and center
  $z_{m_p}$.
  Therefore, $G^{-1} (O)
  =   \bigcup\limits_{p=1}^\infty
  G^{-1}(B_p)$, which together with
  \eqref{wuscp3}  implies
  $G^{-1} (O) \in \calb (X)$ for every open set $O$ in $Z$.
  \end{proof}

In what follows,  we introduce the concept of
multivalued non-autonomous cocycles (random dynamical systems).
As in the  single-valued case, we need to use two distinct 
parametric spaces, say $\Omega_1$ and $\Omega_2$,
to deal with the non-autonomous perturbations of systems:
$\Omega_1$ is for the non-autonomous deterministic perturbation
and $\Omega_2$  for the non-autonomous stochastic perturbation.
Usually, we may   take $\Omega_1$  either as the collection of all translations
of   deterministic external terms or as the collection of all initial times.
In this paper, we will use the latter and hence take $\Omega_1 = \R$.
However, we emphasize that all results can be carried over
in an obvious way  to the case
when $\Omega_1$ is the collection of translations of external terms.
For stochastic perturbation, we take $\Omega_2$  as
a metric dynamical system 
$(\Omega, \calftwo, P,  \{\thtwot\}_{t \in \R})$
where  $(\Omega, \calftwo, P)$
 is a probability space and  $\theta: \R \times \Omega \to \Omega$
 is a measure-preserving  group of translations on $\Omega$.

 Again, we assume that  $(X, d)$
 is a metric space and the Hausdorff semi-distance
 between subsets $A$ and $B$ is written as  $d(A,B)$.

\begin{defn} \label{cocy}
A  multivalued mapping $\Phi$: $ \R^+ \times \R \times \Omega \times X
\to 2^X$ with nonempty  closed images
  is called a multivalued non-autonomous   cocycle on $X$
over  
$(\Omega, \calftwo, P,  \{\thtwot\}_{t \in \R})$
if   for all
  $\tau\in \R$,
  $\omega \in   \Omega $
  and    $t, s \in \R^+$,  the following conditions (i)-(iii)  are satisfied:
\begin{itemize}
\item [(i)]   $\Phi (\cdot, \tau, \cdot, \cdot): \R ^+ \times \Omega \times X
\to 2^X$ is
 $\calb (\R^+)   \times \calftwo \times \calb (X)$-measurable;
 that is, for every open set $O$ of $X$, the set
 $\{(t, \omega, x) \in \R^+ \times \Omega \times X:
 \Phi(t,\tau, \omega, x) \bigcap O
 \neq \emptyset \}$ belongs to
  $\calb (\R^+)   \times \calftwo \times \calb (X)$.

\item[(ii)]    $\Phi(0, \tau, \omega, \cdot) $ is the identity on $X$;

\item[(iii)]    $\Phi(t+ s, \tau, \omega, \cdot) =
 \Phi(t,  \tau +s,  \theta_{s} \omega, \cdot)
 \circ \Phi(s, \tau, \omega, \cdot)$.
    \end{itemize}
    If   there exists  a
    positive number   $T $ such that
    for every $t  \in \R^+$,
     $\tau \in \R$  and $\omega \in \Omega$,
$$
\Phi(t,  \tau +T, \omega, \cdot)
= \Phi(t, \tau,  \omega, \cdot ),
$$
then $\Phi$ is said to be a 
  periodic  cocycle    with period $T$.
\end{defn}
 
 As mentioned earlier,  it is much more difficult to
 obtain measurability of multivalued functions than
 singled-valued ones.  Theorem \ref{wusc} provides an avenue
 to establish the measurability of multivalued cocycles under certain
 conditions. More precisely, we have
 
 \begin{lem}
 \label{mcocy}
  Let  
  $\Phi$: $ \R^+ \times \R \times \Omega \times X
\to 2^X$ be    a  multivalued   function. Suppose
$\Omega$ is a metric space and $X$ a separable Banach
space. If  for  every $\tau \in \R$,   the mapping
     $\Phi (\cdot, \tau, \cdot, \cdot): \R ^+ \times \Omega \times X
\to 2^X$ is weakly upper semicontinuous in the sense of Definition
\ref{wuc}, then
  $\Phi (\cdot, \tau, \cdot, \cdot): \R ^+ \times \Omega \times X
\to 2^X$  is
$\calb (\R^+)   \times \calb(\Omega) \times \calb (X)$-measurable.
\end{lem}
 
 \begin{proof}
  This  result   follows   from Theorem \ref{wusc} immediately since
   $\R^+   \times\Omega \times X$
  is a metric space  in the  present case.
 \end{proof}

 From now on, we assume  
  $\Phi$: $ \R^+ \times \R \times \Omega \times X
\to 2^X$ is   a multivalued non-autonomous   cocycle on $X$
over  
$(\Omega, \calftwo, P,  \{\thtwot\}_{t \in \R})$.

\begin{defn}
\label{omlit}
Let $B=\{B(\tau, \omega): \tau \in \R, \ \omega  \in \Omega\}$
be a family of nonempty subsets of $X$.
Then  the $\Omega$-limit set of $B$ under $\Phi$, 
  $\Omega(B)$, is given by 
$$
\Omega (B)(\tau, \omega)
= \bigcap_{r \ge 0}
\  \overline{ \bigcup_{t\ge  r} \Phi(t,  \tau -t,
 \theta_{ -t} \omega,
 B( \tau -t , \theta_{-t}\omega  ))}
$$
for every $\tau \in \R$ and
$\omega \in \Omega$.
 \end{defn}

\begin{defn}
\label{asycomp}
 Let $\cald$ be  a  collection of    families of  nonempty
 subsets of $X$.
$\Phi$ is said to be  $\cald$-pullback asymptotically
compact in $X$ if
for all $\tau \in \R$,
$\omega \in \Omega$   and $D\in \cald$,    every sequence
$ x_n  \in \Phi(t_n,    \tau -t_n, \theta_{ -t_n} \omega,
   D(\tau -t_n,
  \theta_{ -t_n} \omega ) )  $
    has a convergent  subsequence  in    $X$
 whenever
  $t_n \to \infty$.
\end{defn}

\begin{defn}
\label{defatt}
 Let $\cald$ be a collection of   families of
 nonempty  subsets of $X$
 and
 $\cala = \{\cala (\tau, \omega): \tau \in \R,
  \omega \in \Omega \} \in \cald $.
    Such   $\cala$
is called a    $\cald$-pullback    attractor  for
  $\Phi$
if the following  conditions (i)-(iii) are  satisfied:
for every $\tau \in \R$  and $\omega \in \Omega$,
\begin{itemize}
\item [(i)]   $\cala(\tau, \cdot): \Omega \to  2^X$ is measurable
with respect to $\calf$  
  and
 $\cala(\tau, \omega)$ is compact.

\item[(ii)]   $\cala$  is invariant:  
$  \Phi(t, \tau, \omega, \cala(\tau, \omega)   )
= \cala ( \tau +t, \theta_{t} \omega
), \ \  \forall \   t \ge 0.
$ 

\item[(iii)]   $\cala  $
attracts  every  member   of   $\cald$:  for every
 $D = \{D(\tau, \omega): \tau \in \R, \omega \in \Omega\}
 \in \cald$,
$$ \lim_{t \to  \infty} d (\Phi(t, \tau -t,
 \theta_{-t}\omega, D(\tau -t,
 \theta_{-t}\omega) ) , \cala (\tau, \omega ))=0.
$$
 \end{itemize}
 If, in addition, there exists $T>0$ such that
 $$
 \cala(\tau +T, \omega) = \cala(\tau,    \omega ),
 \quad \forall \  \tau \in \R, \forall \
  \omega \in \Omega,
 $$
 then  $\cala$ is said to be  periodic with period $T$.
\end{defn}

We remark that the $\cald$-pullback attractor
$\cala$ in Definition \ref{defatt} for a multivalued
cocycle   is required to
be measurable with respect to $\calf$
rather than its completion, which is  the same
as  for single-valued cocycles.
In the literature, the attractors for     multivalued 
stochastic equations 
are often assumed to be measurable with respect
to the completion of $\calf$ (not  $\calf$ itself),  
due to the difficulty  to 
  obtain the $\calf$-measurability.
In this paper, we ill  demonstrate  how to 
employ Theorem \ref{wusc} to establish the measurability
of attractors  with respect to   $\calf$.

 As for single-valued cocycles,  a family  
$K=\{K(\tau, \omega): \tau
\in \R, \ \omega  \in \Omega\} \in \mathcal{D}$ 
   is called a  $\cald$-pullback
 absorbing
set for   $\Phi$   if
for all $\tau \in \R$,
$\omega \in \Omega $
and  for every $D \in \cald$,
 there exists $T= T(D, \tau, \omega)>0$ such
that
$$
\Phi(t,  \tau -t, \theta_{ -t} \omega,
D( \tau -t, \theta_{-t} \omega  ))
 \subseteq  K(\tau, \omega)
\quad \mbox{for all} \ t \ge T.
$$
 
For $\Omega$-limit sets of closed 
$\cald$-pullback absorbing sets
of $\Phi$, we have the following result
which is  an extension of single-valued
non-autonomous cocycles and multivalued
autonomous cocycles.

\begin{lem}
\label{olim}
 Let $\cald$ be  an inclusion-closed
 collection of   families of   nonempty subsets of
$X$,  and $\Phi$  be a multivalued
non-autonomous   cocycle on $X$
over $(\Omega, \calftwo, P, \{\thtwot\}_{t \in \R})$.
Suppose 
$\Phi$ is $\cald$-pullback asymptotically
compact in $X$ and the mapping
$\Phi(t, \tau, \omega, \cdot):
X\to 2^X$ is upper semicontinuous
for each $t\in \R^+$, $\tau \in \R$
and $\omega \in \Omega$. If $K\in \cald$ is a   
closed 
     $\cald$-pullback absorbing set
  of $\Phi$,  then 
  the $\Omega$-limit set $\Omega(K)$
   has the properties:
  for every $\tau \in \R$  and $\omega \in \Omega$,
\begin{itemize}
\item [(i)]   $\Omega(K)\in \cald$
and 
 $\Omega(K)(\tau, \omega)$ is compact.

\item[(ii)]   $\Omega(K)$  is invariant:  
$  \Phi(t, \tau, \omega, \Omega(K)(\tau, \omega)   )
= \Omega(K) ( \tau +t, \theta_{t} \omega
), \ \  \forall \   t \ge 0.
$ 

\item[(iii)]   $\Omega(K) $
attracts  every  member   of   $\cald$:  for every
 $D 
 \in \cald$,
$$ \lim_{t \to  \infty} d (\Phi(t, \tau -t,
 \theta_{-t}\omega, D(\tau -t,
 \theta_{-t}\omega) ) ,  \Omega(K) (\tau, \omega ))=0.
$$
 \end{itemize}
 If, in addition, 
 both $\Phi$  and $K$ are periodic with period $T$, then so is
 $\Omega(K)$; that is, 
 $  \Omega(K)(\tau +T, \omega) =  \Omega(K)(\tau,    \omega )$
  for all $ \tau \in \R$ and $
  \omega \in \Omega$.
\end{lem}

\begin{proof}
The proof is just an obvious  combination
of \cite{car1} for multivalued autonomous cocycles
 and \cite{wan5} for single-valued non-autonomous
 cocycles.  The details will not be repeated here again.
\end{proof}

 Next, we prove the measurability of $\Omega$-limit sets
 of $\cald$-pullback absorbing sets.
 To that end,  we   need the following lemma.

\begin{lem}
\label{fmea1}
If  $E_n : \Omega \to 2^X$
is an $\calf$-measurable
set-valued  mapping for every $n\in \N$, then
 the mapping  $\omega \to
  \bigcup\limits_{n=1}^\infty E_n  (\omega) $
  is  $\calf$-measurable, and so is the mapping
  $\omega \to
  \overline{\bigcup\limits_{n=1}^\infty E_n (\omega)} $.
\end{lem}

\begin{proof}
For every open set $O$ in $X$,  we have
$$
\left \{ \omega\in \Omega: 
  (\overline{\bigcup\limits_{n=1}^\infty E_n (\omega)} 
  \bigcap O) \neq \emptyset \right  \}
  =
  \left \{ \omega\in \Omega: 
  ({\bigcup\limits_{n=1}^\infty E_n (\omega)} 
  \bigcap O) \neq \emptyset  \right \}
  $$
 \be\label{fmea1p1}
  =\bigcup\limits_{n=1}^\infty
  \left \{ \omega\in \Omega: 
  (  E_n (\omega)
  \bigcap O) \neq \emptyset  \right \}
  = \bigcup\limits_{n=1}^\infty
   E_n^{-1} (O).  
\ee
  By assumption,  for all $n\in \N$, 
  $ E_n^{-1} (O)$ is measurable
  for every open set $O$, which along with
  \eqref{fmea1p1} completes   the proof.
\end{proof}

The following lemma can be found in \cite{car1} and
\cite{wan9}.
\begin{lem}
\label{fmea2}
Given $n \in  \N$, let $E_n : \Omega \to 2^X$
be an $\calf$-measurable
set-valued  mapping
with nonempty closed images.
Suppose for each fixed  $\omega \in \Omega$,
every sequence
$\{x_n\}_{n=1}^\infty$ with
$x_n \in E_n(\omega)$ is precompact in $X$.
If,  in addition,  $\{E_n (\omega)\}_{n=1}^\infty$
is a decreasing   sequence,
then the map: $ \omega \in \Omega  \to  $
$\bigcap\limits_{n\in \N } E_n(\omega) $
is $\calf$-measurable
with nonempty closed images.
\end{lem}

The  next   is  a result for the     $\calf$-measurability
of $\Omega$-limit sets of absorbing sets of $\Phi$
(see also \cite{car1}).

\begin{lem}
\label{mome}
 Let $\cald$ be  an inclusion-closed
 collection of   families of   nonempty subsets of
$X$,  and $\Phi$  be a multivalued
non-autonomous   cocycle on $X$
over $(\Omega, \calftwo, P, \{\thtwot\}_{t \in \R})$.
Suppose 
$\Phi$ is $\cald$-pullback asymptotically
compact in $X$ and the mapping
$\Phi(t, \tau, \omega, \cdot):
X\to 2^X$ is upper semicontinuous
for each $t\in \R^+$, $\tau \in \R$
and $\omega \in \Omega$. If $K\in \cald$ is a   
closed 
     $\cald$-pullback absorbing set
  of $\Phi$ and
   the mapping
$\omega \to \Phi(t, \tau, \omega, K(\tau, \omega))$
is measurable with respect to $\calf$ for every $t\in \R^+$
and $\tau \in \R$, 
then  the  $\Omega$-limit set $\Omega(K)$  is 
also  measurable with respect to $\calf$.
\end{lem}

\begin{proof}
By Definition \ref{omlit}
and Lemma \ref{olim}, one can check  that
for every $\tau \in \R$  and $\omega \in \Omega$,
\be
\label{momep1}
    \Omega (K) (\tau, \omega) =
  \bigcap_{n=1}^\infty
\  \overline{ \bigcup_{m =n}^\infty \Phi(m, \tau -m,
\theta_{ -m} \omega, K(\tau -m, \theta_{-m}\omega  ))}  
= \bigcap\limits_{n=1}^\infty E_n(\omega),
\ee
where  
\be\label{momep2}
E_n(\omega) = \overline{ \bigcup_{m=n }^\infty \Phi(m, \tau -m,
\theta_{ -m} \omega, K(\tau -m, \theta_{-m}\omega  ) )}.
\ee 
By the measurability of $\theta_{-m}:
\Omega \to \Omega$ and the
  assumption,  we find that for every $m \in \N$, 
 $\Phi(m, \tau -m,
\theta_{ -m} \omega, K(\tau -m, \theta_{-m}\omega  ) )$
is measurable in $\omega$ with respect to $\calf$,  
which together with
  \eqref{momep2} and
 Lemma \ref{fmea1} 
 implies the $\calf$-measurability of $E_n$.

On the other hand, since  $\Phi$ is 
$\cald$-pullback asymptotically compact, we find that
any sequence    $x_n \in E_n(\omega)$ 
is precompact. It is clear that
$E_n(\omega) $ is nonempty and closed  
and  $\{E_n(\omega)\}_{n=1}^\infty$ is
   decreasing.  Therefore,  by  Lemma \ref{fmea2}, 
  the mapping 
   $\omega \to  \bigcap\limits_{n=1}^\infty  E_n(\omega)$
   is measurable, which along with \eqref{momep1}
completes   the proof.
\end{proof}

By Theorem \ref{wusc} and Lemma \ref{mome}
we obtain the following  measurability  of
$\Omega$-limit sets of $\Phi$  which  is convenient in many applications
where $\Phi$ is generated by solutions of  non-autonomous stochastic
equations.

\begin{lem}
\label{mome2}
Suppose $X$ is a separable Banach space and $\Omega$
a metric space with $\Omega =\bigcup\limits_{m=1}^\infty \Omega_m$
and $\Omega_m \in \calb(\Omega)$.
 Let $\cald$ be  an inclusion-closed
 collection of   families of   nonempty subsets of
$X$  and $\Phi$  be a multivalued
non-autonomous   cocycle on $X$
over $(\Omega, \calb(\Omega), P, \{\thtwot\}_{t \in \R})$.
Assume 
$\Phi$ is $\cald$-pullback asymptotically
compact in $X$ and the mapping
$\Phi(t, \tau, \omega, \cdot):
X\to 2^X$ is upper semicontinuous
for each $t\in \R^+$, $\tau \in \R$
and $\omega \in \Omega$.
  If $K\in \cald$ is a   
closed 
     $\cald$-pullback absorbing set
  of $\Phi$ and
   the mapping 
$  \Phi(t, \tau, \cdot, K(\tau, \cdot))$:
$\Omega_m \to 2^X$
is weakly  upper semicontinuous in the sense of
Definition \ref{wuc}  for every $m \in \N$,
 $t\in \R^+$ and $\tau\in \R$, 
then  the  $\Omega$-limit set $\Omega(K)$  is 
  measurable with respect to $\calf$.
\end{lem}

\begin{proof}
Since for every $m \in \N$,
$t\in \R^+$ and $\tau \in \R$,
the  mapping 
$ \omega \to  \Phi(t, \tau, \omega, K(\tau, \omega))$
is weakly  upper semicontinuous 
from 
 $\Omega_m$ to $ 2^X$,
 by     Theorem \ref{wusc},  we find that
 the mapping
 $  \Phi(t, \tau, \cdot, K(\tau, \cdot))$:
$\Omega_m \to 2^X$
     is measurable with respect to $\calb(\Omega_m)$.
 Since 
$ \Omega= \bigcup\limits_{m=1}^\infty \Omega_m$
and $\Omega_m \in \calb(\Omega)$ we infer  that
  the mapping
   $  \Phi(t, \tau, \cdot, K(\tau, \cdot))$:
$\Omega \to 2^X$ 
is measurable with respect to $\calf=\calb(\Omega)$, which together
with Lemma \ref{mome} yields the desired result.
\end{proof}

As an immediate consequence of Lemmas \ref{olim}
and \ref{mome2} we have the following existence and uniqueness
of non-autonomous random attractors.

\begin{thm}
\label{att}
Suppose $X$ is a separable Banach space and $\Omega$
a metric space with $\Omega =\bigcup\limits_{m=1}^\infty \Omega_m$
and $\Omega_m \in \calb(\Omega)$.
 Let $\cald$ be  an inclusion-closed
 collection of   families of   nonempty subsets of
$X$  and $\Phi$  be a multivalued
non-autonomous   cocycle on $X$
over $(\Omega, \calb(\Omega), P, \{\thtwot\}_{t \in \R})$.
Suppose further:

(i) $\Phi$ is $\cald$-pullback asymptotically
compact in $X$.

 (ii)   
$\Phi(t, \tau, \omega, \cdot):
X\to 2^X$ is upper semicontinuous
for each $t\in \R^+$, $\tau \in \R$
and $\omega \in \Omega$.

 (iii)   $\Phi$  has   
a closed 
     $\cald$-pullback absorbing set
     $K\in \cald$.

  (iv)  
$  \Phi(t, \tau, \cdot, K(\tau, \cdot))$:
$\Omega_m \to 2^X$
is weakly  upper semicontinuous in the sense of
Definition \ref{wuc}  for every $m \in \N$,
$t\in \R^+$ and $\tau\in \R$.

Then   
$\Phi$ has a  unique  $\cald$-pullback
attractor $\cala = \Omega (K)$  in $\cald$.
If, in addition, both $\Phi$  and $K$ are $T$-periodic,
then so is the attractor $\cala$.
 \end{thm}

 \section{Multivalued non-autonomous  cocycles for  wave equations}
\setcounter{equation}{0}

 We will  define a multi-valued 
  non-autonomous 
cocycle  in this section for
problem \eqref{intro1}-\eqref{intro2},
and then investigate the continuity properties
of the solutions.

Let
 $z= u_t + \delta u$ where $\delta$  is a nonnegative number to be
 determined later.
By 
\eqref{intro1}   we  get
 \begin{equation}
 \label{spde1}
 {\frac {du}{dt}} + \delta u =z,
 \end{equation}
 \begin{equation}
 \label{spde2}
 {\frac {dz}{dt}} + (\alpha -\delta) z
 + (\lambda + \delta^2 -\alpha \delta) u
 -\Delta u + f(x, u) =g(t, x) + \varepsilon u\circ
  {\frac {dw}{dt}},
 \end{equation}
   with   initial
 conditions
 \begin{equation}
 \label{spde3}
 u(x, \tau) =u_0(x), \quad z(x, \tau) = z_0(x),
 \end{equation}
  where $z_0(x) = u_1(x)+ \delta u_0(x)$.
  Throughout the rest of the paper, we assume 
  $g\in L^2_{loc}(\R, \ltwo)$, 
  $f: \R^n \times \R \to  \R$ is  a continuous function which
  along with its  antiderivative  
     $F(x,s) = \int_0^s f(x,s) ds$ 
     satisfies:  
  for every
 $x\in \mathbb{R}^n$ and $ s\in \R$,
 \begin{equation}
 \label{f1}
 |f(x,s)|
 \le c_1 |s|^\gamma + \phi_1(x),\quad \phi_1 \in L^2(\mathbb{R}^n),
 \end{equation}
 \begin{equation}
 \label{f2}
 f(x,s) s - c_2 F(x, s) \ge \phi_2(x),
 \quad \phi_2 \in L^1(\mathbb{R}^n),
 \end{equation}
 \begin{equation}
 \label{f3}
 F(x,s) \ge c_3 |s|^{\gamma +1} -\phi_3 (x),
 \quad \phi_3 \in L^1(\mathbb{R}^n),
 \end{equation}
 where  $c_1$, $c_2$  and $c_3$
 are positive constants,   $\gamma \in [1,\infty)$ for $n=1, 2$,
 and  $\gamma \in [1, 3]$ for $n=3$.
 In  the three-dimensional case, $\gamma =3$ is called the
 critical exponent.

 Let  $(\Omega, \mathcal{F}, P, \{\theta_t\}_{t\in \R})$ 
 be  the standard  metric 
dynamical  system where  
$ 
\Omega = \{ \omega   \in C(\R, \R ): \ \omega(0) =  0 \}$
endowed with compact-open topology,
  $\calf$  is 
 the Borel $\sigma$-algebra,
 $P$
is   the Wiener
measure on $(\Omega, \calf)$, and   
$\theta_t: \Omega \to \Omega$ is given by  
 $ \theta_t \omega (\cdot) = 
 \omega (\cdot +t) - \omega (t) $ 
 for all  $  \omega \in \Omega$
 and $  t \in \R$.

 Consider the random variable
 $y: \Omega \to \R$ given by 
\be\label{zomega}
y ( \omega)=   
-\alpha \int^0_{-\infty} e^{\alpha \tau}    \omega  (\tau) d \tau,
\quad  \omega  \in \Omega.
\ee
Then the process   $y(\theta_t \omega )$  
satisfies 
\be
\label{ou1}
dy(\theta_t \omega )  + \alpha  y(\theta_t \omega)  dt = d w.
\ee
It follows from \cite{arn1} that there
  exists a
 $\theta_t$-invariant subset of full measure
 (which is still  denoted by $\Omega$)  such that
  $y(\theta_t\omega)$  is
 continuous in $t$  and  $y(\omega)$ is tempered in $\omega
 \in \Omega$.

Let $v$  be a new variable given by
  $v(t, \tau, \omega) = z(t, \tau, \omega)
 - \varepsilon y(\theta_t \omega) u(t, \tau,
 \omega)$.  By
   \eqref{spde1}-\eqref{spde3}  we get 
 \begin{equation}
 \label{pde1}
 {\frac {du}{dt}} + \delta u -v  = \varepsilon y(\theta_t \omega)
 u,
 \end{equation}
 \begin{equation}
 \label{pde2}
 {\frac {dv}{dt}} + (\alpha -\delta) v
 + (\lambda + \delta^2 -\alpha \delta) u
 -\Delta u + f(x, u) =g  
 - \varepsilon y(\theta_t \omega) v
 -\varepsilon 
 \left ( 
 \varepsilon y(\theta_t \omega) -2\delta 
 \right ) y(\theta_t \omega)u,
 \end{equation}
   with   initial
 conditions
 \begin{equation}
 \label{pde3}
 u( \tau, x) =u_0(x), \quad v( \tau, x) = v_0(x),
 \end{equation}
 where $v_0 = z_0 - \varepsilon y(\theta_\tau \omega)
 u_0$.

Note that problem    \eqref{pde1}-\eqref{pde3}
is a pathwise deterministic  system parametrized 
by $\omega \in \Omega$. Then by the deterministic
approach  of  \cite{bal1},  one can show 
that under  conditions \eqref{f1}-\eqref{f3},  
  for  every
 $\omega \in \Omega$,   $\tau \in \mathbb{R}$ and
 $(u_0, v_0) \in \hone \times \ltwo$,
system  \eqref{pde1}-\eqref{pde3}
 has  at least  one    solution,  in the sense of \cite{bal1}, 
 $(u(\cdot, \tau, \omega, u_0), v(\cdot, \tau, \omega, v_0))
 \in C([\tau, \infty), H^1(\mathbb{R}^n) \times L^2(\mathbb{R}^n))$
 with $(u(\tau, \tau, \omega, u_0), v(\tau, \tau, \omega, v_0))
 =(u_0, v_0)$. 
    This implies that
 $(u(\cdot, \tau, \omega, u_0), z(\cdot, \tau, \omega, z_0))
 \in C([\tau, \infty), H^1(\mathbb{R}^n) \times L^2(\mathbb{R}^n))$
satisfies  \eqref{spde1}-\eqref{spde3}
 with
 \be\label{tranz}
 z(t, \tau, \omega, z_0) = v(t, \tau, \omega, v_0)
 + \varepsilon y(\theta_t \omega)  u(t, \tau, \omega,  u_0).
 \ee
 Note that  any solution $(u,v)$ of \eqref{pde1}-\eqref{pde3}
   satisfies the energy equation:
  $$
   {\frac d{dt}}  \left (
  \| v \|^2 + (\lambda +\delta^2 -\alpha \delta) \| u \|^2
  + \| \nabla u \|^2 + 2 \int_{\mathbb{R}^n} F(x, u)dx
 \right )
 $$
 $$
  +2 (\alpha -\delta) \| v\|^2
  + 2 \delta (\lambda + \delta^2 -\alpha \delta ) \| u \|^2
  +2 \delta \| \nabla u \|^2
  + 2 \delta (f(x,u), u)
 $$
 $$
 =  2(g, v) -2  \varepsilon y(\theta_t \omega) \|v\|^2
 -2\varepsilon (\varepsilon y(\theta_t \omega) -2\delta)
 y(\theta_t \omega) (u,v)
 $$
 \begin{equation}
 \label{ener1}
 +2 \varepsilon  (\lambda +\delta^2 -\alpha \delta)
  y(\theta_t \omega) \| u \|^2
  + 2\varepsilon y(\theta_t \omega) \| \nabla u \|^2
  +2 \varepsilon y(\theta_t \omega) (f(x,u), u).
 \end{equation}
 Formally, equation \eqref{ener1} can be obtained by
 taking the inner product of \eqref{pde2}
 with $v$ in $L^2(\mathbb{R}^n)$, and  then using \eqref{pde1}
 to  substitute for $v$ in 
$(u,v) $, 
$(\Delta  u, v)$ and  
$(f(x,u), v )$ (see, e.g., \cite{wan9}). This process  can be justified
by a limiting approach as in \cite{bal1}.

We will prove the weak and strong upper semicontinuity of the solutions of
\eqref{pde1}-\eqref{pde3}. To this end,  for every positive integer
$m\in \N$, we introduce a subset of $\Omega$:
\be\label{defom}
\Omega_m = \{ \omega \in \Omega: 
|\omega (t)| \le |t|\quad
\mbox{and} \quad |\int_0^t |y(\theta_r \omega)|^2 dr |
\le {\frac 1\alpha} |t|
\quad \mbox{for all} \ |t| \ge m \}.
\ee
These  subsets $\Omega_m$ with $m \in \N$ have the following properties:

\begin{lem}
\label{omegam}
Let  $y$ be the random variable given
by \eqref{zomega}, and 
$\Omega_m$ be the subset of $\Omega$  given 
by \eqref{defom} for $m\in \N$.

(i) If $\omega_n \to \omega$ with $\omega_n, \omega \in \Omega_m$
for a fixed $m\in \N$, then
$y(\theta_t \omega_n) \to y(\theta_t \omega)$ uniformly for
$t$  in any compact interval of $\R$. Particularly,
if $t_n \to t$  and $\omega_n \to \omega$
with $\omega_n, \omega \in \Omega_m$ for a fixed $m$, then
$y(\theta_{t_n} \omega_n) \to y(\theta_t \omega)$.

(ii) For every $m \in \N$, $\Omega_m$ is a closed subset of $\Omega$
and
$\Omega = \bigcup\limits_{m=1}^\infty \Omega_m$.

(iii)  Given $m \in \N$,  we have,  for all $ t \le - m$  and $\omega\in \Omega_m$,
\be\label{omegam1}
|y(\theta_t \omega)| \le 2 |t| +\alpha \int_{-\infty}^0
e^{\alpha \tau} |\tau| d\tau.
\ee
 \end{lem}
 
 \begin{proof}
 
 (i).
 Let $[a,b]$ be a compact interval of $\R$ and $\omega_n \to \omega$
 with $\omega_n, \omega \in \Omega_m$.
 By \eqref{zomega} we get
 \be\label{omegam_a1}
 y(\theta_t \omega) = -\alpha \int_{-\infty}^0 e^{\alpha \tau}
 \omega(t+\tau) d\tau + \omega (t)
 =
 -\alpha e^{-\alpha t} \int_{-\infty}^t e^{\alpha \tau}
 \omega(\tau) d\tau + \omega (t).
\ee
 Therefore we have
 \be\label{omegam_p1}
 |y(\theta_t \omega_n) - y(\theta_t \omega)|
\le
  \alpha e^{-\alpha t}  \left | \int_{-\infty}^t e^{\alpha \tau}
( \omega_n(\tau) -\omega (\tau)  d\tau \right |
+ | \omega_n (t)  - \omega (t) |.
\ee
 Since $\int_{-\infty}^0 e^{\alpha \tau} |\tau| d\tau <\infty$,
 given $\varepsilon>0$, there exists $T_1 = T_1(\varepsilon)>0$ such that
 \be\label{omegam_p2}
 \int_{-\infty}^{-T_1} e^{\alpha \tau} |\tau|d\tau <\varepsilon.
 \ee
 Let $T=\max \{ T_1, m, -a\}$. By \eqref{omegam_p1}-\eqref{omegam_p2}
 and the fact  that   $\omega_n, \omega \in \Omega_m$, we obtain, for all
 $t\in [a,b]$,
 $$
   |y(\theta_t \omega_n) - y(\theta_t \omega)|
\le
  \alpha e^{-\alpha t}   \int_{-\infty}^t e^{\alpha \tau}
|\omega_n(\tau) -\omega (\tau) |  d\tau 
+ | \omega_n (t)  - \omega (t) |
$$
 $$
 \le
  \alpha e^{-\alpha t}   \int_{-\infty}^{-T} e^{\alpha \tau}
|\omega_n(\tau) -\omega (\tau) |  d\tau 
+
 \alpha e^{-\alpha t}   \int_{-T}^{t} e^{\alpha \tau}
|\omega_n(\tau) -\omega (\tau) |  d\tau 
+ | \omega_n (t)  - \omega (t) |
$$
  $$
 \le
 2 \alpha e^{-\alpha t}   \int_{-\infty}^{-T} e^{\alpha \tau}
| \tau |  d\tau 
+
 \alpha e^{-\alpha t}   \int_{-T}^{b} e^{\alpha \tau}
|\omega_n(\tau) -\omega (\tau) |  d\tau 
+ | \omega_n (t)  - \omega (t) |
$$
 \be\label{omegam_p5}
 \le
 2 \alpha e^{-\alpha a} \varepsilon   
+
 \alpha e^{-\alpha a}   \int_{-T}^{b} e^{\alpha \tau}
|\omega_n(\tau) -\omega (\tau) |  d\tau 
+ | \omega_n (t)  - \omega (t) |.
\ee
Since $\omega_n \to \omega$ with respect to the
compact-open topology of $\Omega$, 
there exists $N=N(\varepsilon)\ge 1$ such that
for all $n \ge N$ and $t \in [a,b]$,
 $$
 \alpha e^{-\alpha a}   \int_{-T}^{b} e^{\alpha \tau}
|\omega_n(\tau) -\omega (\tau) |  d\tau 
\le \varepsilon
\quad \mbox{and} \quad
  | \omega_n (t)  - \omega (t) | \le \varepsilon,
  $$
  which  together with
  \eqref{omegam_p5}  shows that 
   for all $n \ge N$ and $t \in [a,b]$,
   $$
    |y(\theta_t \omega_n) - y(\theta_t \omega)|
    \le
    2 \alpha e^{-\alpha a} \varepsilon  
    +2 \varepsilon.
    $$
    In other words,  we get
    \be
    \label{omegam_p7}
    y(\theta_t \omega_n) \to  y(\theta_t \omega) 
    \ \mbox{uniformly in} \   t \in [a,b] .
    \ee

    Now suppose $t_n \to t$   and $\omega_n \to \omega$
    with $\omega_n, \omega \in \Omega_m$. Then we have
    $$
     |y(\theta_{t_n} \omega_n) - y(\theta_t \omega)|
     \le
      |y(\theta_{t_n} \omega_n) - y(\theta_{t_n} \omega)|
      +
       |y(\theta_{t_n} \omega) - y(\theta_t \omega)|
       $$
       from which, \eqref{omegam_p7} and the continuity of
       $y(\theta_t \omega)$ in $t$, we obtain
      $$ y(\theta_{t_n} \omega_n) \to  y(\theta_t \omega),
      \quad \mbox{as} \  n \to \infty.
      $$
      
      (ii). Let $\omega_n \to \omega$ with $\omega_n \in \Omega_m$
      and $\omega \in \Omega$. We will show $\omega \in \Omega_m$.
      Since $\omega_n \in \Omega_m$, we have for all
      $|t| \ge m$ and $n \in \N$,
      $$
     | \omega_n (t) | \le |t|
     \quad \mbox{and} 
     \quad 
     |\int_0^t |y(\theta_{r} \omega_n)|^2 dr   | 
     \le {\frac 1\alpha} |t|.
     $$
     Taking the limit as $n \to \infty$,  by (i)   we obtain for all $|t| \ge m$,
       $$
     | \omega (t) | \le |t|
     \quad \mbox{and} 
     \quad 
     |\int_0^t |y(\theta_{r} \omega)|^2 dr   | 
     \le {\frac 1\alpha} |t|.
     $$
     This indicates $\omega \in \Omega_m$ and hence
     $\Omega_m$ is a closed subset of $\Omega$   for all
     $m\in \N$.
     
     Given $\omega \in \Omega$, 
     since $|{\frac {\omega(t)}t}| \to 0$ as $ |t| \to \infty$, 
      there exists $T_1>0$ such that
     for all $ |t| \ge T_1$,
     \be\label{omegam_p20}
     |\omega (t) | \le |t|.
     \ee
     On the other hand, by the ergodicity theorem,
     $$
     \lim_{|t| \to \infty}
     {\frac 1t} \int_0^t |y(\theta_r \omega)|^2 dr
     = {\frac 1{2\alpha}},
     $$
     which shows that   there exists $T_2 \ge T_1$  such that
     for all $|t| \ge T_2$,
     \be\label{omegam_p22}
     | \int_0^t |y(\theta_r \omega)|^2 dr |
     \le  {\frac 1{\alpha}} |t|.
     \ee
     By \eqref{omegam_p20}-\eqref{omegam_p22} we find
     $\omega \in \Omega_m$ for $m >T_2$ and thus
     $\Omega \subseteq \bigcup\limits_{m=1}^\infty \Omega_m$.
     It is evident that $ \bigcup\limits_{m=1}^\infty \Omega_m \subseteq
     \Omega$, and hence  $\Omega =\bigcup\limits_{m=1}^\infty \Omega_m$.
     
     (iii).  By \eqref{omegam_a1} we have, for all
     $\omega \in \Omega_m$, $t \le -m$  and $\tau \le 0$,
    $$ |y(\theta_t \omega)|
     \le
     \alpha \int_{-\infty}^0 e^{\alpha \tau} |\omega (t+\tau)| d\tau
     +
     |\omega (t) |
     $$
     $$
      \le
     \alpha \int_{-\infty}^0 e^{\alpha \tau} (|t| + |\tau|) d\tau
     +
     | t |
      \le 2 | t | +
     \alpha \int_{-\infty}^0 e^{\alpha \tau}  |\tau| d\tau
     $$
     which yields \eqref{omegam1}.
     \end{proof}
   
   We mention that a  similar subset of $\Omega$ was 
   introduced in \cite{car1} to deal with the measurability
   of solutions of  stochastic equations  without uniqueness, 
   which  works   for  many  equations  perturbed by
  {\it  additive} noise, but does not work for equations with
  {\it multiplicative}  noise like \eqref{intro1}.
  That is why,  in the present  paper, we consider the subset
  $\Omega_m$ given by \eqref{defom} and show this  subset
  can be used to  establish the measurability of the solutions
  of \eqref{intro1}.
  
  From now on,  we write
  $\calf_{\Omega_m}$ for    the
  trace $\sigma$-algebra of $\calf$ with respect  to
  $\Omega_m$, and $P_{\Omega_m}$ for the 
  restriction of $P$ to $\calf_{\Omega_m}$.
  Sometimes, it is convenient to consider the measurable
  space $(\Omega_m, \calf_{\Omega_m}, P_{\Omega_m})$
  instead of the original space $(\Omega, \calf, P)$.
  Since $\Omega_m$ is  a closed subset of $\Omega$ by Lemma \ref{omegam},
  we find that  $\calf_{\Omega_m} \subset \calf$.

   It is convenient to  reformulate system \eqref{pde1}-\eqref{pde2}as
\be
\label{abspde}
{\frac {d \xi}{dt}} =A \xi + G(x, t, \omega, \xi),
\ee
where
$$
\xi = \left (
\begin{array}{l}
u\\
v
\end{array}
\right ), \quad 
A= \left (\begin{array}{cc}
0 \ & 1\\
\Delta \ & \delta -\alpha\\
\end{array}
\right ),
$$
and
\be\label{G1}
G(x,t, \omega, \xi) =
\left (
\begin{array}{c}
- \delta u +  \varepsilon y(\theta_t \omega) u\\
g  -(\lambda + \delta^2 -\alpha \delta) u -f(x, u)
 - \varepsilon y(\theta_t \omega) v
 -\varepsilon 
 \left ( 
 \varepsilon y(\theta_t \omega) -2\delta 
 \right ) y(\theta_t \omega)u
\end{array}
\right ).
\ee
As proved in \cite{bal1},  if $ \xi \in  C([\tau, \infty) ,  H^1(\R^n) \times L^2(\R^n))$
is a weak solution of \eqref{pde1}-\eqref{pde3}, then $\xi$ satisfies
\be\label{absol}
\xi (t, \tau, \omega, \xi_0)
=e^{A(t-\tau)}\xi_0
+\int_\tau^t e^{A(t-s)} G(\cdot, s,  \omega, \xi(s, \tau, \omega, \xi_0)) ds.
\ee

The mapping $G$ given by\eqref{G1} has the following property
which proves useful later.

\begin{lem}
\label{Gpro}
Suppose \eqref{f1}-\eqref{f3} hold. Then we have

(i) if $\xi_n \rightharpoonup \xi$ in $\hone \times \ltwo$
and $\omega_n \to \omega$ with $\omega_n, \omega \in \Omega_m$,
then
$G(\cdot, t, \omega_n, \xi_n) \rightharpoonup
G(\cdot, t, \omega, \xi)$
 in $\hone \times \ltwo$
 uniformly for  $ t$  in compact intervals of  $R$.

(ii) if $\xi_n \to  \xi$ strongly  in $\hone \times \ltwo$
and $\omega_n \to \omega$ with $\omega_n, \omega \in \Omega_m$,
then
$ 
G(\cdot, t, \omega_n, \xi_n) \to
G(\cdot, t, \omega, \xi)$
 strongly in  $  \hone \times \ltwo$
 uniformly for  $ t $  in compact intervals of $\R$.
\end{lem}

\begin{proof}
(i). Suppose  $\xi_n = \left (
\begin{array}{c}
u_n\\
v_n
\end{array}
\right )$, 
$\xi = \left (
\begin{array}{c}
u\\
v
\end{array}
\right )$
and $\xi_n \rightharpoonup \xi$
in $\hone \times \ltwo$.
By \eqref{f1}  we find that
$f(\cdot, u_n)$ is bounded in
$\ltwo$ and hence there is
$\phi \in \ltwo$ such that, up to a subsequence,
\be\label{Gpro_p1}
f(\cdot, u_n)\rightharpoonup \phi
\  \mbox{in} \ \ltwo.
\ee
Let $B_k = \{ x \in \R^n: \| x\| < k\}$
for each positive integer $k$.
By the compactness of embedding $H^1(B_k) 
\hookrightarrow L^2(B_k)$,  for every fixed $k\in \N$,
there exists a subsequence (depending on $k$)
of $\{u_n\}_{n=1}^\infty$ 
that is convergent in $L^2(B_k)$ and almost everywhere  on $B_k$.
Then  by a diagonal process, we can extract a further subsequence
(which is still denoted by  $\{u_n\}_{n=1}^\infty$)
 such that 
\be\label{Gpro_p2}
u_n \to u \ \mbox{ a.e.  on  } \R^n,
\ee
which implies that
\be\label{Gpro_p3}
f(x, u_n (x)) \to  f(u(x))
 \ \mbox{ for  a.e.    }  x\in \R^n.
 \ee
 By \eqref{Gpro_p1},  \eqref{Gpro_p3} and the Mazur\rq{}s theorem,
 we get $\phi = f(u) $ and thus
 the entire sequence 
 $f(\cdot, u_n)\rightharpoonup \phi $ in $\ltwo$.
 This together with \eqref{G1}  and  Lemma \ref{omegam} (i) implies
 $G(\cdot, t, \omega_n, \xi_n)
 \rightharpoonup G(\cdot, t, \omega , \xi )$
 in $\hone \times \ltwo$
 uniformly for $t$ in a compact interval.

 (ii). By \eqref{f1} we have
 \be\label{Gpro_p4}
 f^2(x, u_n(x)) \le g_n(x)
 \quad  \mbox{for all } x\in \R^n,
 \ee
 where  $g_n(x) =2c_1^2 |u_n(x)|^{2\gamma}
 +\phi_1^2(x)$. If $u_n \to u$ in $\hone$, then
 $g_n \to g$ in $L^1(\R^n)$, which along with 
 \eqref{Gpro_p2}, \eqref{Gpro_p4} and the dominated
 convergence theorem yields
 \be\label{Gpro_p7}
 \ii f^2 (x, u_n(x)) dx \to \ii f ^2(x, u(x)) dx.
 \ee
 By \eqref{Gpro_p7} and the  weak convergence 
 of $f(\cdot, u_n)$ we get
 $f(\cdot, u_n) \to  f (\cdot, u)$ in $\ltwo$.
 Then using 
 \eqref{G1}  and  Lemma \ref{omegam} (i) we infer that
 $G(\cdot, t, \omega_n, \xi_n)
 \to  G(\cdot, t, \omega , \xi )$ in $\hone \times \ltwo$
 uniformly for $t$ in a compact interval.
\end{proof}

   The following lemma provides  estimates of solutions
   of \eqref{pde1}-\eqref{pde3}  which are uniform
   in time  and initial data on a bounded set.
     
    \begin{lem}
    \label{estf}
    Suppose \eqref{f1}-\eqref{f3} hold. Let
    $\varepsilon \ge 0$,  $\tau \in \R$, $T>0$, $M>0$ and
    $\omega_n \to \omega$ with
    $\omega_n, \omega \in \Omega_m$.
    Then there exists $C= C(\varepsilon, \tau,  T, M,  m, \omega)>0$
    such that the solutions of \eqref{pde1}-\eqref{pde3} satisfy
    $$
    \| u(t, \tau,  \omega_n, u_0)\|_{H^1}
    +
      \| v(t, \tau,  \omega_n, v_0)\|_{L^2}
      \le C
      $$
      for all  $n\in \N$, $t\in [\tau,  \tau+T]$  and $(u_0, v_0)\in \hone \times \ltwo$
      with $\|(u_0, v_0)\|_{H^1\times L^2} \le M$.
    \end{lem}

           \begin{proof}
           By \eqref{ener1} with $\delta =0$ we get 
 $$
   {\frac d{dt}}  \left (
  \| v \|^2 + \lambda  \| u \|^2
  + \| \nabla u \|^2 + 2 \int_{\mathbb{R}^n} F(x, u)dx
 \right )
 $$
 $$
  \le   2(g, v) -2  \varepsilon y(\theta_t \omega) \|v\|^2
 -2\varepsilon^2  
 |y(\theta_t \omega)|^2 (u,v)
 $$
 \begin{equation}
 \label{estf_p1}
 +2 \varepsilon  \lambda
  y(\theta_t \omega) \| u \|^2
  + 2\varepsilon y(\theta_t \omega) \| \nabla u \|^2
  +2 \varepsilon y(\theta_t \omega) (f(x,u), u).
 \end{equation}
 By  \eqref{f1} and \eqref{f3}, we  have
 $$
 2 \varepsilon y(\theta_t \omega) (f(x,u), u)
 \le 2 \varepsilon c_1   | y(\theta_t \omega)|
 \int_{\R^n} |u|^{\gamma+1} dx
 + \varepsilon | y(\theta_t \omega)|
 \|\phi_1 \|^2 
 +  \varepsilon | y(\theta_t \omega)| \| u \|^2
 $$
\begin{equation}
\label{estf_p2}
 \le  \varepsilon c   | y(\theta_t \omega)|
 \int_{\R^n}   F(x,u)  dx
 + \varepsilon c | y(\theta_t \omega)|
 +  \varepsilon | y(\theta_t \omega)| \| u \|^2.
\end{equation}
It follows   from \eqref{estf_p1}-\eqref{estf_p2} that
 $$
   {\frac d{dt}}  \left (
  \| v \|^2 + \lambda  \| u \|^2
  + \| \nabla u \|^2 + 2 \int_{\mathbb{R}^n} F(x, u)dx
 \right )
 $$
 $$
 \le
 \| g\|^2 + (1+ 2 \varepsilon | y(\theta_t \omega)| 
 + \varepsilon^2 | y(\theta_t \omega)| ^2 ) \| v\|^2
 +   \varepsilon | y(\theta_t \omega)| ( 1+ 2\lambda
 + \varepsilon | y(\theta_t \omega)|  ) \| u\|^2
 $$
 $$
 + 2 \varepsilon | y(\theta_t \omega)|  \|\nabla u \|^2
 +\varepsilon c | y(\theta_t \omega)| \ii F(x, u) dx
 +\varepsilon c | y(\theta_t \omega)| .
 $$
 Integrating the above over $(\tau , t)$ 
 and replacing 
  $\omega$ by $ \omega_n$ we obtain
  $$
  \| v(t, \tau,  \omega_n, v_0) \|^2 
  + \lambda  \| u (t, \tau,  \omega_n, u_0)\|^2
  $$
  $$
  + \| \nabla u(t, \tau,  \omega_n, u_0) \|^2 
  + 2 \ii
  F(x, u(t, \tau,  \omega_n, u_0))dx
  $$
   $$
   \le 
   \| v_0 \|^2 
  + \lambda  \| u _0\|^2
  + \| \nabla u _0 \|^2 
  + 2  \ii  F(x, u_0)dx 
  $$
  $$
 +
  \int_{\tau }^t \| g(s)\|^2 ds
   +\int_{\tau }^t  (1+ 2 \varepsilon | y(\theta_{s} \omega_n)| 
 + \varepsilon^2 | y(\theta_{s} \omega_n)| ^2 )
  \| v(s, \tau, \omega_n, v_0)\|^2ds
 $$
 $$
 + \varepsilon \int_{\tau }^t   | y(\theta_{s} \omega_n)| ( 1+ 2\lambda
 + \varepsilon | y(\theta_{s} \omega_n)|  )
  \| u (s, \tau,  \omega_n, u_0) \|^2 ds
 $$
 $$
 + 2 \varepsilon \int_{\tau }^t | y(\theta_{s} \omega_n)| 
  \|\nabla u(s, \tau, \omega_n, u_0) \|^2 ds
  $$
  \be\label{estf_p4}
 +\varepsilon c  \int_{\tau }^t \ii  | y(\theta_{s} \omega_n)|  F(x, u) dxds
 +\varepsilon c \int_{\tau }^t  | y(\theta_{s} \omega_n)| ds .
 \ee
 Since $\omega_n \to \omega$ and $\omega_n, \omega\in \Omega_m$, 
 by Lemma \ref{omegam} we know that
 $y(\theta_r \omega_n) \to y(\theta_r \omega)$ uniformly
 for $r\in [\tau, \tau+T]$ as $n \to \infty$. 
 Therefore, there exists $N=N(T, \tau,  \omega)\ge 1$
 such that for all $n\ge N$ and $r\in [\tau, \tau+T]$,
 $$
 |y(\theta_r \omega_n)| \le 1+  |y(\theta_r \omega)|,
 $$
which together with the continuity of 
$y(\theta_r \omega)$ in $r$ implies that   there exists $C_1=C_1(T, \tau, \omega)>0$
such that for all  $n\ge N$ and $r\in  [\tau, \tau+T]$,
\be\label{estf_p6}
 |y(\theta_r \omega)| \le C_1
 \quad  \mbox{and} \quad |y(\theta_r \omega_n)| \le 1+ C_1 .
 \ee
   It follows from \eqref{estf_p4}-\eqref{estf_p6}  that
   there exists      $C_2=C_2(T, \tau,  \omega)>0$
such that for all  $n\ge N$ and
 $t \in  [\tau, \tau+T]$,
   $$
  \| v(t, \tau,  \omega_n, v_0) \|^2 
  + \lambda  \| u (t, \tau,  \omega_n, u_0)\|^2
  $$
  $$
  + \| \nabla u(t, \tau,  \omega_n, u_0) \|^2 
  + 2\ii
  F(x, u(t, \tau,   \omega_n, u_0))dx
  $$
   $$
   \le 
   \| v_0 \|^2 
  + \lambda  \| u _0\|^2
  + \| \nabla u _0 \|^2 
  + 2  \ii  F(x, u_0)dx 
  $$
  $$
 +
  \int_{\tau }^t \| g(s)\|^2 ds
   +(1 +\varepsilon + \varepsilon^2)C_2
   \int_{\tau }^t  
  \| v(s, \tau, \omega_n, v_0)\|^2ds
 $$
 $$
 + \varepsilon( 1+ \varepsilon)C_2 
 \int_{\tau }^t   
  \| u (s, \tau, \omega_n, u_0) \|^2 ds
 +   \varepsilon C_2  \int_{\tau }^t
  \|\nabla u(s, \tau,  \omega_n, u_0) \|^2 ds
  $$
  \be\label{estf_p10}
 +\varepsilon c  \int_{\tau }^t \ii  | y(\theta_{s} \omega_n)|  F(x, u) dxds
 +\varepsilon C_2 .
 \ee
 By \eqref{f3} we have $F(x,u) +\phi_3(x) \ge 0$
 for all $x\in \R^n$, from which and \eqref{estf_p6} we find that
 there exists 
   $C_3=C_3(T, \tau, \omega)>0$
such that for all  $n\ge N$ and   $t \in  [\tau, \tau+T]$,
 \be\label{estf_p11}
 \varepsilon c \int_{\tau }^t \ii  | y(\theta_{s} \omega_n)|  F(x, u) dxds
  \le \varepsilon C_3\int_{\tau }^t \ii F(x, u(s)) dxds + \varepsilon C_3.
  \ee
  Let $C_4 =\max \{ (1+\varepsilon +\varepsilon^2)C_2,
  \varepsilon(1+\varepsilon)C_2\lambda^{-1},
  \varepsilon C_2,  {\frac 12} \varepsilon C_3 \}$.
  By  \eqref{estf_p10} and \eqref{estf_p11}  we get, 
             for all  $n\ge N$
             and   $t \in  [\tau, \tau+T]$,
   $$
  \| v(t, \tau,  \omega_n, v_0) \|^2 
  + \lambda  \| u (t, \tau,  \omega_n, u_0)\|^2
  $$
  $$
  + \| \nabla u(t, \tau,   \omega_n, u_0) \|^2 
  + 2\ii
  F(x, u(t, \tau,  \omega_n, u_0))dx
  $$
   $$
   \le 
   \| v_0 \|^2 
  + \lambda  \| u _0\|^2
  + \| \nabla u _0 \|^2 
  + 2  \ii  F(x, u_0)dx 
  $$
  $$
 +
  \int_{\tau }^t \| g(s)\|^2 ds
   +C_4
   \int_{\tau }^t  
  \| v(s, \tau,  \omega_n, v_0)\|^2ds
 $$
 $$
 + C_4 \lambda
 \int_{\tau }^t  
  \| u (s, \tau, \omega_n, u_0) \|^2 ds
 +    C_4  \int_{\tau }^t
  \|\nabla u(s, \tau,  \omega_n, u_0) \|^2 ds
  $$
  \be\label{estf_p12}
 +\varepsilon C_3  \int_{\tau }^t \ii    F(x, u(s)) dx ds
 +\varepsilon (C_2+C_3) .
 \ee
 Since $C_4 \ge {\frac 12} \varepsilon C_3$, by  \eqref{f3} we see that
 there exists $C_5>0$ such that
 \be\label{estf_p15}
 \varepsilon C_3  \int_{\tau }^t \ii    F(x, u(s) ) dx ds
 \le 2C_4 \int_{\tau }^t \ii    F(x, u(s)) dxds  +C_5 (1 +\varepsilon).
 \ee
 It follows from \eqref{estf_p12}-\eqref{estf_p15}
 that
   for all  $n\ge N$ and  $t \in  [\tau, \tau+T]$,
   $$
  \| v(t, \tau,  \omega_n, v_0) \|^2 
  + \lambda  \| u (t, \tau,  \omega_n, u_0)\|^2
  $$
  $$
  + \| \nabla u(t, \tau,   \omega_n, u_0) \|^2 
  + 2\ii
  F(x, u(t, \tau,   \omega_n, u_0))dx
  $$
   $$
   \le 
   \| v_0 \|^2 
  + \lambda  \| u _0\|^2
  + \| \nabla u _0 \|^2 
  + 2  \ii  F(x, u_0)dx 
   +C_4
   \int_{\tau }^t
  \| v(s, \tau,  \omega_n, v_0)\|^2ds
 $$
 $$
 + C_4\lambda
 \int_{\tau }^t  
  \| u (s, \tau,  \omega_n, u_0) \|^2 ds
 +    C_4  \int_{\tau }^t
  \|\nabla u(s, \tau,  \omega_n, u_0) \|^2 ds
  $$
  \be\label{estf_p20}
 + 2 C_4 \int_{\tau }^t  \ii    F(x, u(s)) dx ds
 + (1+ \varepsilon)C_6    +
  \int_{\tau }^t \| g(s)\|^2 ds.
 \ee
 By Gronwall\rq{}s inequality, we get from \eqref{estf_p20}
 that   for all  $n\ge N$  and  $t \in  [\tau, \tau+T]$,
  $$
  \| v(t, \tau,   \omega_n, v_0) \|^2 
  + \lambda  \| u (t, \tau,  \omega_n, u_0)\|^2
  $$
  $$
  + \| \nabla u(t, \tau,  \omega_n, u_0) \|^2 
  + 2\ii
  F(x, u(t, \tau,  \omega_n, u_0))dx
  $$
   $$
   \le  e^{C_4 T}
  \left ( \| v_0 \|^2 
  + \lambda  \| u _0\|^2
  + \| \nabla u _0 \|^2 
  + 2  \ii  F(x, u_0)dx 
  +(1+\varepsilon) C_6
   \right )
  +\int_{\tau }^t
  e^{-C_4(s-t)} g(s) ds,
  $$
  which along with \eqref{f1} and \eqref{f3}
  yields the desired estimates.
  \end{proof}

Next,  we prove the solutions of \eqref{pde1}-\eqref{pde3}
are uniformly small outside a large bounded domain
under  certain conditions.

\begin{lem}
\label{tai1}

 Suppose \eqref{f1}-\eqref{f3} hold,  $\varepsilon\ge 0$,   $\tau \in \R$,
 $T>0$,   and   $\omega_n\to \omega$
    with $\omega_n, \omega \in \Omega_m$.
    Let  $  (u(\cdot,\tau, \omega_n, u_{0,n}),
    v(\cdot,\tau, \omega_n, v_{0,n}) )$
     be a solution of \eqref{pde1}-\eqref{pde3}
    with initial data   
    $(u_{0,n}, v_{0,n})$ at initial time $\tau$.
    If  $(u_{0,n}, v_{0,n}) \to (u_0,v_0)$
    in $\hone \times \ltwo$, then for every
       $\eta>0$,  
     there exist    
       $N=N(\varepsilon,  \tau, T, m, \omega, \eta)  \ge 1$ 
       and $K=K(\varepsilon, \tau, T,  m,  \omega, \eta) \ge 1$
       such that  for all 
 $ n \ge  N$   and $t\in [\tau, \tau +T]$,
\be\label{tai1_1}
\int_{|x| \ge K} 
\left (
| u(t, \tau ,  \omega_n, u_{0,n}  ) |^2
+| \nabla u(t, \tau ,   \omega_n, u_{0,n}  ) |^2
+
| v(t, \tau ,   \omega_n, v_{0,n}  ) |^2 
\right ) dx \le \eta.
\ee
 \end{lem}

 \begin{proof}
 We here  formally derive  the estimate of \eqref{tai1_1} 
 which  
   can be justified  by a  limiting method  as in \cite{bal1}.
   Let  $\rho: \R \to \R$ be a smooth function such that
  $0 \leq \rho \leq 1$ and
\be\label{rho}
\rho (s)  = 
    0 \ \mbox{    for  }     |s| < 1   
    \quad \mbox{and} \quad 
  \rho (s)  =  1 \ \mbox{   for  }   |s| > 2.
  \ee

 Setting $\delta =0$ and taking
  the inner product of \eqref{pde2}
 with $\rho\left ({\frac {|x|^2}{k^2}}  \right ) v$ in $L^2(\mathbb{R}^n)$, we get
 $$
  {\frac 12} {\frac d{dt}}
 \int_{\mathbb{R}^n}
  \rho \left ({\frac {|x|^2}{k^2}}  \right )|v|^2 dx
  + \alpha \int_{\mathbb{R}^n}
  \rho\left ({\frac {|x|^2}{k^2}}  \right ) |v|^2 dx
  $$
  $$
  + \lambda \int_{\mathbb{R}^n}
  \rho\left ({\frac {|x|^2}{k^2}}  \right ) uv dx
  -\int_{\mathbb{R}^n}
  \rho\left ({\frac {|x|^2}{k^2}}  \right )v
   \Delta u  dx
   $$
  $$
  + \int_{\mathbb{R}^n}
  \rho\left ({\frac {|x|^2}{k^2}}  \right )  f(x,u) v dx
  = \int_{\mathbb{R}^n}
  \rho\left ({\frac {|x|^2}{k^2}}  \right )
   g v dx 
   $$
 $$
   - \varepsilon y(\theta_t \omega)  \int_{\mathbb{R}^n}
  \rho \left ({\frac {|x|^2}{k^2}}  \right )|v|^2 dx
  -\varepsilon^2   y^2 (\theta_t \omega)  
    \int_{\mathbb{R}^n}
  \rho\left ({\frac {|x|^2}{k^2}}  \right ) uv dx,
  $$
  which along with \eqref{pde1} implies
  $$
{\frac d{dt}}   \int_{\mathbb{R}^n}
  \rho\left ({\frac {|x|^2}{k^2}}  \right )
  \left (|v|^2 +
  \lambda   | u  |^2
  +  | \nabla u  |^2 + 2   F(x, u)
 \right ) dx
  +2 \alpha\int_{\mathbb{R}^n}
  \rho\left ({\frac {|x|^2}{k^2}}  \right )
     | v |^2 dx
 $$
 $$
= 2  \int_{\mathbb{R}^n}
  \rho\left ({\frac {|x|^2}{k^2}}  \right )
   g v dx 
   -  2\varepsilon y(\theta_t \omega)  \int_{\mathbb{R}^n}
  \rho \left ({\frac {|x|^2}{k^2}}  \right )|v|^2 dx
  $$
  $$
  - 2\varepsilon^2  y^2 (\theta_t \omega)   \int_{\mathbb{R}^n}
  \rho\left ({\frac {|x|^2}{k^2}}  \right ) uv dx
  +2 \varepsilon  \lambda 
  y(\theta_t \omega) 
   \int_{\mathbb{R}^n}
  \rho \left ({\frac {|x|^2}{k^2}}  \right )|u|^2 dx
  $$
  $$
  + 2\varepsilon y(\theta_t \omega)
    \int_{\mathbb{R}^n}
  \rho \left ({\frac {|x|^2}{k^2}}  \right )|\nabla u|^2 dx 
 +2 \varepsilon y(\theta_t \omega) 
  \int_{\mathbb{R}^n}
  \rho \left ({\frac {|x|^2}{k^2}}  \right )
 f(x,u)  u dx
 $$
  \begin{equation}
 \label{tai1_p1}
  - 4  \int_{k\le |x| \le \sqrt{2} k}
  \rho^\prime\left ({\frac {|x|^2}{k^2}}  \right )
  v \nabla u {\frac x{k^2}} dx.
  \ee
   By \eqref{f1} and \eqref{f3} we  have
 $$
 2 \varepsilon |y(\theta_t \omega) |
 \ii
  \rho \left ({\frac {|x|^2}{k^2}}  \right )
 f(x,u)  u dx
 \le  \varepsilon C  |y(\theta_t \omega)| 
 \ii
  \rho\left ({\frac {|x|^2}{k^2}}  \right )
  F(x,u) dx
 $$
 \be\label{tai1_p3}
 +\varepsilon C | y(\theta_t \omega) |
 \ii
    \rho\left ({\frac {|x|^2}{k^2}}  \right ) |u|^2dx
     +\varepsilon C  |y(\theta_t \omega) |
 \ii
    \rho\left ({\frac {|x|^2}{k^2}}  \right ) 
    ( |\phi_1|^2 + |\phi_3| )  dx.
\ee
 By \eqref{tai1_p1}-\eqref{tai1_p3} and 
   Young\rq{}s inequality,   we obtain
   (after replacing $\omega$
   by $\omega_n$),
     $$
{\frac d{dt}}   \int_{\mathbb{R}^n}
  \rho\left ({\frac {|x|^2}{k^2}}  \right )
  \left (|v(t, \tau, \omega_n, v_{0,n})|^2 +
  \lambda   | u  (t, \tau, \omega_n,  u_{0,n}) |^2
  +  | \nabla u  |^2 + 2   F(x, u)
 \right ) dx 
 $$
 $$
\le 
  \ii 
  \rho \left ({\frac {|x|^2}{k^2}}  \right )|v(t, \tau, \omega_n, v_{0,n})|^2 
  \left (
  1+ 2\varepsilon |y(\theta_t \omega_n) |
  +\varepsilon^2|y(\theta_t \omega_n) |^2
  \right ) dx
  $$
  $$
   +  \ii 
  \rho \left ({\frac {|x|^2}{k^2}}  \right )| u(t, \tau, \omega_n, u_{0,n})|^2 
  \left ( 
   \varepsilon^2|y(\theta_t \omega_n) |^2
   +  2\varepsilon  \lambda |y(\theta_t \omega_n) |
   +  \varepsilon C   |y(\theta_t \omega_n) |
  \right ) dx
  $$
  $$
  + 2\varepsilon |y(\theta_t \omega_n)|
  \ii
  \rho \left ({\frac {|x|^2}{k^2}}  \right )|\nabla u|^2 dx 
+
\varepsilon C  |y(\theta_t \omega_n) |
  \int_{\R^n}
  \rho\left ({\frac {|x|^2}{k^2}}  \right )
  F(x,u) dx
 $$
 $$
     +\varepsilon C  | y(\theta_t \omega_n) |
  \int_{\R^n}
    \rho\left ({\frac {|x|^2}{k^2}}  \right ) 
    ( |\phi_1|^2 + |\phi_3| )  dx
    +    \int_{\R^n}
    \rho\left ({\frac {|x|^2}{k^2}}  \right ) 
    |g|^2  dx
 $$
\be\label{tai1_p7}
 +{\frac Ck} \left (\|\nabla  u(t, \tau, \omega_n, u_{0,n}) \|^2 
 + \|  v(t, \tau, \omega_n, v_{0,n})\|^2
  \right ).
  \ee
  Since $(u_{0,n}, v_{0,n}) \to (u_0,v_0)$ in $\hone \times \ltwo$
  and $\omega_n \to \omega$,  by Lemma \ref{estf} we find that
  $(u(t, \tau,  \omega_n, u_{0,n}),  v(t, \tau,  \omega_n, v_{0,n}) )$
  is uniformly bounded in $[\tau, \tau+T]$ for $n\in \N$, Therefore,
  there exists $K_1\ge 1$  such that   for    all $k \ge K_1$,
   $ t \in [\tau, \tau+T]$  and  $n\in \N$,
  \be
  \label{tai1_p9}
  {\frac Ck} \left (\|\nabla  u(t, \tau, \omega_n, u_{0,n}) \|^2 
 + \|  v(t, \tau, \omega_n, v_{0,n})\|^2
  \right )
  \le \eta.
  \ee
  On the other hand, by Lemma \ref{omegam} (i) we know
  $y(\theta_t \omega_n) \to  y(\theta_t \omega)$
  uniformly  for  $t\in [\tau, \tau +T]$ as $n \to \infty$, and thus
  there exists $C_1>0$ and $N_1 \ge 1$  such that
  for all $n \ge N_1$  and $t\in [\tau, \tau +T]$,
  \be\label{tai1_p20}
  |y(\theta_t \omega)| \le C_1
  \quad  \mbox{and} \quad
  |y(\theta_t \omega_n)|\le C_1.
  \ee
  By \eqref{f3}  and \eqref{tai1_p20} we get,
  for $n \ge N_1$,
  \be\label{tai1_p22}
  \varepsilon C  |y(\theta_t \omega_n) |
  \int_{\R^n}
  \rho\left ({\frac {|x|^2}{k^2}}  \right )
  F(x,u) dx
  $$
  $$
  \le
  \varepsilon C C_1 
  \int_{\R^n}
  \rho\left ({\frac {|x|^2}{k^2}}  \right )
  F(x,u) dx
  + 2 \varepsilon C C_1 
    \int_{\R^n}
  \rho\left ({\frac {|x|^2}{k^2}}  \right )
  |\phi_3 (x)| dx.
  \ee
  If follows from  \eqref{f3} and \eqref{tai1_p7}-\eqref{tai1_p22} that
  there exists $C_2>0$  such that  for all $n\ge N_1$,
  $k\ge K_1$  and
  $t \in [\tau, \tau+T]$,
    $$
{\frac d{dt}}   \ii
  \rho\left ({\frac {|x|^2}{k^2}}  \right )
  \left (|v(t, \tau, \omega_n, v_{0,n})|^2 +
  \lambda   | u  (t, \tau, \omega_n,  u_{0,n}) |^2
  +  | \nabla u  |^2 + 2   F(x, u)
 \right ) dx 
 $$
 $$
\le 
C_2
 \ii
  \rho\left ({\frac {|x|^2}{k^2}}  \right )
  \left (|v(t, \tau, \omega_n, v_{0,n})|^2 +
  \lambda   | u  (t, \tau, \omega_n,  u_{0,n}) |^2
  +  | \nabla u  |^2 + 2   F(x, u)
 \right ) dx 
$$
  \be\label{tai1_p30}
     +C_2  
\int_{|x| \ge k}
    \rho\left ({\frac {|x|^2}{k^2}}  \right ) 
    ( |\phi_1|^2 + |g|^2 + |\phi_3| )  dx
    +  \eta.
  \ee
  Note that  there exists  
  $K_2 \ge  K_1  $ such that   for all $k \ge K_2$,
  \be\label{tai1_p32}
    C_2   \int_{|x| \ge k}
    \rho\left ({\frac {|x|^2}{k^2}}  \right ) 
   (   |\phi_1  (x)|^2 + |g  (x)|^2 + |\phi_3  (x)|  ) dx
   \le \eta.
 \ee
 By \eqref{tai1_p30}-\eqref{tai1_p32} and Gronwall\rq{}s lemma, we
 obtain,  for all $n\ge N_1$,
  $k\ge K_2$  and
  $t \in [\tau, \tau+T]$,
  $$
  \ii
  \rho\left ({\frac {|x|^2}{k^2}}  \right )
  \left (|v(t, \tau, \omega_n, v_{0,n})|^2 +
  \lambda   | u  (t, \tau, \omega_n,  u_{0,n}) |^2
  +  | \nabla u  |^2 + 2   F(x, u)
 \right ) dx 
 $$
\be\label{tai1_p34}
 \le
  e^{C_2T} \ii
  \rho\left ({\frac {|x|^2}{k^2}}  \right )
  \left (| v_{0,n}|^2 +
  \lambda   |   u_{0,n} |^2
  +  | \nabla u_{0,n}  |^2 + 2   F(x, u_{0,n} )
 \right ) dx 
 + 2\eta C_2^{-1} e^{C_2T}.
\ee
Since $(u_0, v_0) \in \hone \times \ltwo$, there exists
$K_3\ge K_2$ such that for all $k \ge K_3$,
\be\label{tai1_p40}
\int_{|x| \ge k}
( |u_0(x)|^2
+|\nabla u_0(x)|^2
+|v_0(x)|^2 ) dx \le {\frac \eta{4}}.
\ee
Since $(u_{0,n}, v_{0,n})
\to (u_0,v_0)$ in $\hone \times \ltwo$,
there exists $N_2\ge N_1$  such that
for all $n\ge N_2$,
$$
\int_{|x| \ge k}
 ( |u_{0,n} (x) - u_0(x)|^2
+|\nabla u_{0,n} (x) -  \nabla u_0(x)|^2
+|v_{0,n} (x) - v_0(x)|^2 ) dx  
$$
$$
\le \ii 
 ( |u_{0,n} (x) - u_0(x)|^2
+|\nabla u_{0,n} (x) -  \nabla u_0(x)|^2
+|v_{0,n} (x) - v_0(x)|^2 ) dx \le {\frac \eta{4}}
$$
which together with \eqref{tai1_p40} 
implies that for all $n\ge N_2$
and $k\ge K_3$,
\be\label{tai1_p42}
\int_{|x| \ge k}
 ( |u_{0,n} (x)  |^2
+|\nabla u_{0,n} (x) |^2
+|v_{0,n} (x)  |^2 ) dx  
\le \eta.
\ee
By \eqref{f1}  and \eqref{tai1_p42} we find that
there exists $C_3>0$ such that
 for all $n\ge N_2$
and $k\ge K_3$,
$$
   e^{C_2T} \ii
  \rho\left ({\frac {|x|^2}{k^2}}  \right )
  \left (| v_{0,n}|^2 +
  \lambda   |   u_{0,n} |^2
  +  | \nabla u_{0,n}  |^2 + 2   F(x, u_{0,n} )
 \right ) dx 
 $$
 $$
 \le
   e^{C_2T} \int_{|x| \ge k} 
   (| v_{0,n}|^2 +
  \lambda   |   u_{0,n} |^2
  +  | \nabla u_{0,n}  |^2 ) dx
  $$
  \be\label{tai1_p50}
  +
  2 e^{C_2T} \int_{|x| \ge k}
  \rho\left ({\frac {|x|^2}{k^2}}  \right )
     F(x, u_{0,n}  ) dx 
     \le C_3 \eta.
    \ee
 By \eqref{f3}, \eqref{tai1_p34} and \eqref{tai1_p50}
 we get, 
  for all $n\ge N_2$,
  $k\ge K_3$  and
  $t \in [\tau, \tau+T]$,
   $$
  \int_{|x| \ge \sqrt{2} k}
  \left (|v(t, \tau, \omega_n, v_{0,n})|^2 +
  \lambda   | u  (t, \tau, \omega_n,  u_{0,n}) |^2
  +  | \nabla u  |^2  
 \right ) dx 
 $$
  $$
 \le  \ii
  \rho\left ({\frac {|x|^2}{k^2}}  \right )
  \left (|v(t, \tau, \omega_n, v_{0,n})|^2 +
  \lambda   | u  (t, \tau, \omega_n,  u_{0,n}) |^2
  +  | \nabla u  |^2  
 \right ) dx \le C_4\eta.
 $$
 This completes the proof.
 \end{proof}

    Next, we discuss the weak and strong continuity 
    of solutions of system \eqref{pde1}-\eqref{pde3}.
    
    \begin{lem}
    \label{cprv}
    Suppose \eqref{f1}-\eqref{f3} hold,  $\tau \in \R$,
    $t_n \to t$ with $t_n \ge \tau$  and   $\omega_n\to \omega$
    with $\omega_n, \omega \in \Omega_m$.
    Let $ (u(\cdot,\tau, \omega_n, u_{0,n}),
    v(\cdot,\tau, \omega_n, v_{0,n}) )$
     be a solution of \eqref{pde1}-\eqref{pde3}
    with initial data   
    $(u_{0,n}, v_{0,n})$ at initial time $\tau$.
    
    (i)  If  $(u_{0,n}, v_{0,n}) \rightharpoonup 
     (u_0, v_0)$  in $\hone \times \ltwo$, 
     then system \eqref{pde1}-\eqref{pde3} has a solution
     $(u,v) = (u(\cdot, \tau, \omega, u_0),  v(\cdot, \tau, \omega, v_0))$
     with initial condition $(u_0,v_0)$ at initial time $\tau$ such that,
     up to  a subsequence,
     $$ 
     u(t_n, \tau, \omega_n, u_{0,n})
     \rightharpoonup   u(t, \tau, \omega, u_{0})
       \ \mbox{in} \ \hone  
     $$
     and
     $$
     v(t_n, \tau, \omega_n, v_{0,n})
     \rightharpoonup   v(t, \tau, \omega, v_{0})
       \ \mbox{in} \   \ltwo.
     $$
    
     (ii)  If  $(u_{0,n}, v_{0,n})  \to  
     (u_0, v_0)$  in $\hone \times \ltwo$, 
     then system \eqref{pde1}-\eqref{pde3} has a solution
     $(u,v) = (u(\cdot, \tau, \omega, u_0),  v(\cdot, \tau, \omega, v_0))$
     with initial condition $(u_0,v_0)$ at initial time $\tau$ such that,
      up to  a subsequence,
     $$ 
     u(t_n, \tau, \omega_n, u_{0,n})
     \to  u(t, \tau, \omega, u_{0})
       \ \mbox{in} \ \hone  
     $$
     and
     $$
     v(t_n, \tau, \omega_n, v_{0,n})
     \to   v(t, \tau, \omega, v_{0})
       \ \mbox{in} \   \ltwo.
     $$
    \end{lem}    
    
         \begin{proof}
        (i).  Since $t_n \to t$ with $t_n \ge \tau$
         and 
               $(u_{0,n}, v_{0,n}) \rightharpoonup 
     (u_0, v_0)$  in $\hone \times \ltwo$,
     there exist $T>0$  and $M>0$ such that
     for all $n\in \N$,
     \be\label{cprv_p1}
     t_n, t \in [\tau, \tau +T]
     \quad \mbox{and} \quad
     \|(u_{0,n}, v_{0,n})\|_{H^1\times L^2}
     \le M.
     \ee
             Since $\omega_n \to \omega$ with
             $\omega_n, \omega \in \Omega_m$,
             by \eqref{cprv_p1}
             and Lemma \ref{estf} there exists 
             a constant $C>0$ such that
             for all $n\in \N$  and $r\in [\tau, \tau +T]$,
               \be\label{cprv_p2}
               \| u(r,\tau, \omega_n, u_{0,n})\|_{H^1}
               +\|v(r,\tau, \omega_n, v_{0,n})\|_{L^2}
               \le C.
               \ee
Let $\xi_{0,n} = \left (
  \begin{array}{c}
  u_{0,n}\\
  v_{0,n}
  \end{array}
\right )$ and 
 $\xi_{n}(r,\tau, \omega_n, \xi_{0,n}) = \left (
  \begin{array}{c}
  u(r,\tau, \omega_n, u_{0,n}) \\
  v(r,\tau, \omega_n, v_{0,n})
  \end{array}
\right )$. By \eqref{absol} we get
for all $r\ge \tau$  and $\psi \in \hone \times \ltwo$,
\be\label{cprv_p5}
<\xi_n (r, \tau, \omega_n, \xi_{0,n} ), \psi>
= <e^{A(r-\tau)}\xi_{0,n}, \psi>
$$
$$
+\int_\tau^r < e^{A(r-s)} G(\cdot, s,  \omega_n,
 \xi_n(s, \tau, \omega_n, \xi_{0,n})) , \psi> ds,
\ee
where $<\cdot, \cdot>$ is the inner product of
$\hone\times\ltwo$.
Let $G_1(s) = \left (
\begin{array}{c}
0\\
g(s)
\end{array}
\right ) $
and  $G_2(x, s, \omega_n, \xi_n)= G (x, s, \omega_n, \xi_n) -G_1(s)$.
Then by \eqref{G1}, \eqref{cprv_p2} and Lemma \ref{omegam} (i), we
find that there exists $C_1>0$ such that
for all $n\in \N$ and $s\in [\tau, \tau +T]$,
\be\label{cprv_p6}
\|G_2(\cdot, s, \omega_n, \xi_n(s, \tau, \omega_n, \xi_{0,n})
\|_{H^1\times L^2}
\le C_1.
\ee
 We now prove 
 the equicontinuity of 
 $<\xi_n (r, \tau, \omega_n, \xi_{0,n} ), \psi>$
 by the approach of \cite{bal1}.
 By \eqref{cprv_p5} we have for $r, r+p\in
 [\tau, \tau+T]$ with $p\ge 0$,
 \be\label{cprv_p7}
<\xi_n (r +p, \tau, \omega_n, \xi_{0,n} )- \xi_n (r, \tau, \omega_n, \xi_{0,n} ),  \psi>
$$
$$
= <(e^{A(r+p-\tau)}-e^{A(r-\tau)})\xi_{0,n}, \psi>
$$
$$
+\int_\tau^r < (e^{A(r+p-s)}- e^{A(r-s)} )   G(\cdot, s,  \omega_n,
 \xi_n(s, \tau, \omega_n, \xi_{0,n}) ) , \psi> ds
 $$
 $$
 +\int_r^{r+p} <e^{A(r+p-s)} G(\cdot, s,  \omega_n,
 \xi_n(s, \tau, \omega_n, \xi_{0,n}) ) , \psi> ds.
\ee
 Given $\eta_1>0$ and $M>0$, it follows from \cite{bal1}
 that there exists $\eta_2>0$ such that
 \be\label{cprv_p8}
|< (e^{A(r+p)}-e^{Ar} ) \phi,  \psi>|
\le \eta_1
\ee
for all $\|\phi\|_{H^1\times L^2} \le M$
and $r, r+p\in [0, T]$
with $0\le p\le \eta_2$.
By \eqref{cprv_p6}-\eqref{cprv_p8} we obtain 
 the equicontinuity of 
 $<\xi_n (r, \tau, \omega_n, \xi_{0,n} ), \psi>$.
 Therefore,  as in \cite{bal1},   we find that
 there exist a subsequence
 $\xi_{n^\prime} (r, \tau, \omega_{n^\prime}, \xi_{0,n^\prime} )$
 and a weakly continuous mapping $\xi: [\tau, \tau +T]
 \to \hone \times \ltwo$ with
 $\xi(\tau) 
 =\left (
\begin{array}{c}
u_0\\v_0\end{array}
\right )$ 
   such that
  \be\label{cprv_p10}
 \xi_{n^\prime} (r, \tau, \omega_{n^\prime}, \xi_{0,n^\prime} )
 \rightharpoonup
 \xi (r),
 \quad \mbox{uniformly for   }\  r\in[\tau, \tau+T].
 \ee
 Note that the subsequence 
 $ \xi_{n^\prime} (r, \tau, \omega_{n^\prime}, \xi_{0,n^\prime} )$
 satisfies \eqref{cprv_p5}.  By taking the limit as $n^\prime \to \infty$,
 using \eqref{cprv_p5},  \eqref{cprv_p6},
 \eqref{cprv_p10}, the weak continuity of $e^{Ar}$
 and Lemma \ref{Gpro} (i), we can get for $r\in [\tau, \tau+T]$,
 \be\label{cprv_p20}
<\xi(r ), \psi>
= <e^{A(r-\tau)}\xi(\tau), \psi>
+\int_\tau^r < e^{A(r-s)} G(\cdot, s,  \omega,
 \xi(s) ) , \psi> ds.
\ee
Thus  $\xi$  is a weak solution of \eqref{pde1}-\eqref{pde2}
with initial condition $\xi_0 =\xi(\tau) $
at initial time $\tau$, that is,
$\xi =\xi(r, \tau, \omega, \xi_0)$
for $r\in [\tau, \tau+T]$.
By a diagonal process, we can choose a further subsequence (not relabeled)
such that 
 \be\label{cprv_p22}
 \xi_{n^\prime} (r, \tau, \omega_{n^\prime}, \xi_{0,n^\prime} )
 \rightharpoonup
 \xi (r,\tau, \omega, \xi_0),
 \quad \mbox{uniformly for}\  r  \ \mbox{in any compact interval of } \ \R.
 \ee
Since $t_n \to t$, by \eqref{cprv_p22} we obtain
\be\label{cprv_p23}
\xi_{n^\prime} (t_{n^\prime}, \tau, \omega_{n^\prime}, \xi_{0,n^\prime} )
 \rightharpoonup
 \xi (t,\tau, \omega, \xi_0),
 \ee
  as desired.
 
 (ii). We now suppose $(u_{0,n}, v_{0,n}) \to (u_0,v_0)$ strongly
 in $\hone\times \ltwo$ and $t_n \to t$. 
 Next, we show the weak convergence of \eqref{cprv_p23}
 is actually a strong convergence under the current conditions.
 For convenience, in the sequel, we will write the subsequence
  $\xi_{n^\prime} (r, \tau, \omega_{n^\prime}, \xi_{0,n^\prime} )$
  as $\xi_{n} (r, \tau, \omega_n, \xi_{0,n} )$.
 
 Choose a positive number $T$ such that
 $t_n , t\in [\tau, \tau+T]$.
Let $E_1: \hone \times \ltwo \to \R$ be a functional given  by,
for $(u,v) \in \hone \times \ltwo$,
\be\label{cprv_p27}
E_1(u,v)=
 \| v \|^2 +  \lambda   \| u \|^2
  + \| \nabla u \|^2 + 2 \ii  F(x, u) dx.
\ee
It follows    from \eqref{ener1} with $\delta =0$ that
$$
   {\frac {dE_1}{dt}}   
  = - 2  \alpha   \| v\|^2
  +  2(g, v)  -2  \varepsilon y(\theta_t \omega) \|v\|^2
 -2\varepsilon^2  y^2 (\theta_t \omega)  
   (u,v)
 $$
 \begin{equation}
 \label{cprv_p30}
 +2 \varepsilon   \lambda  
  y(\theta_t \omega) \| u \|^2
  + 2\varepsilon y(\theta_t \omega) \| \nabla u \|^2
  +2 \varepsilon y(\theta_t \omega) (f(x,u), u).
 \end{equation}
 Since the sign of $ y(\theta_t \omega)$ is indefinite, so
 are the coefficients of $\|v\|^2$,
 $\|u\|^2$ and $ \|\nabla u \|^2$ in \eqref{cprv_p30}.
 This prevents us from applying  Fatou's Theorem to
 estimate the integrals of these terms in time.
 To solve the problem, we add 
 $-2 \varepsilon |y(\theta_t \omega)| E_1$
 to both sides of \eqref{cprv_p30}  to make
 these terms negative.  By doing so, we obtain  
 $$
   {\frac {dE_1}{dt}}   -2 \varepsilon |y(\theta_t \omega)| E_1 
  = - 2  (\alpha  + \varepsilon (|y(\theta_t \omega)| + y(\theta_t \omega)))  \| v\|^2
  $$
  $$
  -2 \varepsilon \lambda (|y(\theta_t \omega) | - y(\theta_t \omega) ) \| u \|^2
  -2\varepsilon (|y(\theta_t \omega)|- y(\theta_t \omega)) \| \nabla u \|^2
  $$
 \begin{equation}
 \label{cprv_p32}  
  +2 \varepsilon y(\theta_t \omega) (f(x,u), u)
  -4\varepsilon | y(\theta_t \omega)|\ii F(x,u) dx 
  -2 \varepsilon^2 y^2(\theta_t \omega) (u,v)
  +2(g,v).
 \end{equation}
 Solving for $E_1$ from \eqref{cprv_p32} on $[\tau, t]$ we obtain
 $$
 E_1(u(t,\tau, \omega, u_0), v(t, \tau, \omega, v_0) )
 =
 e^{-2\varepsilon \int_t^\tau |y(\theta_r\omega)|dr} E_1(u_0,v_0)
 $$
 $$
-2 \int_\tau^t  e^{-2\varepsilon \int_t^s |y(\theta_r\omega)|dr}
 (\alpha  + \varepsilon (|y(\theta_s \omega)| + y(\theta_s \omega))) 
  \| v(s,\tau, \omega, v_0)\|^2 ds
 $$
 $$
 -2 \varepsilon \lambda
  \int_\tau^t  e^{-2\varepsilon \int_t^s |y(\theta_r\omega)|dr}
   (|y(\theta_s \omega) | - y(\theta_s \omega) ) \| u (s,\tau, \omega, u_0)\|^2ds
 $$
  $$
 -2 \varepsilon  
  \int_\tau^t  e^{-2\varepsilon \int_t^s |y(\theta_r\omega)|dr}
   (|y(\theta_s \omega) | - y(\theta_s \omega) ) \| \nabla u (s,\tau, \omega, u_0)\|^2ds
 $$
 $$
 +2 \varepsilon 
 \int_\tau^t  e^{-2\varepsilon \int_t^s |y(\theta_r\omega)|dr}
 y(\theta_s \omega) (f(x, u(s,\tau,\omega, u_0)), u(s,\tau,\omega, u_0) ) ds
 $$
   $$
 -4  \varepsilon 
 \int_\tau^t  e^{-2\varepsilon \int_t^s |y(\theta_r\omega)|dr}
 |y(\theta_s \omega) | \ii F(x, u(s,\tau, \omega, u_0))dx  ds
 $$
 $$
 -2 \varepsilon ^2
 \int_\tau^t  e^{-2\varepsilon \int_t^s |y(\theta_r\omega)|dr}
 y^2(\theta_s \omega) ( u(s,\tau,\omega, u_0), v(s,\tau,\omega, v_0)) ds
 $$
 \be\label{cprv_p40}
 +2   
 \int_\tau^t  e^{-2\varepsilon \int_t^s |y(\theta_r\omega)|dr}
   (g, v(s,\tau,\omega, v_0)) ds.
\ee
By \eqref{cprv_p27} and \eqref{cprv_p40} we obtain
$$
 \| v(t,\tau, \omega, v_{0}) \|^2 
 +  \lambda   \| u (t,\tau, \omega, u_{0}) \|^2
  + \| \nabla u(t,\tau, \omega, u_{0}) \|^2 
  $$
  $$
 =- 2 \ii  F(x, u(t,\tau, \omega, u_{0}) ) dx
  $$
  $$
+ e^{-2\varepsilon \int_{t}^\tau |y(\theta_r\omega)|dr} 
 (\| v_{0} \|^2 +  \lambda   \| u_{0} \|^2
  + \| \nabla u_{0} \|^2 + 2 \ii  F(x, u_{0}) dx )
 $$
 $$
-2 \int_\tau^{t}  e^{-2\varepsilon \int_{t}^s |y(\theta_r\omega)|dr}
 (\alpha  + \varepsilon (|y(\theta_s \omega)| + y(\theta_s \omega))) 
  \| v(s,\tau, \omega, v_{0})\|^2 ds
 $$
 $$
 -2 \varepsilon \lambda
  \int_\tau^{t}  e^{-2\varepsilon \int_{t}^s |y(\theta_r\omega)|dr}
   (|y(\theta_s \omega) | - y(\theta_s \omega) ) 
   \| u (s,\tau, \omega, u_{0})\|^2ds
 $$
  $$
 -2 \varepsilon  
  \int_\tau^{t}  e^{-2\varepsilon \int_{t}^s |y(\theta_r\omega)|dr}
   (|y(\theta_s \omega) | - y(\theta_s \omega) )
    \| \nabla u (s,\tau, \omega, u_{0})\|^2ds
 $$
 $$
 +2 \varepsilon 
 \int_\tau^{t }  e^{-2\varepsilon \int_{t}^s |y(\theta_r\omega)|dr}
 y(\theta_s \omega) (f(x, u(s,\tau,\omega, u_{0})), 
 u(s,\tau,\omega, u_{0}) ) ds
 $$
   $$
 -4  \varepsilon 
 \int_\tau^{t}  e^{-2\varepsilon \int_{t}^s |y(\theta_r\omega)|dr}
 |y(\theta_s \omega)| \ii F(x, u(s,\tau, \omega, u_{0}))dx  ds
 $$
 $$
 -2 \varepsilon ^2
 \int_\tau^{t}  e^{-2\varepsilon \int_{t}^s |y(\theta_r\omega)|dr}
 y^2(\theta_s \omega) ( u(s,\tau,\omega, u_{0}), v(s,\tau,\omega, v_{0})) ds
 $$
 \be\label{cprv_p41}
 +2   
 \int_\tau^{t}  e^{-2\varepsilon \int_{t}^s |y(\theta_r\omega)|dr}
   (g, v(s,\tau,\omega, v_{0})) ds.
\ee 
Replacing $t$ by $t_{n}$, $\omega$ by $\omega_ {n}$,
$u_{0, {n}}$ by $u_0$, $v_{0, {n}}$ by $v_0$
in  \eqref{cprv_p41},    we get
$$
 \| v(t_ {n},\tau, \omega_ {n}, v_{0, {n}}) \|^2 
 +  \lambda   \| u (t_ {n},\tau, \omega_ {n}, u_{0, {n}}) \|^2
  + \| \nabla u(t_ {n},\tau, \omega_ {n}, u_{0, {n}}) \|^2 
  $$
  $$
 =- 2 \ii  F(x, u(t_ {n} ,\tau, \omega_ {n}, u_{0, {n}}) ) dx
  $$
  $$
+ e^{-2\varepsilon \int_{t_ {n}}^\tau |y(\theta_r\omega_ {n})|dr} 
 (\| v_{0, {n}} \|^2 +  \lambda   \| u_{0, {n}} \|^2
  + \| \nabla u_{0, {n}} \|^2 + 2 \ii  F(x, u_{0, {n}}) dx )
 $$
 $$
-2 \int_\tau^{t_n}  e^{-2\varepsilon \int_{t_n}^s |y(\theta_r\omega_n)|dr}
 (\alpha  + \varepsilon (|y(\theta_s \omega_n)| + y(\theta_s \omega_n))) 
  \| v(s,\tau, \omega_n, v_{0,n})\|^2 ds
 $$
 $$
 -2 \varepsilon \lambda
  \int_\tau^{t_n}  e^{-2\varepsilon \int_{t_n}^s |y(\theta_r\omega_n)|dr}
   (|y(\theta_s \omega_n) | - y(\theta_s \omega_n) ) 
   \| u (s,\tau, \omega_n, u_{0,n})\|^2ds
 $$
  $$
 -2 \varepsilon  
  \int_\tau^{t_n}  e^{-2\varepsilon \int_{t_n}^s |y(\theta_r\omega_n)|dr}
   (|y(\theta_s \omega_n) | - y(\theta_s \omega_n) )
    \| \nabla u (s,\tau, \omega_n, u_{0,n})\|^2ds
 $$
 $$
 +2 \varepsilon 
 \int_\tau^{t _n}  e^{-2\varepsilon \int_{t_n}^s |y(\theta_r\omega_n)|dr}
 y(\theta_s \omega_n) (f(x, u(s,\tau,\omega_n, u_{0,n})), 
 u(s,\tau,\omega_n, u_{0,n}) ) ds
 $$
   $$
 -4  \varepsilon 
 \int_\tau^{t_n}  e^{-2\varepsilon \int_{t_n}^s |y(\theta_r\omega_n)|dr}
 |y(\theta_s \omega_n)| \ii F(x, u(s,\tau, \omega_n, u_{0,n}))dx  ds
 $$
 $$
 -2 \varepsilon ^2
 \int_\tau^{t_n}  e^{-2\varepsilon \int_{t_n}^s |y(\theta_r\omega_n)|dr}
 y^2(\theta_s \omega_n) ( u(s,\tau,\omega_n, u_{0,n}), v(s,\tau,\omega_n, v_{0,n})) ds
 $$
 \be\label{cprv_p45}
 +2   
 \int_\tau^{t_n}  e^{-2\varepsilon \int_{t_n}^s |y(\theta_r\omega_n)|dr}
   (g, v(s,\tau,\omega_n, v_{0,n})) ds.
\ee
 We need to examine the limit of each term in \eqref{cprv_p45}.
 By \eqref{cprv_p22} we have for each $r \in [\tau, \tau +T]$,
 \be\label{cprv_p47}
 u(r,\tau, \omega_n, u_{0,n})
 \rightharpoonup 
 u(r,\tau, \omega, u_{0})
 \quad \mbox{in} \  \hone,
 \ee
 and
 \be\label{cprv_p49}
 v(r,\tau, \omega_n, v_{0,n})
 \rightharpoonup 
 v(r,\tau, \omega, v_{0})
 \quad \mbox{in} \  \ltwo.
 \ee
 We claim that  for every $r\in [ \tau,\tau+T]$,
 \be\label{cprv_p50}
  u(r,\tau, \omega_n, u_{0,n})
 \to 
 u(r,\tau, \omega, u_{0})
 \quad \mbox{ strongly in} \  \ltwo.
 \ee
 Since $(u_{0,n}, v_{0,n})
 \to (u_0,v_0)$ in $\hone \times \ltwo$
 and $\omega_n \to \omega$
 with $\omega_n, \omega\in \Omega_m$, by Lemma \ref{tai1} we
 find that for every $\eta>0$, there exist $N_1 \ge 1$
 and $K\ge 1$ such that for all $n \ge N_1$  and $r\in [\tau, \tau+T]$,
 \be\label{cprv_p52}
 \int_{|x| \ge K}
 |u(r, \tau, \omega_n, u_{0,n})|^2 dx \le  {\frac \eta{8}}
 \quad \mbox{and} \quad
  \int_{|x| \ge K}
 |u(r, \tau, \omega, u_{0})|^2 dx \le  {\frac \eta{8}}.
 \ee
 On the other hand, by the compactness of embedding
 $H^1(B_K) \hookrightarrow L^2(B_K)$ with
 $B_K = \{ x\in \R^n: |x| <  K\}$, 
 we get from \eqref{cprv_p47} that
 \be\label{cprv_p54}
 u(r,\tau, \omega_n, u_{0,n})
\to 
 u(r,\tau, \omega, u_{0})
 \quad \mbox{strongly in} \  L^2(B_K).
 \ee
 By \eqref{cprv_p54}, there exists $N_2\ge N_1$ such that
 for all $n \ge N_2$,
 \be\label{cprv_p56}
 \int_{|x|<K} |u(r,\tau, \omega_n, u_{0,n})-
 u(r,\tau, \omega, u_{0})|^2 dx \le  {\frac \eta{2}}.
 \ee
 It  follows from  \eqref{cprv_p52} and \eqref{cprv_p56} that
 for all  $n \ge N_2$,
 $$
 \ii |u(r,\tau, \omega_n, u_{0,n})-
 u(r,\tau, \omega, u_{0})|^2 dx \le   \eta,
$$
which implies \eqref{cprv_p50}.
Similarly, we also have
\be\label{cprv_p57}
  u(t_n,\tau, \omega_n, u_{0,n})
 \to 
 u(t,\tau, \omega, u_{0})
 \quad \mbox{ strongly in} \  \ltwo,
 \ee
 which follows from \eqref{cprv_p23}, \eqref{cprv_p52}
 and the arguments of \eqref{cprv_p50}.
By \eqref{f1} and \eqref{cprv_p2} we get
 $$
 \ii |F(x, u(t_n, \tau,  \omega_n, u_{0,n})) - 
 F(x, u(t, \tau,  \omega, u_{0}) )| dx
 $$
 $$
 \le 
 C (\|   u(t_n, \tau,  \omega_n, u_{0,n} )\|_{H^1}^\gamma
 + \|   u(t, \tau,  \omega, u_{0} )\|_{H^1}^\gamma
 + \|\phi_1\|^2)\|u(t_n,\tau, \omega_n, u_{0,n})-
 u(t,\tau, \omega, u_{0})\|
 $$
 $$
 \le C_1  \|u(t_n,\tau, \omega_n, u_{0,n})-
 u(t,\tau, \omega, u_{0})\|,
 $$
 which together with \eqref{cprv_p57} yields
 \be\label{cprv_p58}
 \ii  F(x, u(t_n, \tau,  \omega_n, u_{0,n}) )dx
 \to 
 \ii F(x, u(t, \tau,  \omega, u_{0}) )  dx.
 \ee
 Analogously,  by $u_{0,n} \to  u_0$  we can get
  \be\label{cprv_p60}
 \ii  F(x,   u_{0,n} ) dx
 \to 
 \ii F(x,   u_{0})    dx.
 \ee
By  \eqref{cprv_p2},  \eqref{cprv_p50},
Lemma \ref{omegam} (i), the
   arguments of \eqref{cprv_p58}
 and the Lebesgue dominated convergence theorem, we obtain
\be\label{cprv_p62}
 \lim_{n\to \infty}
 \int_\tau^{t_n}  e^{-2\varepsilon \int_{t_n}^s |y(\theta_r\omega_n)|dr}
 y(\theta_s \omega_n) \ii F(x, u(s,\tau, \omega_n, u_{0,n}))dx  ds
 $$
 $$
=
 \int_\tau^{t}  e^{-2\varepsilon \int_{t}^s |y(\theta_r\omega)|dr}
 y(\theta_s \omega) \ii F(x, u(s,\tau, \omega, u_{0}))dx  ds.
 \ee
 By \eqref{f1} we have for all $s\in [\tau, \tau +T]$,
 \be\label{cprv_p64}
 |f(x, u(s,\tau,\omega_n, u_{0,n})) 
 u(s,\tau,\omega_n, u_{0,n})|
 \le c_1  | u(s,\tau,\omega_n, u_{0,n})|^{\gamma +1}
 +|\phi_1| |u(s,\tau,\omega_n, u_{0,n})|.
 \ee
By \eqref{cprv_p2} and \eqref{cprv_p50} we see that
the right-hand side of \eqref{cprv_p64} is convergent
in $L^1(\R^n)$. Thus by a  dominated convergence theorem
in \cite{bal1} we obtain 
\be\label{cprv_p66}
 \lim_{n \to \infty}
 \ii f(x, u(s,\tau,\omega_n, u_{0,n})) 
 u(s,\tau,\omega_n, u_{0,n}) dx
 =
  \ii f(x, u(s,\tau,\omega, u_{0})) 
 u(s,\tau,\omega, u_{0}) dx.
 \ee
 By \eqref{cprv_p2}, Lemma \ref{omegam} (i) and
 the Lebesgue  dominated convergence  theorem,
 we obtain from \eqref{cprv_p66} that
 \be\label{cprv_p68}
\lim_{n \to \infty}
 \int_\tau^{t _n}  e^{-2\varepsilon \int_{t_n}^s |y(\theta_r\omega_n)|dr}
 y(\theta_s \omega_n) (f(x, u(s,\tau,\omega_n, u_{0,n})), 
 u(s,\tau,\omega_n, u_{0,n}) ) ds
 $$
 $$
 =
 \int_\tau^{t }  e^{-2\varepsilon \int_{t}^s |y(\theta_r\omega)|dr}
 y(\theta_s \omega) (f(x, u(s,\tau,\omega, u_{0})), 
 u(s,\tau,\omega, u_{0}) ) ds.
\ee
We now deal with the third term on the right-hand side of
\eqref{cprv_p45}. Note that
 $$
 \int_\tau^{t_n}  e^{-2\varepsilon \int_{t_n}^s |y(\theta_r\omega_n)|dr}
 (\alpha  + \varepsilon (|y(\theta_s \omega_n)| + y(\theta_s \omega_n))) 
  \| v(s,\tau, \omega_n, v_{0,n})\|^2 ds
  $$
  $$
  = 
  e^{-2\varepsilon \int_{t_n}^t |y(\theta_r\omega_n)|dr}
   \int_\tau^t e^{-2\varepsilon \int_{t}^s |y(\theta_r\omega_n)|dr}
   (\alpha  + \varepsilon (|y(\theta_s \omega_n)| + y(\theta_s \omega_n))) 
  \| v(s,\tau, \omega_n, v_{0,n})\|^2 ds
  $$
\be\label{cprv_p70}
  +
  e^{-2\varepsilon \int_{t_n}^t |y(\theta_r\omega_n)|dr}
   \int_t^{t_n} e^{-2\varepsilon \int_{t}^s |y(\theta_r\omega_n)|dr}
   (\alpha  + \varepsilon (|y(\theta_s \omega_n)| + y(\theta_s \omega_n))) 
  \| v(s,\tau, \omega_n, v_{0,n})\|^2 ds.
 \ee
Since $\omega_n \to \omega$
and $t_n \to t$,   by \eqref{estf_p6}  and \eqref{cprv_p2} we get
\be\label{cprv_p72}
\lim_{n\to \infty}
e^{-2\varepsilon \int_{t_n}^t |y(\theta_r\omega_n)|dr} =1,
\ee
and
\be\label{cprv_p74}
 \lim_{n\to \infty} 
   \int_t^{t_n} e^{-2\varepsilon \int_{t}^s |y(\theta_r\omega_n)|dr}
   (\alpha  + \varepsilon (|y(\theta_s \omega_n)| + y(\theta_s \omega_n))) 
  \| v(s,\tau, \omega_n, v_{0,n})\|^2 ds
  =0.
  \ee
On the other hand, by Fatou\rq{}s Theorem and Lemma \ref{omegam} (i),
 we have
$$
\liminf_{n\to \infty}
 \int_\tau^{t}  e^{-2\varepsilon \int_{t}^s |y(\theta_r\omega_n)|dr}
 (\alpha  + \varepsilon (|y(\theta_s \omega_n)| + y(\theta_s \omega_n))) 
  \| v(s,\tau, \omega_n, v_{0,n})\|^2 ds
  $$
  $$
  \ge 
  \int_\tau^{t}\liminf_{n\to \infty} ( e^{-2\varepsilon \int_{t}^s |y(\theta_r\omega_n)|dr}
 (\alpha  + \varepsilon (|y(\theta_s \omega_n)| + y(\theta_s \omega_n))) 
  \| v(s,\tau, \omega_n, v_{0,n})\|^2) ds
  $$
\be\label{cprv_p76}
  \ge 
  \int_\tau^{t}  e^{-2\varepsilon \int_{t}^s |y(\theta_r\omega)|dr}
 (\alpha  + \varepsilon (|y(\theta_s \omega)| + y(\theta_s \omega))) 
  \liminf_{n\to \infty}
  \| v(s,\tau, \omega_n, v_{0,n})\|^2 ds.
 \ee
 Since  
  $v(s,\tau, \omega_n, v_{0,n}) \rightharpoonup
  v(s,\tau, \omega, v_0)$   by \eqref{cprv_p49}, we get
  from \eqref{cprv_p76} that
   $$
\liminf_{n\to \infty}
 \int_\tau^{t}  e^{-2\varepsilon \int_{t}^s |y(\theta_r\omega_n)|dr}
 (\alpha  + \varepsilon (|y(\theta_s \omega_n)| + y(\theta_s \omega_n))) 
  \| v(s,\tau, \omega_n, v_{0,n})\|^2 ds
  $$
  \be\label{cprv_p78}
  \ge 
  \int_\tau^{t}  e^{-2\varepsilon \int_{t}^s |y(\theta_r\omega)|dr}
 (\alpha  + \varepsilon (|y(\theta_s \omega)| + y(\theta_s \omega))) 
  \| v(s,\tau, \omega, v_{0})\|^2 ds.
 \ee
 It follows   from \eqref{cprv_p70}-\eqref{cprv_p74}
 and \eqref{cprv_p78} that
 $$
\liminf_{n\to \infty}
 \int_\tau^{t_n}  e^{-2\varepsilon \int_{t_n}^s |y(\theta_r\omega_n)|dr}
 (\alpha  + \varepsilon (|y(\theta_s \omega_n)| + y(\theta_s \omega_n))) 
  \| v(s,\tau, \omega_n, v_{0,n})\|^2 ds
  $$
 \be\label{cprv_p80}
  \ge 
  \int_\tau^{t}  e^{-2\varepsilon \int_{t}^s |y(\theta_r\omega)|dr}
 (\alpha  + \varepsilon (|y(\theta_s \omega)| + y(\theta_s \omega))) 
  \| v(s,\tau, \omega, v_{0})\|^2 ds.
 \ee
 By similar arguments, we can also obtain
 $$
\liminf_{n\to \infty}
 \int_\tau^{t_n}  e^{-2\varepsilon \int_{t_n}^s |y(\theta_r\omega_n)|dr}
     (|y(\theta_s \omega_n)| + y(\theta_s \omega_n) ) 
  \| \nabla u (s,\tau, \omega_n, u_{0,n})\|^2 ds
  $$
 \be\label{cprv_p82}
  \ge 
  \int_\tau^{t}  e^{-2\varepsilon \int_{t}^s |y(\theta_r\omega)|dr}
    (|y(\theta_s \omega)| + y(\theta_s \omega) ) 
  \| \nabla u (s,\tau, \omega, u_{0})\|^2 ds.
 \ee
 Now taking the limit superior of \eqref{cprv_p45} as $n\to \infty$,
 since $(u_{0,n}, v_{0,n}) \to (u_0, v_0)$ in $\hone\times \ltwo$,
 by Lemma \ref{omegam} (i), \eqref{cprv_p2},
  \eqref{cprv_p49}-\eqref{cprv_p50},
 \eqref{cprv_p58}-\eqref{cprv_p62},
 \eqref{cprv_p68} and
 \eqref{cprv_p80}-\eqref{cprv_p82}, we get
 $$
    \limsup_{n\to \infty}
    \left (
 \| v(t_ {n},\tau, \omega_ {n}, v_{0, {n}}) \|^2 
 +  \lambda   \| u (t_ {n},\tau, \omega_ {n}, u_{0, {n}}) \|^2
  + \| \nabla u(t_ {n},\tau, \omega_ {n}, u_{0, {n}}) \|^2 
  \right )
  $$
  $$
 \le - 2 \ii  F(x, u(t ,\tau, \omega, u_{0}) ) dx
  $$
  $$
+ e^{-2\varepsilon \int_{t}^\tau |y(\theta_r\omega)|dr} 
 (\| v_{0} \|^2 +  \lambda   \| u_{0} \|^2
  + \| \nabla u_{0} \|^2 + 2 \ii  F(x, u_{0}) dx )
 $$
 $$
-2 \int_\tau^{t}  e^{-2\varepsilon \int_{t}^s |y(\theta_r\omega)|dr}
 (\alpha  + \varepsilon (|y(\theta_s \omega)| + y(\theta_s \omega))) 
  \| v(s,\tau, \omega, v_{0})\|^2 ds
 $$
 $$
 -2 \varepsilon \lambda
  \int_\tau^{t}  e^{-2\varepsilon \int_{t}^s |y(\theta_r\omega)|dr}
   (|y(\theta_s \omega) | - y(\theta_s \omega) ) 
   \| u (s,\tau, \omega, u_{0})\|^2ds
 $$
  $$
 -2 \varepsilon  
  \int_\tau^{t}  e^{-2\varepsilon \int_{t}^s |y(\theta_r\omega)|dr}
   (|y(\theta_s \omega) | - y(\theta_s \omega) )
    \| \nabla u (s,\tau, \omega, u_{0})\|^2ds
 $$
 $$
 +2 \varepsilon 
 \int_\tau^{t }  e^{-2\varepsilon \int_{t}^s |y(\theta_r\omega)|dr}
 y(\theta_s \omega) (f(x, u(s,\tau,\omega, u_{0})), 
 u(s,\tau,\omega, u_{0}) ) ds
 $$
   $$
 -4  \varepsilon 
 \int_\tau^{t}  e^{-2\varepsilon \int_{t}^s |y(\theta_r\omega)|dr}
 y(\theta_s \omega) \ii F(x, u(s,\tau, \omega, u_{0}))dx  ds
 $$
 $$
 -2 \varepsilon ^2
 \int_\tau^{t}  e^{-2\varepsilon \int_{t}^s |y(\theta_r\omega)|dr}
 y^2(\theta_s \omega) ( u(s,\tau,\omega, u_{0}), v(s,\tau,\omega, v_{0})) ds
 $$
 \be\label{cprv_p90}
 +2   
 \int_\tau^{t}  e^{-2\varepsilon \int_{t}^s |y(\theta_r\omega)|dr}
   (g, v(s,\tau,\omega, v_{0})) ds.
\ee         
Note that the right-hand side of \eqref{cprv_p90}
is exactly the same as that of \eqref{cprv_p41}. Thus we obtain
   $$
    \limsup_{n\to \infty}
    \left (
 \| v(t_ {n},\tau, \omega_ {n}, v_{0, {n}}) \|^2 
 +  \lambda   \| u (t_ {n},\tau, \omega_ {n}, u_{0, {n}}) \|^2
  + \| \nabla u(t_ {n},\tau, \omega_ {n}, u_{0, {n}}) \|^2 
  \right )
  $$
\be\label{cprv_p92}
  \le 
 \| v(t,\tau, \omega, v_{0}) \|^2 
 +  \lambda   \| u (t,\tau, \omega, u_{0}) \|^2
  + \| \nabla u(t,\tau, \omega, u_{0}) \|^2 .
 \ee
 On the other hand, 
 since the right-hand side of \eqref{cprv_p92} 
 is equivalent to  the  norm 
 of $(u,v)$ in $\hone \times \ltwo$,
   by \eqref{cprv_p47}-\eqref{cprv_p49} we have
  $$
    \liminf_{n\to \infty}
    \left (
 \| v(t_ {n},\tau, \omega_ {n}, v_{0, {n}}) \|^2 
 +  \lambda   \| u (t_ {n},\tau, \omega_ {n}, u_{0, {n}}) \|^2
  + \| \nabla u(t_ {n},\tau, \omega_ {n}, u_{0, {n}}) \|^2 
  \right )
  $$
\be\label{cprv_p94}
\ge
 \| v(t,\tau, \omega, v_{0}) \|^2 
 +  \lambda   \| u (t,\tau, \omega, u_{0}) \|^2
  + \| \nabla u(t,\tau, \omega, u_{0}) \|^2 .
 \ee
 It follows  from \eqref{cprv_p92}-\eqref{cprv_p94} that
  $$
    \lim _{n\to \infty}
    \left (
 \| v(t_ {n},\tau, \omega_ {n}, v_{0, {n}}) \|^2 
 +  \lambda   \| u (t_ {n},\tau, \omega_ {n}, u_{0, {n}}) \|^2
  + \| \nabla u(t_ {n},\tau, \omega_ {n}, u_{0, {n}}) \|^2 
  \right )
  $$
$$
=
 \| v(t,\tau, \omega, v_{0}) \|^2 
 +  \lambda   \| u (t,\tau, \omega, u_{0}) \|^2
  + \| \nabla u(t,\tau, \omega, u_{0}) \|^2 
$$
which along with 
  \eqref{cprv_p47}-\eqref{cprv_p49} implies
 $$
  (   u (t_ {n},\tau, \omega_ {n}, u_{0, {n}}),
    v(t_ {n},\tau, \omega_ {n}, v_{0, {n}}))
    \to 
      (   u (t,\tau, \omega, u_{0}),
    v(t,\tau, \omega, v_{0}))
    \quad \mbox{in } \ \hone \times \ltwo
    $$
     as desired.    \end{proof}

      As an immediate consequence of Lemma \ref{cprv},
      we get the following weak continuity of solutions of
      \eqref{pde1}-\eqref{pde3} which  is  useful for
      proving  the asymptotic compactness of solutions later.
      
        \begin{lem}
    \label{weakc}
    Suppose \eqref{f1}-\eqref{f3} hold,  $\tau \in \R$
      and   $ \omega \in \Omega$.
    Let $ (u(\cdot,\tau, \omega, u_{0,n}),
    v(\cdot,\tau, \omega, v_{0,n}) )$
     be a solution of \eqref{pde1}-\eqref{pde3}
    with initial data   
    $(u_{0,n}, v_{0,n})$ at initial time $\tau$.
      If  $(u_{0,n}, v_{0,n}) \rightharpoonup 
     (u_0, v_0)$  in $\hone \times \ltwo$, 
     then system \eqref{pde1}-\eqref{pde3} has a solution
     $(u,v) = (u(\cdot, \tau, \omega, u_0),  v(\cdot, \tau, \omega, v_0))$
     with initial condition $(u_0,v_0)$ at initial time $\tau$ such that,
     up to  a subsequence, for all $t \ge \tau$,
     $$ 
     u(t, \tau, \omega, u_{0,n})
     \rightharpoonup   u(t, \tau, \omega, u_{0})
       \ \mbox{in} \ \hone  
     $$
     and
     $$
     v(t, \tau, \omega, v_{0,n})
     \rightharpoonup   v(t, \tau, \omega, v_{0})
       \ \mbox{in} \   \ltwo.
     $$
 \end{lem}
 
 \begin{proof}
 Given $\omega \in \Omega$, since 
 $\Omega = \bigcup\limits_{m=1}^\infty \Omega_m$ by Lemma \ref{omegam},
 we find that there exists $m \in \N$ such that
 $\omega \in \Omega_m$.
 Then the result follows from Lemma \ref{cprv} (i) for
 $t_n =t$ and $\omega_n =\omega$ for all $n\in \N$.
  \end{proof}

        In the sequel,  we
        write $(U(t+\tau,\tau, \omega, u_0), Z(t+\tau,
        \tau, \omega, z_0))$  for the collection of all solutions of
        \eqref{spde1}-\eqref{spde3}   at time $t+\tau$  with initial condition
        $(u_0, z_0)$  and initial time $\tau$; that is,
        $$
        (U(t+\tau,\tau, \omega, u_0), Z(t+\tau,
        \tau, \omega, z_0))
        $$
        $$
        = \left \{ ( u(t+\tau , \tau, \omega, u_0),  z(t+\tau, \tau, \omega, z_0) ): 
         (u,z) \mbox{ is a solution of  $\eqref{spde1}$-$\eqref{spde3}$}
 \right \}.
        $$
       For the 
         measurability of $(U,Z)$, we have the following result.
         
         \begin{lem}
         \label{meuz}
         Suppose \eqref{f1}-\eqref{f3} hold and $\tau \in \R$.
         Then the multivalued mapping
         $$
         (U(\cdot, \tau, \cdot, \cdot), Z(\cdot, \tau, \cdot, \cdot)):
          \R^+\times \Omega \times \hone\times \ltwo
         \to  2^{H^1(\mathbb{R}^n) \times L^2(\mathbb{R}^n)}
         $$
         is measurable with respect to $\calb(\R^+) \times
         \calb(\Omega) \times \calb(\hone \times \ltwo)$. 
   \end{lem}
   
   \begin{proof}
   We  first show 
   $(U(\cdot, \tau, \cdot, \cdot), Z(\cdot, \tau, \cdot, \cdot))$:
         $\R^+\times \Omega_m \times \hone\times \ltwo
         \to  2^{H^1(\mathbb{R}^n) \times L^2(\mathbb{R}^n)}$
         is  weakly upper semicontinuous  for every fixed $m\in \N$.
         Suppose $t_n \to t$ with $t_n \ge 0$,
         $\omega_n \to \omega$ with $\omega_n, \omega
         \in \Omega_m$, and $(u_{0,n}, z_{0,n}) \to
         (u_0, z_0)$ in $\hone \times \ltwo$.
         Let $(u(t,\tau, \omega_n, u_{0,n}), z(t,\tau, \omega_n, z_{0,n}))$
         be a solution of \eqref{spde1}-\eqref{spde3} with
         initial condition $(u_{0,n}, z_{0,n})$ at $\tau$.
         By \eqref{tranz},
         $(u(t,\tau, \omega_n, u_{0,n}), v(t,\tau, \omega_n, v_{0,n}))$
         is a   solution of \eqref{pde1}-\eqref{pde3} with
         initial condition $(u_{0,n}, v_{0,n})$ at $\tau$,  where
         \be\label{meuz_p1}
         v(t,\tau, \omega_n, v_{0,n})
         =  z(t,\tau, \omega_n, z_{0,n})
         -\varepsilon y(\theta_t \omega_n)
         u(t,\tau, \omega_n, u_{0,n})
         \  \mbox{and} \ 
         v_{0,n} = z_{0,n}  -\varepsilon y(\theta_\tau \omega_n)
           u_{0,n} .
         \ee
         Let $v_0 = z_{0}  -\varepsilon y(\theta_\tau \omega) u_{0}$.
         Since $(u_{0,n}, z_{0,n}) \to (u_0, z_0)$ in
         $\hone \times \ltwo$, we get
         $(u_{0,n}, v_{0,n}) \to (u_0, v_0)$ in
         $\hone \times \ltwo$.
         By  Lemma \ref{cprv} (i),  we find that  system \eqref{pde1}-\eqref{pde3} has
         a solution 
         $(u(\cdot,\tau, \omega, u_{0}), v(\cdot,\tau, \omega, v_{0}) )$
         with initial data $(u_0,v_0)$ at $\tau$ such that, up to a subsequence,
         \be\label{meuz_p2}
         (u(t_n,\tau, \omega_n, u_{0,n}), v(t_n,\tau, \omega_n, v_{0,n} ))
         \rightharpoonup
         (u(t,\tau, \omega, u_{0}), v(t,\tau, \omega, v_{0}) )
         \ee
         in $\hone \times \ltwo$.
         By \eqref{cprv_p2},  \eqref{meuz_p1}-\eqref{meuz_p2} and Lemma \ref{omegam} (i),
          we obtain
         \be\label{meuz_p3}
         z(t_n,\tau, \omega_n, z_{0,n} )
         \rightharpoonup v(t ,\tau, \omega , v_{0 } )
         +  \varepsilon y(\theta_t \omega)
         u(t,\tau, \omega, u_{0})
         =z(t,\tau, \omega, z_{0} )
       \ee
         where the last equality follows from \eqref{tranz}.
         By \eqref{meuz_p2}-\eqref{meuz_p3} we have
           \be\label{meuz_p4}
         (u(t_n,\tau, \omega_n, u_{0,n}), z(t_n,\tau, \omega_n, z_{0,n} ))
            \rightharpoonup
         (u(t,\tau, \omega, u_{0}), z(t,\tau, \omega, z_{0}) )
         \ee
         in $\hone \times \ltwo$.
       Note that \eqref{meuz_p4} implies the weak upper semicontinuity
       of $(U(\cdot, \tau, \cdot, \cdot), Z(\cdot, \tau, \cdot, \cdot))$:
         $\R^+\times \Omega_m \times \hone\times \ltwo
         \to  2^{H^1(\mathbb{R}^n) \times L^2(\mathbb{R}^n)}$, and
         thus by Lemma \ref{mcocy}
           it is measurable as a multi-valued mapping  from 
         $\R^+\times \Omega_m \times \hone\times \ltwo$
         to $2^{H^1(\mathbb{R}^n) \times L^2(\mathbb{R}^n)}$
         with respect to 
         $\calb(\R^+) \times
         \calb(\Omega_m) \times \calb(\hone \times \ltwo)$.
         Since $\Omega =\bigcup\limits_{m=1}^\infty
         \Omega_m$ and $\Omega_m$ is measurable by Lemma \ref{omegam},
         we find that
         the mapping
          $(U(\cdot, \tau, \cdot, \cdot), V(\cdot, \tau, \cdot, \cdot))$
           from
          $\R^+\times \Omega  \times \hone\times \ltwo$
         to $ 2^{H^1(\mathbb{R}^n) \times L^2(\mathbb{R}^n)}$
           is measurable
         with respect to
         $\calb(\R^+) \times
         \calb(\Omega) \times \calb(\hone \times \ltwo)$.
         This completes    the proof.
   \end{proof}

  \begin{lem}\label{close}
    Suppose \eqref{f1}-\eqref{f3} hold.
    Then for every $t\in \R^+$, $\tau \in \R$,
    $\omega \in \Omega$ and $(u_0,v_0) \in \hone \times \ltwo$,
    the set
    $(U(t+\tau,\tau, \omega, u_0), Z(t+\tau,
        \tau, \omega, z_0))$  is closed in $\hone \times \ltwo$. 
  \end{lem}
  
  \begin{proof}
  Let 
$({\widetilde{u}}_n,  {\widetilde{z}}_n)
\in  (U(t+\tau,\tau, \omega, u_0), Z(t+\tau,
        \tau, \omega, z_0))$
        and  $({\widetilde{u}},  {\widetilde{z}}) \in \hone \times \ltwo$
        such that
        $({\widetilde{u}}_n,  {\widetilde{z}}_n)
        \to ({\widetilde{u}},  {\widetilde{z}})$.
  We    will prove
$({\widetilde{u}},  {\widetilde{z}})
\in
(U(t+\tau,\tau, \omega, u_0), Z(t+\tau,
        \tau, \omega, z_0))$.
     Since    $({\widetilde{u}}_n,  {\widetilde{z}}_n)
\in  (U(t+\tau,\tau, \omega, u_0), Z(t+\tau,
        \tau, \omega, z_0))$,
  there  exists  
$(u_n(\cdot, \tau,  \omega, u_{0}),
z_n(\cdot, \tau,  \omega, z_{0}))$ 
such that
\be\label{closep2}
({\widetilde{u}}_n,  {\widetilde{z}}_n)
= (u_n( t+\tau, \tau,  \omega, u_{0}),
z_n(t+\tau, \tau, \omega, z_{0})) .
\ee
By \eqref{tranz} and \eqref{closep2}  we have
\be\label{closep5}
({\widetilde{u}}_n,  {\widetilde{z}}_n)
= (u_n( t+\tau, \tau,  \omega, u_{0}),
v_n(t+\tau, \tau, \omega, v_{0})
+  \varepsilon y(\theta_{t+\tau} \omega)
u_n(t+\tau, \tau,  \omega, u_{0}) )
\ee
with 
$v_{0} =z_{0}
-\varepsilon y(\theta_{\tau} \omega) u_{0}$.
For given $\omega$, 
by Lemma \ref{omegam}, there exists $m \in \N$
such that $\omega \in \Omega_m$.
Then applying  Lemma \ref{cprv} (i)
to $(u_n(\cdot, \tau,  \omega, u_{0}),
v_n(\cdot, \tau,  \omega, v_{0}))$ 
with $t_n =t+\tau$,
$\omega_n =\omega$,
$u_{0,n} = u_0$
and $v_{0,n} = v_0$ we find that
   system \eqref{pde1}-\eqref{pde3} has a solution
     $(u(\cdot, \tau, \omega, u_0),  v(\cdot, \tau, \omega, v_0))$
     with initial condition $(u_0,v_0)$ at initial time $\tau$ such that,
     up to  a subsequence,
     $$ 
     u_n(t+\tau, \tau, \omega, u_{0})
     \rightharpoonup   u(t+\tau, \tau, \omega, u_{0})
       \ \mbox{in} \ \hone  
     $$
     and
     $$
     v_n(t+\tau, \tau, \omega, v_{0})
     \rightharpoonup   v(t, \tau, \omega, v_{0})
       \ \mbox{in} \   \ltwo,
     $$
which together with \eqref{closep5} and \eqref{tranz} implies
\be\label{closep7}
({\widetilde{u}}_n,  {\widetilde{z}}_n)
 \rightharpoonup
 (  u(t+\tau, \tau, \omega, u_{0}),  z(t+\tau, \tau, \omega, z_{0}) )
 \ \mbox{ in } \ 
 \hone \times \ltwo.
\ee
Since  $({\widetilde{u}}_n,  {\widetilde{z}}_n)
        \to ({\widetilde{u}},  {\widetilde{z}})$
        strongly in $\hone\times \ltwo$, 
        by \eqref{closep7} we get
        $$
          ({\widetilde{u}},  {\widetilde{z}})
        =  (  u(t+\tau, \tau, \omega, u_{0}),  z(t+\tau, \tau, \omega, z_{0}) )
\in
(U(t+\tau,\tau, \omega, u_0), Z(t+\tau,
        \tau, \omega, z_0))
        $$
as desired.
  \end{proof}

We are now ready to define a multi-valued non-autonomous
random dynamical system for problem \eqref{spde1}-\eqref{spde3}. 
    Let  
 $\Phi$: $\mathbb{R}^+ \times \R \times \Omega \times
 H^1(\mathbb{R}^n) \times L^2(\mathbb{R}^n)$
 $\to 2^{H^1(\mathbb{R}^n) \times L^2(\mathbb{R}^n)}$
 be a multi-valued mapping  given   by
 \begin{equation}
 \label{rdsw1}
 \Phi(t,  \tau, \omega, (u_0, z_0))
 =  ( U(t +\tau, \tau, \theta_{-\tau}
 \omega, u_0),  Z(t+\tau, \tau, \theta_{-\tau}
 \omega, z_0) )
 $$
 $$
 = \left \{ ( u(t +\tau, \tau, \theta_{-\tau}
 \omega, u_0),  z(t+\tau, \tau, \theta_{-\tau}
 \omega, z_0) ): \  (u,z) \mbox{ is a solution of  $\eqref{spde1}$-$\eqref{spde3}$ }
 \right \}
  \end{equation}
  for every
  $(t, \tau, \omega, (u_0, z_0)) \in \mathbb{R}^+
  \times \R \times \Omega
  \times H^1(\mathbb{R}^n) \times L^2(\mathbb{R}^n)$.

  Since $\theta_{-\tau}: \Omega \to \Omega$ is measurable, by Lemma
  \ref{meuz} we infer that
  for every $\tau \in \R$,
  $\Phi(\cdot, \tau, \cdot, \cdot) $:
         $\R^+\times \Omega \times \hone\times \ltwo
         \to  2^{H^1(\mathbb{R}^n) \times L^2(\mathbb{R}^n)}$
         is measurable. 
         It is clear that $\Phi$ satisfies
         (ii) and (iii) of Definition \ref{cocy}. Therefore,   
 by  Lemma \ref{close},  
           $\Phi$ is   
  a multivalued non-autonomous
cocycle   in  $\hone \times \ltwo$.
In addition,  by   Lemma \ref{cprv} (ii)
we find  the following upper semicontinuity of $\Phi$ 
which is needed for construction of random attractors.

  \begin{lem}\label{uscphi}
    Suppose \eqref{f1}-\eqref{f3} hold.
    Then for every $t\in \R^+$, $\tau \in \R$  and 
    $\omega \in \Omega$,
    the mapping 
    $\Phi(t, \tau, \omega, \cdot)$:
     $\hone \times \ltwo \to$
     $2^{\hone \times \ltwo}$
     is upper semicontinuous. 
  \end{lem}
  
  \begin{proof}
  This  follows
  directly
  from Lemma \ref{cprv} (ii)  and
  Lemma \ref{uscc}.
  \end{proof}

     By \eqref{tranz}  we have
    \begin{equation}
 \label{rdsw2}
  z(t+\tau, \tau, \theta_{-\tau}
 \omega, z_0)=
 v(t+\tau, \tau, \theta_{-\tau}
 \omega, v_0)
 +\varepsilon y(\theta_t \omega) 
  u(t +\tau, \tau, \theta_{-\tau}
 \omega, u_0) 
 \end{equation}
 with $   v_0 = z_0 -  \varepsilon y(\omega)u_0$.
 It is worth noticing that  
for each  $t\in \R^+$,
 $\tau \in \R$ and  $\omega \in \Omega$, 
 $$
 \Phi(t, \tau -t,  \theta_{-t} \omega, (u_0, z_0))
 $$
 $$
 =
 \left \{
 (u(\tau ,  \tau -t  , \theta_{-\tau} \omega, u_0),
  z(\tau , \tau -t ,  \theta_{- \tau}\omega, z_0))
   : \  (u,z) \mbox{ is a solution of  $\eqref{spde1}$-$\eqref{spde3}$}
  \right \}
  $$
  {\footnotesize
 \begin{equation}
 \label{shift}
 =
 \left \{
 (u(\tau ,  \tau -t  , \theta_{-\tau} \omega, u_0),
  v(\tau , \tau -t ,  \theta_{- \tau}\omega, v_0)
  + \varepsilon  y(\omega)
 u(\tau ,  \tau -t  , \theta_{-\tau} \omega, u_0)  )
  :   (u,v) \mbox{ is a solution of  $\eqref{pde1}$-$\eqref{pde3}$}
  \right \}.
 \end{equation}
 }

In the rest of the paper, we will investigate
the pullback asymptotic behavior of $\Phi$. 
From now on,   we fix a small  positive  number
$\delta$ such that 
\be
\label{delta}
\alpha -\delta>0 , \quad   \lambda +\delta^2 -\alpha \delta >0,
\ee
and  put 
\begin{equation}
\label{kappa}
\sigma =  \min \left  \{ {\frac {\delta}2}, \  {\frac {\alpha-\delta}4},  \ 
   {\frac {\delta c_2}4} \right \},
  \end{equation}
  where $c_2$ is the positive constant in \eqref{f2}.
  From now on,  we assume  $g$  satisfies
 \be
 \label{g1}
 \int_{-\infty}^0 
 e^{{\frac {1}{2\gamma +2}} \sigma  s }
  \| g(s+\tau,\cdot)\|^2ds
<  \infty ,
\quad \mbox{for all } \ \tau \in \R,
  \ee
   where $\gamma$ is the number     in \eqref{f1}.

As usual, for 
   a  bounded nonempty  subset 
 $B$  of $\hone \times \ltwo$,  the distance between $B$ and the
 origin is written as 
   $  \| B\| = \sup\limits_{ \phi  \in B}
   \| \phi\|_{\hone \times \ltwo }$. 
   In the next section, we will prove $\Phi$  has a $\cald$-pullback
   random attractor where 
   $\cald = \{ 
   D =\{ D(\tau, \omega): \tau \in \R, \omega \in \Omega \} \}$ is
   the  collection of all       families
 of 
  bounded nonempty   subsets of $\hone \times \ltwo $
 such that 
  for every    $\tau \in \R$   and $\omega \in \Omega$, 
$$
 \lim_{s \to   \infty} e^{ -  \sigma s} 
 \| D( \tau  -s, \theta_{ -s} \omega ) \| ^{\gamma +1} =0.
$$

\section{Asymptotic compactness of solutions}
\setcounter{equation}{0}

In this section,  we 
 derive uniform estimates   of   
system  \eqref{pde1}-\eqref{pde3} and establish
the pullback asymptotic compactness   of   solutions.
   We first provide the uniform 
   estimates in $\hone \times \ltwo$.
   
\begin{lem}
\label{lest}
 Assume that  
 \eqref{f1}-\eqref{f3}  and \eqref{g1}
  hold.  Let    $\tau \in \R$, $\omega \in \Omega$   
and $D=\{D(\tau, \omega)
: \tau \in \R,  \omega \in \Omega\}  \in \cald$.
   Then there exist   
    $\varepsilon_0 = \varepsilon_0 (\alpha, \lambda, f) \in (0,1)$
and   $T=T(\tau, \omega,  D) >0$ such that 
 for all  $\varepsilon \in (0, \varepsilon_0]$,
  $ t\ge  T$ and $s\in [-t, 0]$,
the solution $(u,v)$ of problem \eqref{pde1}-\eqref{pde3}
satisfies
$$
  \|    u(\tau +s, \tau -t, \theta_{-\tau}\omega, u_0) \|^2_\hone
  + \|  v(\tau +s, \tau -t, \theta_{-\tau}\omega, v_0 ) \|^2
  $$
  $$
  + \|u(\tau +s, \tau -t, \theta_{-\tau}\omega, u_0 )\|^{\gamma +1}_{L^{\gamma +1}
  (\R^n)}
  $$
  $$
  +\int_{\tau -t}^{\tau +s}
  e^{\int^{\xi}_{\tau+s} 
  (2\sigma -\varepsilon c
-\varepsilon c | y(\theta_{r-\tau} \omega)|^2)dr }
\left (
\|u(\xi, \tau-t, \theta_{-\tau} \omega, u_0 )\|^2_{H^1}
+ 
\|v(\xi, \tau-t, \theta_{-\tau} \omega, v_0 )\|^2
\right ) d\xi
  $$
  $$
   \le M   + 
 e^{ \int_{s}^{ 0}(2\sigma -\varepsilon c
-\varepsilon c | y(\theta_{r} \omega)|^2)dr }
R(\tau, \omega),
$$
 where $(u_0, z_0)\in D(\tau -t, \theta_{ -t} \omega)$
 and $v_0 = z_0 -\varepsilon y(\omega) u_0$,
   $M$ and $c$  are
   positive     constants  independent of $\tau$, $\omega$, $D$ 
  and $\varepsilon$, and $R(\tau, \cdot)$ is 
  a random variable.
  \end{lem}

\begin{proof}
This lemma can be proved by energy equation
\eqref{ener1}. The details are quite similar to
 the proof  Lemma 4.2
  in \cite{wan9} with minor  changes,  and hence 
  are omitted here.
\end{proof}

The next estimate is an improvement of Lemma 4.3
in \cite{wan9} on the tails of solutions outside a
bounded domains when time is sufficiently large.

\begin{lem}
\label{utai}
Assume that  
 \eqref{f1}-\eqref{f3}  and \eqref{g1}
  hold.  Let  $\eta>0$,   $\tau \in \R$, $\omega \in \Omega$   
and $D=\{D(\tau, \omega)
: \tau \in \R,  \omega \in \Omega\}  \in \cald$.
   Then there exist   
    $\varepsilon_0 = \varepsilon_0 (\alpha, \lambda, f)\in (0,1)$,
       $T=T(\tau, \omega,  D, \eta) >0$ 
       and $K=K(\tau, \omega, \eta) \ge 1$
       such that 
 for all  $\varepsilon \in (0, \varepsilon_0]$,
 $ t\ge  T$  and $s\in [-t,0]$,
the solution $(u,v)$ of problem \eqref{pde1}-\eqref{pde3}
satisfies
$$
\int_{|x| \ge K} 
\left (
| u(\tau +s, \tau -t,  \theta_{ -\tau} \omega, u_{0}  ) |^2
+| \nabla u(\tau +s, \tau -t,  \theta_{ -\tau} \omega, u_{0}  ) |^2
+
| v(\tau+s, \tau -t,  \theta_{ -\tau} \omega, v_{0}  ) |^2 
\right ) dx 
$$
$$
\le \eta 
+\eta
  e^{ \int_{s}^{ 0}(2\sigma -\varepsilon c
-\varepsilon c | y(\theta_{r} \omega)|^2)dr },
$$
  where $(u_0, z_0)\in D(\tau -t, \theta_{ -t} \omega)$
 and $v_0 = z_0 -\varepsilon y(\omega) u_0$.
 \end{lem}
  
 \begin{proof}
 Let $\rho$ be the smooth cut-off function given by
 \eqref{rho}.
 Taking the inner product of \eqref{pde2}
 with $\rho\left ({\frac {|x|^2}{k^2}}  \right ) 
 v$ in $L^2(\mathbb{R}^n)$, by \eqref{f1}-\eqref{f3}
 and \eqref{pde1}, after some calculations, we obtain the
 following inequality (see   Lemma 4.3
in \cite{wan9} for more details).
 $$
{\frac d{dt}}   \int_{\mathbb{R}^n}
  \rho\left ({\frac {|x|^2}{k^2}}  \right )
  \left (|v|^2 +
  (\lambda +\delta^2 -\alpha \delta)  | u  |^2
  +  | \nabla u  |^2 + 2   F(x, u)
 \right ) dx
 $$
 $$
 +  ( 2\sigma - \varepsilon c -
  \varepsilon c  |y(\theta_t \omega) |^2 )
  \int_{\mathbb{R}^n}
  \rho\left ({\frac {|x|^2}{k^2}}  \right )
  \left (|v|^2 +
  (\lambda +\delta^2 -\alpha \delta)  | u  |^2
  +  | \nabla u  |^2 + 2   F(x, u)
 \right ) dx
$$
 $$
\le 
{\frac ck} \left (\|\nabla u \|^2 + \| v\|^2
  \right )
     + c   \int_{|x| \ge k}
    \rho\left ({\frac {|x|^2}{k^2}}  \right ) 
      |g|^2   dx
   $$
 \be\label{utai_p1}
   +
 \varepsilon c  (1+ |y(\theta_t \omega) |^2)
  \int_{|x| \ge k}
    \rho\left ({\frac {|x|^2}{k^2}}  \right ) 
    ( |\phi_1|^2 + |\phi_2| + |\phi_3| )  dx.
\ee
By \eqref{utai_p1},  for every  $\eta>0$,  there exists
  $K_1 = K_1(\eta) \ge 1$ such that   for all $k \ge K_1$,
   $$
{\frac d{dt}}   \int_{\mathbb{R}^n}
  \rho\left ({\frac {|x|^2}{k^2}}  \right )
  \left (|v|^2 +
  (\lambda +\delta^2 -\alpha \delta)  | u  |^2
  +  | \nabla u  |^2 + 2   F(x, u)
 \right ) dx
 $$
 $$
 +  ( 2\sigma - \varepsilon c -
  \varepsilon c  |y(\theta_t \omega) |^2 )
  \int_{\mathbb{R}^n}
  \rho\left ({\frac {|x|^2}{k^2}}  \right )
  \left (|v|^2 +
  (\lambda +\delta^2 -\alpha \delta)  | u  |^2
  +  | \nabla u  |^2 + 2   F(x, u)
 \right ) dx
$$
   \be\label{utai_p2}
\le 
\eta  \left ( 1+ \|\nabla u \|^2 + \| v\|^2
  \right )
  +
  \varepsilon  \eta   (1+ |y(\theta_t \omega) |^2)
     + c   \int_{|x| \ge k} 
      |g(t,x)|^2  dx.
      \ee
      Multiplying \eqref{utai_p2} by
      $e^{ \int_0^{t}(2\sigma -\varepsilon c
-\varepsilon c | y(\theta_r \omega)|^2)dr }$
and then integrating on $(\tau-t,\tau+s)$ with
$t\ge 0$  and $s\in [-t,0]$,  
      we  get  for all $k \ge K_1$,  
  $$
  \ii
  \rho\left ({\frac {|x|^2}{k^2}}  \right )
  \left (|v (\tau +s, \tau -t, \omega, v_0)  |^2 +
  (\lambda +\delta^2 -\alpha \delta)  | u (\tau+s, \tau -t, \omega, u_0) |^2
  \right ) dx
  $$
  $$
  + 
    \ii
  \rho\left ({\frac {|x|^2}{k^2}}  \right )
 \left ( | \nabla u (\tau+s, \tau -t, \omega, u_0) |^2 
  + 2   F(x, u (\tau+s, \tau -t, \omega, u _0))
 \right ) dx
 $$
  $$
  \le
  e^{ \int_{\tau+s}^{\tau -t}(2\sigma -\varepsilon c
-\varepsilon c | y(\theta_r \omega)|^2)dr }
 \int_{\mathbb{R}^n}
  \rho\left ({\frac {|x|^2}{k^2}}  \right )
 \left (
 | v_0 (x)  |^2 + (\lambda +\delta^2 -\alpha \delta) | u _0 (x) |^2
 \right ) dx 
 $$
 $$+
 e^{ \int_ {\tau+s}^{\tau -t}(2\sigma -\varepsilon c
-\varepsilon c | y(\theta_r \omega)|^2)dr }
 \int_{\mathbb{R}^n}
  \rho\left ({\frac {|x|^2}{k^2}}  \right )
 \left (
    | \nabla u_0  (x) |^2 + 2 F(x, u_0 (x) )
  \right )dx
  $$
  $$
  +\eta  \int^ {\tau+s}_{\tau -t} 
   e^{ \int_ {\tau+s}^{\xi}(2\sigma -\varepsilon c
-\varepsilon c | y(\theta_r \omega)|^2)dr }
\left (
  1+ \|\nabla u (\xi, \tau -t, \omega, u_0) \|^2 + \| v
  (\xi,\tau -t, \omega, v_0) \|^2
\right ) d\xi
  $$
   $$
  + \varepsilon \eta  \int^{\tau+s}_{\tau -t} 
   e^{ \int_{\tau+s}^{\xi}(2\sigma -\varepsilon c
-\varepsilon c | y(\theta_r \omega)|^2)dr }
\left (
 1+  |y(\theta_\xi \omega)|^2 
\right ) d\xi
  $$
    \be\label{utai_p5}
  + c \int^ {\tau+s}_{\tau -t} 
   e^{ \int_ {\tau+s}^{\xi}(2\sigma -\varepsilon c
-\varepsilon c | y(\theta_r \omega)|^2)dr }
\int_{|x| \ge k} |g(\xi ,x)|^2 dx d\xi.
  \ee
 Replacing 
  $\omega$  by $\theta_{-\tau} \omega$
  in \eqref{utai_p5} we  get,  
  for all $t\ge 0$,  $s\in [-t,0]$
  and  $k \ge K_1$, 
   $$
  \ii
  \rho\left ({\frac {|x|^2}{k^2}}  \right )
  \left (|v (\tau +s, \tau -t, \theta_{-\tau}\omega, v_0)  |^2 +
  (\lambda +\delta^2 -\alpha \delta)  
  | u (\tau+s, \tau -t, \theta_{-\tau}\omega, u_0) |^2
  \right ) dx
  $$
  $$
  + 
    \ii
  \rho\left ({\frac {|x|^2}{k^2}}  \right )
 \left ( | \nabla u (\tau+s, \tau -t, \theta_{-\tau}\omega, u_0) |^2 
  + 2   F(x, u (\tau+s, \tau -t, \theta_{-\tau}\omega, u _0))
 \right ) dx
 $$
  $$
  \le
  e^{ \int_{\tau+s}^{\tau -t}(2\sigma -\varepsilon c
-\varepsilon c | y(\theta_{r-\tau} \omega)|^2)dr }
 \int_{\mathbb{R}^n}
  \rho\left ({\frac {|x|^2}{k^2}}  \right )
 \left (
 | v_0 (x)  |^2 + (\lambda +\delta^2 -\alpha \delta) | u _0 (x) |^2
 \right ) dx 
 $$
 $$+
 e^{ \int_ {\tau+s}^{\tau -t}(2\sigma -\varepsilon c
-\varepsilon c | y(\theta_{r-\tau} \omega)|^2)dr }
 \int_{\mathbb{R}^n}
  \rho\left ({\frac {|x|^2}{k^2}}  \right )
 \left (
    | \nabla u_0  (x) |^2 + 2 F(x, u_0 (x) )
  \right )dx
  $$
  $$
  +\eta  \int^ {\tau+s}_{\tau -t} 
   e^{ \int_ {\tau+s}^{\xi}(2\sigma -\varepsilon c
-\varepsilon c | y(\theta_{r-\tau} \omega)|^2)dr }
\left (
  1+ \|\nabla u (\xi, \tau -t,  \theta_{-\tau}\omega, u_0) \|^2 + \| v
  (\xi,\tau -t, \theta_{-\tau}\omega, v_0) \|^2
\right ) d\xi
  $$
   $$
  + \varepsilon \eta  \int^{\tau+s}_{\tau -t} 
   e^{ \int_{\tau+s}^{\xi}(2\sigma -\varepsilon c
-\varepsilon c | y(\theta_{r-\tau} \omega)|^2)dr }
\left (
 1+  |y(\theta_{\xi -\tau} \omega)|^2 
\right ) d\xi
  $$
 $$
  + c \int^ {\tau+s}_{\tau -t} 
   e^{ \int_ {\tau+s}^{\xi}(2\sigma -\varepsilon c
-\varepsilon c | y(\theta_{r-\tau} \omega)|^2)dr }
\int_{|x| \ge k} |g(\xi ,x)|^2 dx d\xi
$$
 $$
  \le
  e^{ \int_{s}^{ -t}(2\sigma -\varepsilon c
-\varepsilon c | y(\theta_{r} \omega)|^2)dr }
\ii
  \rho\left ({\frac {|x|^2}{k^2}}  \right )
 \left (
 | v_0 (x)  |^2 + (\lambda +\delta^2 -\alpha \delta) | u _0 (x) |^2
 \right ) dx 
 $$
 $$+
 e^{ \int_ {s}^{-t}(2\sigma -\varepsilon c
-\varepsilon c | y(\theta_{r} \omega)|^2)dr }
 \int_{\mathbb{R}^n}
  \rho\left ({\frac {|x|^2}{k^2}}  \right )
 \left (
    | \nabla u_0  (x) |^2 + 2 F(x, u_0 (x) )
  \right )dx
  $$
  $$
  +\eta  \int^ {\tau+s}_{\tau -t} 
   e^{ \int_ {\tau+s}^{\xi}(2\sigma -\varepsilon c
-\varepsilon c | y(\theta_{r-\tau} \omega)|^2)dr }
\left (
    \|\nabla u (\xi, \tau -t,  \theta_{-\tau}\omega, u_0) \|^2 + \| v
  (\xi,\tau -t, \theta_{-\tau}\omega, v_0) \|^2
\right ) d\xi
  $$
   $$
  + (1+\varepsilon ) \eta  \int^{s}_{ -t} 
   e^{ \int_{s}^{\xi}(2\sigma -\varepsilon c
-\varepsilon c | y(\theta_{r} \omega)|^2)dr }
\left (
 1+  |y(\theta_{\xi } \omega)|^2 
\right ) d\xi
  $$
 $$
  + c \int^ {s}_{ -t} 
   e^{ \int_ {s}^{\xi}(2\sigma -\varepsilon c
-\varepsilon c | y(\theta_{r} \omega)|^2)dr }
\int_{|x| \ge k} |g(\xi+\tau ,x)|^2 dx d\xi
$$
 $$
  \le
  C e^{ \int_{s}^{ -t}(2\sigma -\varepsilon c
-\varepsilon c | y(\theta_{r} \omega)|^2)dr }
\ii
 \left ( \| v_0  \|^2 + \|  u _0  \|^2_{H^1} + \|u_0\|^{\gamma +1}_{H^1}
 \right ) dx 
 $$
  $$
  +\eta  \int^ {\tau+s}_{\tau -t} 
   e^{ \int_ {\tau+s}^{\xi}(2\sigma -\varepsilon c
-\varepsilon c | y(\theta_{r-\tau} \omega)|^2)dr }
\left (
    \|\nabla u (\xi, \tau -t,  \theta_{-\tau}\omega, u_0) \|^2 + \| v
  (\xi,\tau -t, \theta_{-\tau}\omega, v_0) \|^2
\right ) d\xi
  $$
   $$
  + (1+\varepsilon ) \eta  \int^{0}_{ - \infty} 
   e^{ \int_{s}^{\xi}(2\sigma -\varepsilon c
-\varepsilon c | y(\theta_{r} \omega)|^2)dr }
\left (
 1+  |y(\theta_{\xi } \omega)|^2 
\right ) d\xi
  $$
 $$
  + c \int^ {0}_{ -\infty} 
   e^{ \int_ {s}^{\xi}(2\sigma -\varepsilon c
-\varepsilon c | y(\theta_{r} \omega)|^2)dr }
\int_{|x| \ge k} |g(\xi+\tau ,x)|^2 dx d\xi
$$
 $$
  \le
  C_1 e^{ \int_{s}^{ 0}(2\sigma -\varepsilon c
-\varepsilon c | y(\theta_{r} \omega)|^2)dr }
e^{ \int_{0}^{ -t}(2\sigma -\varepsilon c
-\varepsilon c | y(\theta_{r} \omega)|^2)dr }
(1 + \| D(\tau -t, \theta_{-t} \omega )\|^{\gamma +1})
 $$
  $$
  +\eta 
   \int^ {\tau+s}_{\tau -t} 
   e^{ \int_ {\tau+s}^{\xi}(2\sigma -\varepsilon c
-\varepsilon c | y(\theta_{r-\tau} \omega)|^2)dr }
\left (
    \|\nabla u (\xi, \tau -t,  \theta_{-\tau}\omega, u_0) \|^2 + \| v
  (\xi,\tau -t, \theta_{-\tau}\omega, v_0) \|^2
\right ) d\xi
  $$
   $$
  + (1+\varepsilon ) \eta 
  e^{ \int_{s}^{ 0}(2\sigma -\varepsilon c
-\varepsilon c | y(\theta_{r} \omega)|^2)dr }
    \int^{0}_{ - \infty} 
   e^{ \int_{0}^{\xi}(2\sigma -\varepsilon c
-\varepsilon c | y(\theta_{r} \omega)|^2)dr }
\left (
 1+  |y(\theta_{\xi } \omega)|^2 
\right ) d\xi
  $$
 \be\label{utai_p7}
  + c 
  e^{ \int_{s}^{ 0}(2\sigma -\varepsilon c
-\varepsilon c | y(\theta_{r} \omega)|^2)dr }
  \int^ {0}_{ -\infty} 
   e^{ \int_ {0}^{\xi}(2\sigma -\varepsilon c
-\varepsilon c | y(\theta_{r} \omega)|^2)dr }
\int_{|x| \ge k} |g(\xi+\tau ,x)|^2 dx d\xi.
\ee
Note that
$$
\lim_{t\to \infty} {\frac 1t}
\int_{-t}^0 |y(\theta_r \omega)|^2 dr
=E(y^2) = {\frac 1{2\alpha}}.
$$
Therefore, there  exists $T=T(\omega)>0$ such that
for all $t \ge T$,
\be\label{utai_p11}
{\frac 1t}
\int_{-t}^0 |y(\theta_r \omega)|^2 dr
\le  {\frac 1{\alpha}}.
\ee
Let $\varepsilon_1 ={\frac {\alpha \sigma}{c(1+\alpha)}}$.
Then by  \eqref{utai_p11} we get, for all $\varepsilon \in (0, \varepsilon_1)$
and $t\ge T$,
\be\label{utai_p14}
 e^{ \int_{0}^{ -t}(2\sigma -\varepsilon c
-\varepsilon c | y(\theta_{r} \omega)|^2)dr }
\le e^{- \sigma  t}.
\ee
Since $D\in\cald$,  by \eqref{utai_p14} we have
for all $\varepsilon \in (0, \varepsilon_1)$,
\be\label{utai_p20}
 \lim_{t \to \infty}
  e^{ \int_0^{ -t}(2\sigma -\varepsilon c
-\varepsilon c | y(\theta_{r} \omega)|^2)dr }
  (1 + \| D(\tau -t, \theta_{-t} \omega )\|^{\gamma +1})
=0.
 \ee
 By  \eqref{utai_p14} and  \eqref{g1} we find
  \be\label{utai_p22}
  \int^ {0}_{ -\infty} 
   e^{ \int_ {0}^{\xi}(2\sigma -\varepsilon c
-\varepsilon c | y(\theta_{r} \omega)|^2)dr }
\int_{\R^n} |g(\xi+\tau ,x)|^2 dx d\xi <\infty,
\ee
and
\be\label{utai_p24}
    \int^{0}_{ - \infty} 
   e^{ \int_{0}^{\xi}(2\sigma -\varepsilon c
-\varepsilon c | y(\theta_{r} \omega)|^2)dr }
\left (
 1+  |y(\theta_{\xi } \omega)|^2 
\right ) d\xi <\infty.
 \ee
 Notice that \eqref{utai_p22} implies
 \be\label{utai_p26}
 \lim_{k\to \infty}
  \int^ {0}_{ -\infty} 
   e^{ \int_ {0}^{\xi}(2\sigma -\varepsilon c
-\varepsilon c | y(\theta_{r} \omega)|^2)dr }
\int_{|x| \ge k } |g(\xi+\tau ,x)|^2 dx d\xi =0
\ee
By \eqref{utai_p7}, \eqref{utai_p20},
\eqref{utai_p24}-\eqref{utai_p26}  and 
Lemma \ref{lest},  we 
infer that   there exist  $K_2\ge K_1$
and $T_1\ge T$ such that 
  for all $t\ge T_1$,  $s\in [-t,0]$,
  $\varepsilon \in (0, \varepsilon_1)$,
  and  $k \ge K_1$, 
   $$
  \ii
  \rho\left ({\frac {|x|^2}{k^2}}  \right )
  \left (|v (\tau +s, \tau -t, \theta_{-\tau}\omega, v_0)  |^2 +
  (\lambda +\delta^2 -\alpha \delta)  
  | u (\tau+s, \tau -t, \theta_{-\tau}\omega, u_0) |^2
  \right ) dx
  $$
  $$
  + 
    \ii
  \rho\left ({\frac {|x|^2}{k^2}}  \right )
 \left ( | \nabla u (\tau+s, \tau -t, \theta_{-\tau}\omega, u_0) |^2 
  + 2   F(x, u (\tau+s, \tau -t, \theta_{-\tau}\omega, u _0))
 \right ) dx
 $$
 \be\label{utai_p30}
 \le 
 C_2\eta
 +C_2\eta
   e^{ \int_{s}^{ 0}(2\sigma -\varepsilon c
-\varepsilon c | y(\theta_{r} \omega)|^2)dr }.
\ee
Note that  $ \rho\left ({\frac {|x|^2}{k^2}}  \right ) =1$
for $|x| \ge \sqrt{2} k$, which 
together  with
\eqref{f3} and \eqref{utai_p30} completes the proof.
 \end{proof}
  
    Next, we prove the asymptotic compactness
    of solutions of   \eqref{pde1}-\eqref{pde3}
    by the energy equation \eqref{ener1}.
    To this end, we  set for $(u,v) \in \hone \times \ltwo$,
    \be\label{enerE}
     E(u,v)
    =\| v\|^2 + (\lambda + \delta^2 -\alpha \delta) \| u\|^2
    +\| \nabla u \|^2 +2\ii F(x, u) dx.
    \ee
    By \eqref{ener1} we get
    \be\label{ener}
    {\frac {dE}{dt}}
    + (4 \sigma -2\varepsilon |y(\theta_t \omega)|) E(u,v)
    =G(u,v),
    \ee
    where  $G(u,v)$ is give by
    \be\label{enerH}
    G ( u(t,\tau,\omega, u_0), v (t,\tau,\omega, v_0 ))
    =
    -2  \left (
    \alpha -\delta-2\sigma + \varepsilon (|y(\theta_t \omega)|
    + y(\theta_t \omega) )
    \right ) \| v\|^2
    $$
    $$
    -2 (\lambda +\delta^2 -\alpha \delta)
    (\delta -2\sigma  + \varepsilon (|y(\theta_t \omega)|
    - y(\theta_t \omega)  )) \| u\|^2
    $$
    $$
    -2 (
    \delta -2\sigma  + \varepsilon (|y(\theta_t \omega)|
    - y(\theta_t \omega)  )) \| \nabla u \|^2
    +2 (g,v)
    $$
    $$
    -2 \varepsilon (\varepsilon y(\theta_t \omega) -2 \delta)
    y(\theta_t \omega) (u,v)
    +2 (\varepsilon y(\theta_t \omega) -\delta ) (f(x,u), u)
    $$
    $$
    + 4 (2\sigma -\varepsilon  |y(\theta_t \omega)| )
    \ii F(x, u) dx.
    \ee
    Multiplying \eqref{ener} by
    $e^{\int_0^t  (4 \sigma -2\varepsilon |y(\theta_r \omega)|)  dr }$
    and  then solving for $E(u,v)$,
     we obtain, for
    $\tau \in \R$, $t \ge \tau$ and $\omega \in \Omega$,
    \be\label{ene}
    E(u(t,\tau, \omega, u_0),   v(t,\tau, \omega, v_0) )
    = 
    e^{\int_t^\tau  (4 \sigma -2\varepsilon |y(\theta_r \omega)|)  dr } E(u_0,v_0)
    $$
    $$
    + \int_\tau^t
    e^{\int_t^s  (4 \sigma -2\varepsilon |y(\theta_r \omega)|)  dr }
    G(u(s,\tau, \omega, u_0),   v(s,\tau, \omega, v_0) ) ds.
    \ee

\begin{lem}
\label{asyv}
Assume that  
 \eqref{f1}-\eqref{f3}  and \eqref{g1}
  hold. 
   Then there exists   
    $\varepsilon_0 = \varepsilon_0 (\alpha, \lambda, f)\in (0,1)$
       such that 
 for all  $\varepsilon \in (0, \varepsilon_0]$,
 $\tau \in \R$  and $\omega \in \Omega$, 
 the  sequence   
 $$
 \left \{ (u (\tau,  \tau -t_n, \theta_{-\tau} \omega,  u_{0,n} ),
  v (\tau,  \tau -t_n, \theta_{-\tau} \omega,  v_{0,n}) )
  \right \}_{n=1}^\infty
  $$
  of solutions of 
 \eqref{pde1}-\eqref{pde3}
    has a
 convergent  subsequence in $ H^1(\mathbb{R}^n)
\times L^2(\mathbb{R}^n)$   if 
 $t_n \to \infty$ and $(u_{0,n}, z_{0,n}) \in
  D(\tau -t_n, \theta_{-t_n} \omega )$ with
$D    \in  \cald$
and   $v_{0,n} = z_{0,n} -\varepsilon y(\omega) u_{0,n}$.
 \end{lem}

  \begin{proof}
  By  Lemma \ref{lest}
  with  $s=0$  we find that
  there exists $\varepsilon_1 \in (0,1)$
  such that for every $\varepsilon \in (0, \varepsilon_1]$, 
  the sequence 
    $( u(\tau, \tau -t_n,  \theta_{-\tau} \omega, u_{0,n}),
   v(\tau, \tau -t_n,  \theta_{-\tau}\omega, v_{0,n})  )$
   is bounded in $\hone \times \ltwo$.
   Therefore,  
  there exist
  $(\ut, \vt) \in \hone \times \ltwo$ 
  and a subsequence (not relabeled) such that
  \be
  \label{asyv_p1}
   ( u(\tau, \tau -t_n,  \theta_{-\tau} \omega, u_{0,n}),
   v(\tau, \tau -t_n,  \theta_{-\tau}\omega, v_{0,n})  )
  \rightharpoonup  (\ut, \vt) \  \mbox{in} \ \hone \times \ltwo,
  \ee
  which implies 
  \be
  \label{asyv_p2}
  \liminf_{n \to \infty} \| ( u(\tau, \tau -t_n,  \theta_{-\tau} \omega, u_{0,n}),
   v(\tau, \tau -t_n,  \theta_{-\tau}\omega, v_{0,n})  )\|_{H^1 \times L^2}
  \ge \| (\ut, \vt) \|_{H^1 \times L^2}.
  \ee
  The proof will be completed if we can show
   \be
  \label{asyv_p3}
  \limsup_{n \to \infty} \| ( u(\tau, \tau -t_n,  \theta_{-\tau} \omega, u_{0,n}),
   v(\tau, \tau -t_n,  \theta_{-\tau}\omega, v_{0,n})  )\|_{H^1 \times L^2}
  \le \| (\ut, \vt) \|_{H^1 \times L^2},
  \ee
because  \eqref{asyv_p1}-\eqref{asyv_p3} yield
  the strong convergence of 
  $( u(\tau, \tau -t_n,  \theta_{-\tau} \omega, u_{0,n}),
   v(\tau, \tau -t_n,  \theta_{-\tau}\omega, v_{0,n})  )$
   in $\hone \times \ltwo$.
   Next, we prove  
  \eqref{asyv_p3}  by the energy equation
  \eqref{ene}.
  By Lemma \ref{lest}, there exist
  $\varepsilon_2 \in (0, \varepsilon_1)$
  and $N_1=N_1(\tau, \omega, D)\ge 1$
  such that for all 
   $\varepsilon \in (0, \varepsilon_2]$,
   $n\ge N_1$ and 
   $s\in [-t_n, 0]$, 
$$
  \|    u(\tau +s, \tau -t_n, \theta_{-\tau}\omega, u_{0,n}) \|^2_\hone
  + \|  v(\tau +s, \tau -t_n, \theta_{-\tau}\omega, v_{0,n} ) \|^2
  $$
  $$
  + \|  u(\tau +s, \tau -t_n, \theta_{-\tau}\omega, u_{0,n} ) \|^{\gamma +1}
  _{L^{\gamma +1} (\R^n)}
  $$
  \be\label{asyv_p6}
   \le C_1   + 
 e^{ \int_{s}^{ 0}(2\sigma -\varepsilon c
-\varepsilon c | y(\theta_{r} \omega)|^2)dr }
R(\tau, \omega).
\ee
  Let $m$ be a positive integer and $N_2=N_2(\tau, \omega, D, m)
  \ge N_1$ such that $t_n \ge m$ for all $n \ge N_2$.
  By \eqref{asyv_p6}  we have,  for $n\ge N_2$,
  $$
  \|    u(\tau -m, \tau -t_n, \theta_{-\tau}\omega, u_{0,n}) \|^2_\hone
  + \|  v(\tau -m, \tau -t_n, \theta_{-\tau}\omega, v_{0,n} ) \|^2
  $$
  $$
 +  \|    u(\tau -m, \tau -t_n, \theta_{-\tau}\omega, u_{0,n}) \|^{\gamma+1}
  _{L^{\gamma +1} (\R^n)}
  $$
 \be\label{asyv_p7}
   \le C_1   + 
 e^{ \int_{-m}^{ 0}(2\sigma -\varepsilon c
-\varepsilon c | y(\theta_{r} \omega)|^2)dr }
R(\tau, \omega).
\ee
Thus, for every fixed $m$, the sequence
  $ \{(u(\tau -m, \tau -t_n, \theta_{-\tau}\omega, u_{0,n}), 
     v(\tau -m, \tau -t_n, \theta_{-\tau}\omega, v_{0,n}  ))\}_{n=1}^\infty$
     is bounded in $\hone \times \ltwo$, and hence, by
     a diagonal 
  process we infer that    there exist
  a subsequence (not relabeled) and
   $ (\ut_m, \vt_m) \in \hone \times \ltwo$
   for every $m \in \N$ such that,  as $n \to \infty$,
  \be
  \label{asyv_p9}
  (u(\tau-m,  \tau-t_n,  \theta_{-\tau} \omega, u_{0,n}),
    v(\tau-m , \tau-t_n, \theta_{-\tau}  \omega, v_{0,n}) )
   \rightharpoonup 
   (\ut_m, \vt_m )
  \ee
  in $ \hone \times \ltwo$.
Note that for $s \in [-m,0]$,
$$
 (u(\tau +s,  \tau-t_n,  \theta_{-\tau} \omega, u_{0,n}),
    v(\tau+s , \tau-t_n, \theta_{-\tau}  \omega, v_{0,n}) )
$$ 
\be\label{asyv_p11}
  = (u(\tau+s,  \tau-m,  \theta_{-\tau}
   \omega, u(\tau-m,  \tau-t_n,  \theta_{-\tau} \omega, u_{0,n})),
     v(\tau+s,  \tau-m,  \theta_{-\tau} \omega, 
     v(\tau-m,  \tau-t_n,  \theta_{-\tau} \omega, v_{0,n}))).
 \ee
 By \eqref{asyv_p11} 
 and Lemma \ref{weakc}  we find that,  for every  $m\in \N$,
 there exist a subsequence 
 $(t_{n^\prime},  u_{0,n^\prime}  ,  v_{0,n^\prime} )$
and a solution  
 $
 ( u_m(\cdot,  \tau-m,  \theta_{-\tau} \omega, \ut_m  ),
  v_m(\cdot,  \tau-m,  \theta_{-\tau} \omega, \vt_m  ) )
  $  of    \eqref{pde1}-\eqref{pde3} 
   with initial condition $(\ut_m. \vt_m)$
  such that
 for all  $s\in [-m,0]$,  when
  $n^\prime  \to \infty$,
 $$
   u(\tau +s,  \tau-t_ {n^\prime},  \theta_{-\tau} \omega, u_{0, {n^\prime}} )
     \rightharpoonup 
      u_m (\tau+s,  \tau-m,  \theta_{-\tau} \omega, \ut_m  )
   \  \mbox{in} \  \hone ,
 $$
  and
  $$
   v(\tau+s,  \tau-t_ {n^\prime},  \theta_{-\tau} \omega, v_{0, {n^\prime}} )
     \rightharpoonup 
      v_m (\tau+s,  \tau-m,  \theta_{-\tau} \omega, \vt_m  )
   \  \mbox{in} \  \ltwo .
$$
  Note that this subsequence  
  $(t_{n^\prime},  u_{0,n^\prime}  ,  v_{0,n^\prime} )$
  may depend on $m$. However, by an appropriate diagonal process, we
  can choose a common subsequence
  (which is denoted  again by $(t_{n},  u_{0,n}  ,  v_{0,n})$) such that for
  all $m\in \N$ and   $s\in [-m,0]$,  when
  $n \to \infty$,
   \be
  \label{asyv_p13}
   u(\tau +s,  \tau-t_ {n},  \theta_{-\tau} \omega, u_{0, {n}} )
     \rightharpoonup 
      u_m (\tau+s,  \tau-m,  \theta_{-\tau} \omega, \ut_m  )
   \  \mbox{in} \  \hone ,
  \ee
  and
  \be
  \label{asyv_p15}
   v(\tau+s,  \tau-t_ {n},  \theta_{-\tau} \omega, v_{0, {n}} )
     \rightharpoonup 
      v_m (\tau+s,  \tau-m,  \theta_{-\tau} \omega, \vt_m  )
   \  \mbox{in} \  \ltwo .
  \ee
  As a result of \eqref{asyv_p1} and \eqref{asyv_p13}-\eqref{asyv_p15}
  with $s=0$
  we get
  \be
  \label{asyv_p17}
  \ut = u_m(\tau,  \tau-m,  \theta_{-\tau} \omega, \ut_m)
  \quad \mbox{and}
  \quad
 \vt = v_m(\tau,  \tau-m, \theta_{-\tau}  \omega, \vt_m).
 \ee
 For
 $(u(\tau,  \tau-m,  \theta_{-\tau} \omega, 
 u(\tau-m,  \tau-t_n,  \theta_{-\tau} \omega, u_{0,n})),
     v(\tau,  \tau-m,  \theta_{-\tau} \omega, 
     v(\tau-m,  \tau-t_n,  \theta_{-\tau} \omega, v_{0,n})))
 $ with
 initial condition
 $ (u(\tau-m,  \tau-t_n,  \theta_{-\tau} \omega, u_{0,n}),
 v(\tau-m,  \tau-t_n,  \theta_{-\tau} \omega, u_{0,n}))$,
 by \eqref{asyv_p11} and the  energy equation
 \eqref{ene}   we get
$$
 E ( u(\tau, \tau -t_n, \theta_{-\tau}
 \omega, u_{0,n} ), v(\tau, \tau -t_n, \theta_{-\tau}  \omega, v_{0,n} ))
 $$
$$
=
e^{\int_\tau^{\tau -m}
(4\sigma -2\varepsilon |y(\theta_{r-\tau} \omega )| ) dr }
 E(u(\tau-m, \tau-t_n,  \theta_{-\tau}\omega, u_{0,n} ),
 v(\tau-m,\tau-t_n, \theta_{-\tau}\omega, v_{0,n}))
  $$
 $$
  + 
  \int_{\tau -m}^\tau 
e^{\int_\tau^{s}
(4\sigma -2\varepsilon |y(\theta_{r-\tau} \omega )| ) dr }
      G( u(s,\tau-m, \theta_{-\tau}\omega, 
    u(\tau-m,\tau-t_n, \theta_{-\tau}\omega, u_{0,n})),
    v ) ds
  $$
  $$
=
e^{\int_0^{ -m}
(4\sigma -2\varepsilon |y(\theta_{r} \omega )| ) dr }
 E(u(\tau-m, \tau-t_n,  \theta_{-\tau}\omega, u_{0,n} ),
 v(\tau-m,\tau-t_n, \theta_{-\tau}\omega, v_{0,n}))
  $$
 $$
  + 
  \int_{ -m}^0 
e^{\int_0^{s}
(4\sigma -2\varepsilon |y(\theta_{r} \omega )| ) dr }
      G( u(s+\tau,\tau-t_n, \theta_{-\tau}\omega, u_{0,n} ),
      v(s+\tau,\tau-t_n, \theta_{-\tau}\omega, v_{0,n} ) ) ds
  $$
  which along  with \eqref{enerH}  produces
  $$
 E ( u(\tau, \tau -t_n, \theta_{-\tau}
 \omega, u_{0,n} ), v(\tau, \tau -t_n, \theta_{-\tau}  \omega, v_{0,n} ))
 $$
  $$
=
e^{\int_0^{ -m}
(4\sigma -2\varepsilon |y(\theta_{r} \omega )| ) dr }
 E(u(\tau-m, \tau-t_n,  \theta_{-\tau}\omega, u_{0,n} ),
 v(\tau-m,\tau-t_n, \theta_{-\tau}\omega, v_{0,n}))
  $$
  $$
    - 2  \int_{-m}^0 
e^{\int_0^{s}
(4\sigma -2\varepsilon |y(\theta_{r} \omega )| ) dr }
      \left (
    \alpha -\delta-2\sigma + \varepsilon (|y(\theta_{s} \omega)|
    + y(\theta_{s} \omega) )
    \right )  \| v(s+\tau, \tau-t_n, \theta_{-\tau}\omega, v_{0,n} ) \|^2 ds
    $$
    $$
    -2 (\lambda +\delta^2 -\alpha \delta)
     \int_{-m}^0 
e^{\int_0^{s}
(4\sigma -2\varepsilon |y(\theta_{r} \omega )| ) dr }
     (\delta -2\sigma  + \varepsilon (|y(\theta_s\omega)|
    - y(\theta_s \omega)  ) )
     \| u(s+\tau, \tau-t_n, \theta_{-\tau}\omega, u_{0,n} )  \|^2
    $$
    $$
    -2
     \int_{-m}^0 
e^{\int_0^{s}
(4\sigma -2\varepsilon |y(\theta_{r} \omega )| ) dr }
      (
    \delta -2\sigma  + \varepsilon (|y(\theta_s \omega)|
    - y(\theta_s \omega)  )) \| \nabla  
    u(s+\tau, \tau-t_n, \theta_{-\tau}\omega, u_{0,n} ) \|^2ds
    $$
    $$
    +2 
     \int_{-m}^0 
e^{\int_0^{s}
(4\sigma -2\varepsilon |y(\theta_{r} \omega )| ) dr }
    (g(s+\tau),  v(s+\tau, \tau-t_n, \theta_{-\tau}\omega, v_{0,n} ))ds
    $$
    $$
    -2\varepsilon
     \int_{-m}^0 
e^{\int_0^{s}
(4\sigma -2\varepsilon |y(\theta_{r} \omega )| ) dr }
       (\varepsilon y(\theta_s \omega) -2 \delta)
    y(\theta_s \omega) ( u(s+\tau, \tau-t_n, \theta_{-\tau}\omega, u_{0,n} ), v) ds
    $$
    $$
    +2
    \int_{-m}^0 
e^{\int_0^{s}
(4\sigma -2\varepsilon |y(\theta_{r} \omega )| ) dr }
      (\varepsilon y(\theta_s \omega) -\delta ) (f(x,u(s+\tau)), u(s+\tau))ds
    $$
    \be\label{asyv_p20}
    + 4 
    \int_{-m}^0 
e^{\int_0^{s}
(4\sigma -2\varepsilon |y(\theta_{r} \omega )| ) dr }
    (2\sigma -\varepsilon  |y(\theta_s \omega)| )
    \ii F(x, u(s+\tau, \tau-t_n, \theta_{-\tau}\omega, u_{0,n} )) dxds.
\ee
 By \eqref{utai_p11} there exists $M_1 =M_1(\omega)\ge 1$  such 
 that for all $m\ge M_1$
 $$
 \int_{-m}^0  |y(\theta_{r} \omega )|  dr \le {\frac 1\alpha} m,
 $$
 and thus for $\varepsilon \in (0, {\frac 12} \alpha \sigma)$ we get
 \be\label{asyv_p22}
 \int^{-m}_0 (4\sigma -2\varepsilon |y(\theta_{r} \omega )| ) dr
 \le  -3\sigma m.
 \ee
 By \eqref{enerE} and \eqref{asyv_p22} we have
 for $m\ge M_1$   and $t_n \ge m$,
$$
 e^{\int_0^{ -m}
(4\sigma -2\varepsilon |y(\theta_{r} \omega )| ) dr }
 E(u(\tau-m, \tau-t_n,  \theta_{-\tau}\omega, u_{0,n} ),
 v(\tau-m,\tau-t_n, \theta_{-\tau}\omega, v_{0,n}))
 $$
 $$
 \le
  e^{-3 \sigma m}
   \left (\| v(\tau-m, \tau-t_n,  \theta_{-\tau}\omega, v_{0,n} )  \|^2
  + (\lambda + \delta^2 -\alpha \delta)
  \|  u(\tau-m, \tau-t_n,  \theta_{-\tau}\omega, u_{0,n} ) \|^2
  \right )
  $$
  $$
  +  e^{-3 \sigma m}  
  \|\nabla  u(\tau-m, \tau-t_n,  \theta_{-\tau}\omega, u_{0,n} )   \|^2
  $$
 \be\label{asyv_p24}
  +
  2 e^{\int_0^{ -m}
(4\sigma -2\varepsilon |y(\theta_{r} \omega )| ) dr }
  \int_{\R^3} F(x,   u(\tau-m, \tau-t_n,  \theta_{-\tau}\omega, u_{0,n} ) ) dx.
 \ee
 By \eqref{f1}-\eqref{f2} and \eqref{asyv_p7} we get
 for all  $n \ge N_2$,
  $$
    \ii F(x,   u(\tau-m, \tau-t_n,  \theta_{-\tau}\omega, u_{0,n} )   ) dx
   $$
   $$
   \le  C_3 
   \left (
    \|  u(\tau-m, \tau-t_n,  \theta_{-\tau}\omega, u_{0,n} )  \|^2
    + \|  u(\tau-m, \tau-t_n,  \theta_{-\tau}\omega, u_{0,n} )\|^{\gamma +1}_{L^{\gamma +1}}
    +1
   \right )
   $$ 
 \be\label{asyv_p25}
   \le C_4   + 
 C_4 e^{ \int_{-m}^{ 0}(2\sigma -\varepsilon c
-\varepsilon c | y(\theta_{r} \omega)|^2)dr }
 \le C_4   + 
 C_4 e^{ 2\sigma m }. 
\ee
 By \eqref{asyv_p7} and \eqref{asyv_p22}-\eqref{asyv_p25}
  we get
 for all $m \ge M_1$  and $n \ge N_2$,  
    $$
 e^{\int_0^{ -m}
(4\sigma -2\varepsilon |y(\theta_{r} \omega )| ) dr }
 E(u(\tau-m, \tau-t_n,  \theta_{-\tau}\omega, u_{0,n} ),
 v(\tau-m,\tau-t_n, \theta_{-\tau}\omega, v_{0,n}))
 $$
 \be\label{asyv_p27}
\le
 C_5 e^{-3 \sigma m} (1+  e^{2 \sigma m}
 )
 \le 
 2C_5 e^{-  \sigma m}.
 \ee
   Note that
   $  \alpha -\delta-2\sigma + \varepsilon (|y(\theta_{s} \omega)|
    + y(\theta_{s} \omega)  $
    is nonnegative for  all $s\in \R$ by
    \eqref{delta} and \eqref{kappa}.
    Therefore, by \eqref{asyv_p15} and Fatou\rq{}s Theorem we obtain
    the following inequality  for the second term on the right-hand side
    of \eqref{asyv_p20}:
     $$
   \liminf_{n\to \infty}
     \int_{-m}^0 
e^{\int_0^{s}
(4\sigma -2\varepsilon |y(\theta_{r} \omega )| ) dr }
      \left (
    \alpha -\delta-2\sigma + \varepsilon (|y(\theta_{s} \omega)|
    + y(\theta_{s} \omega) )
    \right )  \| v(s+\tau, \tau-t_n, \theta_{-\tau}\omega, v_{0,n} ) \|^2 ds
    $$
    $$
    \ge 
     \int_{-m}^0 
e^{\int_0^{s}
(4\sigma -2\varepsilon |y(\theta_{r} \omega )| ) dr }
      \left (
    \alpha -\delta-2\sigma + \varepsilon (|y(\theta_{s} \omega)|
    + y(\theta_{s} \omega) )
    \right ) 
   \liminf_{n\to \infty}
    \| v(s+\tau, \tau-t_n, \theta_{-\tau}\omega, v_{0,n} ) \|^2 ds
    $$
    \be\label{asyv_p30}
    \ge 
     \int_{-m}^0 
e^{\int_0^{s}
(4\sigma -2\varepsilon |y(\theta_{r} \omega )| ) dr }
      \left (
    \alpha -\delta-2\sigma + \varepsilon (|y(\theta_{s} \omega)|
    + y(\theta_{s} \omega) )
    \right )  
    \| v_m (s+\tau, \tau- m, \theta_{-\tau}\omega, \vt_m ) \|^2 ds.
  \ee
   Similarly, we can also prove
    $$
   \liminf_{n\to \infty}
     \int_{-m}^0 
e^{\int_0^{s}
(4\sigma -2\varepsilon |y(\theta_{r} \omega )| ) dr }
     (\delta -2\sigma  + \varepsilon (|y(\theta_s\omega)|
    - y(\theta_s \omega)  ) )
     \| u(s+\tau, \tau-t_n, \theta_{-\tau}\omega, u_{0,n} )  \|^2
    $$
    \be\label{asyv_p31}
    \ge
      \int_{-m}^0 
e^{\int_0^{s}
(4\sigma -2\varepsilon |y(\theta_{r} \omega )| ) dr }
     (\delta -2\sigma  + \varepsilon (|y(\theta_s\omega)|
    - y(\theta_s \omega)  ) )
     \| u_m (s+\tau, \tau-m, \theta_{-\tau}\omega, \ut_m )  \|^2
   \ee
   and
    $$
    \liminf_{n\to \infty}
      \int_{-m}^0 
e^{\int_0^{s}
(4\sigma -2\varepsilon |y(\theta_{r} \omega )| ) dr }
      (
    \delta -2\sigma  + \varepsilon (|y(\theta_s \omega)|
    - y(\theta_s \omega)  )) \| \nabla  
    u(s+\tau, \tau-t_n, \theta_{-\tau}\omega, u_{0,n} ) \|^2 
    $$
    \be\label{asyv_p32}
    \ge
       \int_{-m}^0 
e^{\int_0^{s}
(4\sigma -2\varepsilon |y(\theta_{r} \omega )| ) dr }
      (
    \delta -2\sigma  + \varepsilon (|y(\theta_s \omega)|
    - y(\theta_s \omega)  )) \| \nabla  
    u_m (s+\tau, \tau-m, \theta_{-\tau}\omega, \ut_m ) \|^2.
    \ee
   We now deal with the nonlinear terms
   in \eqref{asyv_p20},  for which we first prove the
   following convergence in $\ltwo$:  for every $s\in [-m, 0]$,
  \be\label{asyv_p34}
  \lim_{n\to \infty}
  u(\tau+s, \tau-t_n, \theta_{-\tau}\omega, u_{0,n} )
  =  u _m (\tau+s, \tau -m, \theta_{-\tau}\omega, \ut_m )
  \ \mbox{in} \ \ltwo.
   \ee
   Given $\eta>0$, by Lemma \ref{utai} we find that
   there exist $\varepsilon_3 \in (0, \varepsilon_2)$,
   $N_3=N_3(\tau, \omega, D, \eta,m)\ge 1$ and $K
   =K(\tau, \omega, \eta,m) \ge 1$ such that
   for all $\varepsilon\in (0, \varepsilon_3)$,
   $s\in [-m, 0]$  and $n\ge N_3$,
  \be\label{asyv_p35}
   \int_{|x| \ge K}
   | u(s+\tau, \tau-t_n, \theta_{-\tau}\omega, u_{0,n} ) | ^2dx
   \le  \eta.
   \ee
  Thus  \eqref{asyv_p34} follows from 
     \eqref{asyv_p13} and
   the compact   embedding $H^1  \hookrightarrow
   L^2$ in bounded domains.
   By \eqref{f1}-\eqref{f2},
   \eqref{asyv_p6}, \eqref{asyv_p34}
   and the arguments of \eqref{cprv_p62}
   and \eqref{cprv_p68} we obtain
    \be\label{asyv_p37}
   \lim_{n\to \infty}
    \int_{-m}^0 
e^{\int_0^{s}
(4\sigma -2\varepsilon |y(\theta_{r} \omega )| ) dr }
    (2\sigma -\varepsilon  |y(\theta_s \omega)| )
    \ii F(x, u(s+\tau, \tau-t_n, \theta_{-\tau}\omega, u_{0,n} )) dxds
    $$
    $$
    = 
       \int_{-m}^0 
e^{\int_0^{s}
(4\sigma -2\varepsilon |y(\theta_{r} \omega )| ) dr }
    (2\sigma -\varepsilon  |y(\theta_s \omega)| )
    \ii F(x, u_m(s+\tau, \tau-m, \theta_{-\tau}\omega, \ut_m )) dxds
\ee
   and 
   \be\label{asyv_p38}
   \lim_{n\to \infty}
    \int_{-m}^0 
e^{\int_0^{s}
(4\sigma -2\varepsilon |y(\theta_{r} \omega )| ) dr }
      (\varepsilon y(\theta_s \omega) -\delta ) (f(x,u(s+\tau)), u(s+\tau,
       \tau-t_n, \theta_{-\tau}\omega, u_{0,n} ))ds
    $$
    $$
      =
      \int_{-m}^0 
e^{\int_0^{s}
(4\sigma -2\varepsilon |y(\theta_{r} \omega )| ) dr }
      (\varepsilon y(\theta_s \omega) -\delta ) (f(x,u_m(s+\tau)), u_m(s+\tau,
       \tau-m, \theta_{-\tau}\omega, \ut_m ))ds.
      \ee
      By  \eqref{asyv_p6}, \eqref{asyv_p34}
      and the Lebesgue dominated convergence
      theorem, we also obtain    
     $$ 
    \lim_{n\to \infty}
     \int_{-m}^0 
e^{\int_0^{s}
(4\sigma -2\varepsilon |y(\theta_{r} \omega )| ) dr }
    (g(s+\tau),  v(s+\tau, \tau-t_n, \theta_{-\tau}\omega, v_{0,n} ))ds
    $$
     \be\label{asyv_p40}
     = \int_{-m}^0 
e^{\int_0^{s}
(4\sigma -2\varepsilon |y(\theta_{r} \omega )| ) dr }
    (g(s+\tau),  v_m(s+\tau, \tau-m, \theta_{-\tau}\omega, \vt_m ))ds
 \ee
 and
    $$
  \lim_{n\to \infty}
     \int_{-m}^0 
e^{\int_0^{s}
(4\sigma -2\varepsilon |y(\theta_{r} \omega )| ) dr }
       (\varepsilon y(\theta_s \omega) -2 \delta)
    y(\theta_s \omega) ( u(s+\tau, \tau-t_n, \theta_{-\tau}\omega, u_{0,n} ), v) ds
    $$  
    \be\label{asyv_p41}
    =   \int_{-m}^0 
e^{\int_0^{s}
(4\sigma -2\varepsilon |y(\theta_{r} \omega )| ) dr }
       (\varepsilon y(\theta_s \omega) -2 \delta)
    y(\theta_s \omega) ( u_m(s+\tau, \tau-m, \theta_{-\tau}\omega, \ut_m, v_m) ds.
    \ee
    It follows from \eqref{asyv_p20}, \eqref{asyv_p30}-\eqref{asyv_p32}
    and \eqref{asyv_p37}-\eqref{asyv_p41}
    that for all $m \ge M_1$,
      $$
\limsup_{n\to \infty}
 E ( u(\tau, \tau -t_n, \theta_{-\tau}
 \omega, u_{0,n} ), v(\tau, \tau -t_n, \theta_{-\tau}  \omega, v_{0,n} ))
 $$
  $$
\le 2C_2 
e^{-\sigma m}  
    - 2  \int_{-m}^0 
e^{\int_0^{s}
(4\sigma -2\varepsilon |y(\theta_{r} \omega )| ) dr }
      \left (
    \alpha -\delta-2\sigma + \varepsilon (|y(\theta_{s} \omega)|
    + y(\theta_{s} \omega) )
    \right )  \| v_m(s+\tau, \tau-m, \theta_{-\tau}\omega, \vt_m  ) \|^2 ds
    $$
    $$
    -2 (\lambda +\delta^2 -\alpha \delta)
     \int_{-m}^0 
e^{\int_0^{s}
(4\sigma -2\varepsilon |y(\theta_{r} \omega )| ) dr }
     (\delta -2\sigma  + \varepsilon (|y(\theta_s\omega)|
    - y(\theta_s \omega)  ) )
     \| u_m(s+\tau, \tau-m, \theta_{-\tau}\omega, \ut_m )  \|^2
    $$
    $$
    -2
     \int_{-m}^0 
e^{\int_0^{s}
(4\sigma -2\varepsilon |y(\theta_{r} \omega )| ) dr }
      (
    \delta -2\sigma  + \varepsilon (|y(\theta_s \omega)|
    - y(\theta_s \omega)  )) \| \nabla  
    u_m(s+\tau, \tau-m, \theta_{-\tau}\omega,\ut_m ) \|^2ds
    $$
    $$
    +2 
     \int_{-m}^0 
e^{\int_0^{s}
(4\sigma -2\varepsilon |y(\theta_{r} \omega )| ) dr }
    (g(s+\tau),  v_m(s+\tau, \tau-m, \theta_{-\tau}\omega, \vt_m ))ds
    $$
    $$
    -2\varepsilon
     \int_{-m}^0 
e^{\int_0^{s}
(4\sigma -2\varepsilon |y(\theta_{r} \omega )| ) dr }
       (\varepsilon y(\theta_s \omega) -2 \delta)
    y(\theta_s \omega) ( u_m(s+\tau, \tau-m, \theta_{-\tau}\omega, \ut_m ), v_m) ds
    $$
    $$
    +2
    \int_{-m}^0 
e^{\int_0^{s}
(4\sigma -2\varepsilon |y(\theta_{r} \omega )| ) dr }
      (\varepsilon y(\theta_s \omega) -\delta ) (f(x,u_m(s+\tau)),
       u_m(s+\tau,\tau-m, \theta_{-\tau}\omega,
      \ut_m  ) )ds
    $$
    \be\label{asyv_p44}
    + 4 
    \int_{-m}^0 
e^{\int_0^{s}
(4\sigma -2\varepsilon |y(\theta_{r} \omega )| ) dr }
    (2\sigma -\varepsilon  |y(\theta_s \omega)| )
    \ii F(x, u_m(s+\tau, \tau-m, \theta_{-\tau}\omega,\ut_m )) dxds.
\ee
   On the other hand,    by \eqref{ene} and \eqref{asyv_p17}
   we have
      $$ 
      E(\ut, \vt) = 
 E ( u_m(\tau, \tau -m, \theta_{-\tau}
 \omega, \ut_m  ), v_m(\tau, \tau -m, \theta_{-\tau}  \omega, \vt_m ))
 $$
  $$
= e^{  \int_0^{-m} 
(4\sigma -2\varepsilon |y(\theta_{r} \omega )| ) dr } E(\ut_m, \vt_m)
$$
$$
   - 2  \int_{-m}^0 
e^{\int_0^{s}
(4\sigma -2\varepsilon |y(\theta_{r} \omega )| ) dr }
      \left (
    \alpha -\delta-2\sigma + \varepsilon (|y(\theta_{s} \omega)|
    + y(\theta_{s} \omega) )
    \right )  \| v_m(s+\tau, \tau-m, \theta_{-\tau}\omega, \vt_m  ) \|^2 ds
    $$
    $$
    -2 (\lambda +\delta^2 -\alpha \delta)
     \int_{-m}^0 
e^{\int_0^{s}
(4\sigma -2\varepsilon |y(\theta_{r} \omega )| ) dr }
     (\delta -2\sigma  + \varepsilon (|y(\theta_s\omega)|
    - y(\theta_s \omega)  ) )
     \| u_m(s+\tau, \tau-m, \theta_{-\tau}\omega, \ut_m )  \|^2
    $$
    $$
    -2
     \int_{-m}^0 
e^{\int_0^{s}
(4\sigma -2\varepsilon |y(\theta_{r} \omega )| ) dr }
      (
    \delta -2\sigma  + \varepsilon (|y(\theta_s \omega)|
    - y(\theta_s \omega)  )) \| \nabla  
    u_m(s+\tau, \tau-m, \theta_{-\tau}\omega,\ut_m ) \|^2ds
    $$
    $$
    +2 
     \int_{-m}^0 
e^{\int_0^{s}
(4\sigma -2\varepsilon |y(\theta_{r} \omega )| ) dr }
    (g(s+\tau),  v_m(s+\tau, \tau-m, \theta_{-\tau}\omega, \vt_m ))ds
    $$
    $$
    -2\varepsilon
     \int_{-m}^0 
e^{\int_0^{s}
(4\sigma -2\varepsilon |y(\theta_{r} \omega )| ) dr }
       (\varepsilon y(\theta_s \omega) -2 \delta)
    y(\theta_s \omega) ( u_m(s+\tau, \tau-m, \theta_{-\tau}\omega, \ut_m ), v_m) ds
    $$
    $$
    +2
    \int_{-m}^0 
e^{\int_0^{s}
(4\sigma -2\varepsilon |y(\theta_{r} \omega )| ) dr }
      (\varepsilon y(\theta_s \omega) -\delta ) (f(x,u_m(s+\tau)),
       u_m(s+\tau,\tau-m, \theta_{-\tau}\omega,
      \ut_m  ) )ds
    $$
    \be\label{asyv_p46}
    + 4 
    \int_{-m}^0 
e^{\int_0^{s}
(4\sigma -2\varepsilon |y(\theta_{r} \omega )| ) dr }
    (2\sigma -\varepsilon  |y(\theta_s \omega)| )
    \ii F(x, u_m(s+\tau, \tau-m, \theta_{-\tau}\omega,\ut_m )) dxds.
\ee
By \eqref{asyv_p44}-\eqref{asyv_p46},
\eqref{enerE}, \eqref{f3} and \eqref{asyv_p22}  we get
  for all $m \ge M_1$,
      $$
\limsup_{n\to \infty}
 E ( u(\tau, \tau -t_n, \theta_{-\tau}
 \omega, u_{0,n} ), v(\tau, \tau -t_n, \theta_{-\tau}  \omega, v_{0,n} ))
 $$
 \be\label{asyv_p50}
\le 2C_2 
e^{-\sigma m}   
-    e^{  \int_0^{-m} 
(4\sigma -2\varepsilon |y(\theta_{r} \omega )| ) dr } E(\ut_m, \vt_m)
     +  E(\ut, \vt)  
     $$
     $$
     \le 2C_2 
e^{-\sigma m}   
-   2 e^{  \int_0^{-m} 
(4\sigma -2\varepsilon |y(\theta_{r} \omega )| ) dr }
\ii F(x, \ut_m) dx   
     +  E(\ut, \vt)  
     $$
     $$
      \le 2C_2 
e^{-\sigma m}   
+  2 e^{  \int_0^{-m} 
(4\sigma -2\varepsilon |y(\theta_{r} \omega )| ) dr }
\| \phi_3\|_{L^1(\R^n)}  
     +  E(\ut, \vt)  
      $$
     $$
      \le 2C_2 
e^{-\sigma m}   
+  2 e^{   -3\sigma m} 
\| \phi_3\|_{L^1(\R^n)}  
     +  E(\ut, \vt)  .
 \ee
  By \eqref{asyv_p34} with $s=0$ and the arguments
  of \eqref{cprv_p58} we infer that   
  $$
  \lim_{n\to \infty}
  \ii F(x,  u(\tau, \tau -t_n, \theta_{-\tau}
 \omega, u_{0,n} )) dx
 = \ii F(x, u_m(\tau, \tau -m, \theta_{-\tau}
 \omega, \ut_m )dx
$$
 which together with   \eqref{asyv_p17}   yields
 \be\label{asyv_p52}
  \lim_{n\to \infty}
  \ii F(x,  u(\tau, \tau -t_n, \theta_{-\tau}
 \omega, u_{0,n} )) dx
  =\ii F(x, \ut)dx.
 \ee
 By \eqref{enerE} and \eqref{asyv_p50}-\eqref{asyv_p52} we get 
  $$
   \limsup_{n\to \infty}
  ( \|v(\tau, \tau -t_n, \theta_{-\tau}  \omega, v_{0,n} )\|^2
     + (\lambda + \delta^2 -\alpha \delta) 
     \| u(\tau, \tau -t_n, \theta_{-\tau}  \omega, u_{0,n} )\|^2
    +\| \nabla u  (\tau, \tau -t_n, \theta_{-\tau}  \omega,  _{0,n} ) \|^2 )
     $$
     \be\label{asyv_p60}
     \le 
      \| \vt \|^2
     + (\lambda + \delta^2 -\alpha \delta) 
     \|  \ut \|^2
    +\| \nabla \ut  \|^2 
    + 2C_2 
e^{-\sigma m}   
+  2 e^{   -3\sigma m} 
\| \phi_3\|_{L^1(\R^n)}  .
   \ee
   Taking the limit of \eqref{asyv_p60} as 
   $m \to \infty$ we get
   $$
   \limsup_{n\to \infty}
  ( \|v(\tau, \tau -t_n, \theta_{-\tau}  \omega, v_{0,n} )\|^2
     + (\lambda + \delta^2 -\alpha \delta) 
     \| u(\tau, \tau -t_n, \theta_{-\tau}  \omega, u_{0,n} )\|^2
    +\| \nabla u  (\tau, \tau -t_n, \theta_{-\tau}  \omega,  _{0,n} ) \|^2 )
     $$
     \be\label{asyv_p62}
     \le 
      \| \vt \|^2
     + (\lambda + \delta^2 -\alpha \delta) 
     \|  \ut \|^2
    +\| \nabla \ut  \|^2   .
   \ee  Note that
    $\left ( \| v \|^2
     + (\lambda + \delta^2 -\alpha \delta) 
     \|  u \|^2
    +\| \nabla u  \|^2
    \right )^{\frac 12} $
    is an equivalent norm  for $(u,v) \in \hone \times \ltwo$.
    Therefore,
    from \eqref{asyv_p62} we obtain
  $$
     \limsup_{n\to \infty}
  (  \| u(\tau, \tau -t_n, \theta_{-\tau}  \omega, u_{0,n} )\|^2_{H^1}
  +
  \|v(\tau, \tau -t_n, \theta_{-\tau}  \omega, v_{0,n} )\|^2  )
     \le 
    \| \ut\|_{H^1}^2 +   \| \vt \|^2.
   $$     
     This  yields \eqref{asyv_p3} and thus completes the proof.
    \end{proof}

\section{Random attractors for wave  equations}
\setcounter{equation}{0}

This  section is devoted to  existence
and uniqueness  of 
$\cald$-pullback attractors
for   system
\eqref{spde1}-\eqref{spde3}.
We first recall the existence of $\cald$-pullback absorbing sets
from \cite{wan9}.

\begin{lem}
\label{lem51}
Under conditions 
 \eqref{f1}-\eqref{f3}  and \eqref{g1},
   there exists  
    $\varepsilon_0 = \varepsilon_0 (\alpha, \lambda, f)\in
    (0, 1)$
  such that 
 for all  $\varepsilon \in (0, \varepsilon_0]$, 
 the multivalued non-autonomous 
  cocycle $\Phi$ has a closed  
 $\cald$-pullback absorbing set
 $K  \in \cald$:
 \be\label{lem51_1}
 K (\tau, \omega)
 =\{ (u,z) \in \hone \times \ltwo:
 \|u\|_\hone^2 + \| z\|^2 \le L(\tau, \omega) \}
 \ee
 for   $\tau \in \R$
 and $\omega\in \Omega$,
where  $L(\tau, \omega)$ is given by 
\be\label{lem51_2}
L(\tau, \omega)
 =   
 M(1+ \varepsilon y^2 (\omega) )
 \left (1 + 
    \int^0_{ - \infty} 
   e^{ \int_0^{s}(2\sigma -\varepsilon c
-\varepsilon c | y(\theta_{r} \omega)|^2)dr }
\left (
 1+ \| g(s+\tau,  \cdot) \|^2 + \varepsilon|y(\theta_{s} \omega)|^2 
\right ) ds \right )
\ee
with    $M$ and $c$  being 
   positive     constants  independent of $\tau$, $\omega$, $D$ 
  and $\varepsilon$.
  \end{lem}
  
  \begin{proof}
  This lemma can be derived from \eqref{ener1} and 
  the details can be found in  \cite{wan9}.
  \end{proof}

\begin{lem}
\label{lem52}
Suppose 
 \eqref{f1}-\eqref{f3} and \eqref{g1}
  hold and $K$ is the closed
 $\cald$-pullback absorbing set given by
 \eqref{lem51_1}-\eqref{lem51_2}.
 Then
  there exists  
    $\varepsilon_0 = \varepsilon_0 (\alpha, \lambda, f)\in
    (0, 1)$
  such that 
 for all  $\varepsilon \in (0, \varepsilon_0]$, 
    the mapping
 $\Phi (t, \tau, \cdot, K(\tau, \cdot)):
 \Omega_m \to 2^{H^1\times L^2}$
 is weakly upper semicontinuous for all
 $m \in \N$, $t\in \R^+$  and $\tau \in \R$.
 \end{lem}

\begin{proof}
Let $\omega_n \to \omega$ in $\Omega_m$ and
$({\widetilde{u}}_n,  {\widetilde{z}}_n)
\in \Phi (t, \tau, \omega_n,
K(\tau, \omega_n))$. We need to find
$({\widetilde{u}},  {\widetilde{z}})
\in \Phi (t, \tau, \omega, K(\tau, \omega))$
such that, up to  a subsequence,
\be\label{lem52p1}
({\widetilde{u}}_n, {\widetilde{z}}_n)
\rightharpoonup
({\widetilde{u}}, {\widetilde{z}})
\ \mbox{in } \hone \times \ltwo.
\ee
Since 
$({\widetilde{u}}_n,  {\widetilde{z}}_n)
\in \Phi (t, \tau, \omega_n,
K(\tau, \omega_n))$,
for every $n\in \N$, 
there  exist
$(u_{0,n}, z_{0,n})
\in K(\tau, \omega_n) $
and  
$(u_n(\cdot, \tau, \theta_{-\tau} \omega_n, u_{0,n}),
z_n(\cdot, \tau, \theta_{-\tau}\omega_n, z_{0,n}))$
   with
initial condition $(u_{0,n}, z_{0,n})$ at $\tau$
such that
\be\label{lem52p2}
({\widetilde{u}}_n,  {\widetilde{z}}_n)
= (u_n( t+\tau, \tau, \theta_{-\tau} \omega_n, u_{0,n}),
z_n(t+\tau, \tau, \theta_{-\tau}\omega_n, z_{0,n})) .
\ee
By \eqref{rdsw2} we have
\be\label{lem52p3}
z_n(t, \tau, \theta_{-\tau}\omega_n, z_{0,n})
=
v_n(t, \tau, \theta_{-\tau}\omega_n, v_{0,n})
+\varepsilon y(\theta_{t-\tau} \omega_n)
u_n(t, \tau, \theta_{-\tau}\omega_n, u_{0,n})
\ee
with
\be\label{lem52p3a}
v_{0,n} =z_{0,n}
-\varepsilon y(\omega_n) u_{0,n}.
\ee
Since $(u_{0,n}, z_{0,n})
\in K(\tau, \omega_n) $,  by \eqref{lem51_1}-\eqref{lem51_2} we get
for all $n\in \N$,
$$
\|u_{0,n}\|_{H^1}^2 +  \| z_{0,n}\|^2
$$
\be\label{lem52p4}
\le
 M(1+ \varepsilon y^2 (\omega_n) )
 \left (1 + 
    \int^0_{ - \infty} 
   e^{ \int_0^{s}(2\sigma -\varepsilon c
-\varepsilon c | y(\theta_{r} \omega_n)|^2)dr }
\left (
 1+ \| g(s+\tau,  \cdot) \|^2 + \varepsilon|y(\theta_{s} \omega_n)|^2 
\right ) ds \right ).
\ee
Since $\omega_n \in \Omega_m$, by \eqref{defom} we find that
for all $n \in \N$ and $s\le -m$,
\be\label{lem52p6}
\int_s^0 | y(\theta_{r} \omega_n)|^2dr \le -{\frac s\alpha}.
\ee
Let $\varepsilon_1 = {\frac {\alpha \sigma}{c(1+ \alpha)}}$.
By \eqref{lem52p6} and \eqref{omegam1}
we get for all $\varepsilon \in (0, \varepsilon_1)$,
$n \in \N$ and $s\le -m$,
\be\label{lem52p7}
  e^{ \int_0^{s}(2\sigma -\varepsilon c
-\varepsilon c | y(\theta_{r} \omega_n)|^2)dr }
\le e^{\sigma s}
\quad \mbox{and} \ 
|y(\theta_{s} \omega_n)|^2 
\le 4 s^2 +C_1.
\ee
By \eqref{lem52p7}, Lemma \ref{omegam} (i) and the dominated
convergence theorem we obtain
$$\lim_{n \to \infty}
\int_{-\infty}^{-m}
 e^{ \int_0^{s}(2\sigma -\varepsilon c
-\varepsilon c | y(\theta_{r} \omega_n)|^2)dr }
\left (
 1+ \| g(s+\tau,  \cdot) \|^2 + \varepsilon|y(\theta_{s} \omega_n)|^2 
\right ) ds
$$
$$
 =\int_{-\infty}^{-m}
  e^{ \int_0^{s}(2\sigma -\varepsilon c
-\varepsilon c | y(\theta_{r} \omega)|^2)dr }
\left (
 1+ \| g(s+\tau,  \cdot) \|^2 + \varepsilon|y(\theta_{s} \omega)|^2 
\right ) ds
$$
which along with Lemma \ref{omegam} (i) again implies
$$\lim_{n \to \infty}
\int_{-\infty}^{0}
 e^{ \int_0^{s}(2\sigma -\varepsilon c
-\varepsilon c | y(\theta_{r} \omega_n)|^2)dr }
\left (
 1+ \| g(s+\tau,  \cdot) \|^2 + \varepsilon|y(\theta_{s} \omega_n)|^2 
\right ) ds
$$
\be\label{lem52p9}
 =\int_{-\infty}^0
  e^{ \int_0^{s}(2\sigma -\varepsilon c
-\varepsilon c | y(\theta_{r} \omega)|^2)dr }
\left (
 1+ \| g(s+\tau,  \cdot) \|^2 + \varepsilon|y(\theta_{s} \omega)|^2 
\right ) ds.
\ee
By \eqref{lem52p9} and Lemma \ref{omegam} (i), we get from
\eqref{lem52p4}
\be\label{lem52p11}
\limsup_{n\to \infty}
(\|u_{0,n}\|_{H^1}^2 +  \| z_{0,n}\|^2)
$$
$$
\le
 M(1+ \varepsilon y^2 (\omega) )
 \left (1 + 
    \int^0_{ - \infty} 
   e^{ \int_0^{s}(2\sigma -\varepsilon c
-\varepsilon c | y(\theta_{r} \omega)|^2)dr }
\left (
 1+ \| g(s+\tau,  \cdot) \|^2 + \varepsilon|y(\theta_{s} \omega)|^2 
\right ) ds \right ).
\ee
This indicates  the  sequence $\{(u_{0,n}, z_{0,n})\}_{n=1}^\infty$ 
is bounded in $\hone \times \ltwo$ and hence there exists
$(u_0, z_0) \in \hone \times \ltwo$ such that, up to a subsequence,
\be\label{lem52p12}
 (u_{0,n}, z_{0,n}) \rightharpoonup (u_0,z_0)
 \ \mbox{ in  } \   \hone \times \ltwo.
 \ee
 By \eqref{lem52p11}-\eqref{lem52p12} we get
 \be\label{lem52p14}
 \|u_{0}\|_{H^1}^2 +  \| z_{0}\|^2
$$
$$
\le
 M(1+ \varepsilon y^2 (\omega) )
 \left (1 + 
    \int^0_{ - \infty} 
   e^{ \int_0^{s}(2\sigma -\varepsilon c
-\varepsilon c | y(\theta_{r} \omega)|^2)dr }
\left (
 1+ \| g(s+\tau,  \cdot) \|^2 + \varepsilon|y(\theta_{s} \omega)|^2 
\right ) ds \right ).
\ee
By \eqref{lem52p3} with $t=\tau$,
\eqref{lem52p3a},   \eqref{lem52p12}
 and Lemma \ref{omegam} we have
\be\label{lem52p13}
 (u_{0,n}, v_{0,n}) \rightharpoonup (u_0,v_0)
 \ \mbox{ in } \  \hone \times \ltwo
 \ \mbox{with} \  
 v_0 = z_0 -\varepsilon y(\omega) u_0.
 \ee
 Using \eqref{lem52p13}, following the proof
 of Lemma \ref{cprv} (i), we infer that
 system \eqref{pde1}-\eqref{pde3} has a solution
 $(u,v) = (u(\cdot, \tau, \theta_{-\tau} \omega, u_0), 
  v(\cdot, \tau,  \theta_{-\tau} \omega, v_0))$
  with initial condition $(u_0,v_0)$ at $\tau$ 
 such that, up to a  subsequence,  for all $r\ge \tau$,
  $$(u_n( r, \tau, \theta_{-\tau} \omega_n, u_{0,n}),
v_n(r, \tau, \theta_{-\tau}\omega_n, v_{0,n}) )
\rightharpoonup
(u(r, \tau, \theta_{-\tau} \omega, u_0), 
  v(r, \tau,  \theta_{-\tau} \omega, v_0))
  $$
  in $\hone \times \ltwo$, which together with
  \eqref{lem52p3}-\eqref{lem52p3a}  and Lemma \ref{omegam}
   implies for all $r\ge \tau$,
   \be\label{lem52p20} 
  (u_n( r, \tau, \theta_{-\tau} \omega_n, u_{0,n}),
z_n(r, \tau, \theta_{-\tau}\omega_n, z_{0,n}))  
\rightharpoonup
(u(r, \tau, \theta_{-\tau} \omega, u_0), 
  z(r, \tau,  \theta_{-\tau} \omega, z_0))
\ee
in  $\hone \times \ltwo$ where
$$
z(r, \tau,  \theta_{-\tau} \omega, z_0)
=  
v (r, \tau, \theta_{-\tau}\omega , v_{0})
+\varepsilon y(\theta_{r-\tau} \omega)
u(r, \tau, \theta_{-\tau}\omega_n, u_{0})
$$
with $
v_{0} =z_{0}
-\varepsilon y(\omega) u_{0} $.
By \eqref{lem52p2} and \eqref{lem52p20} we get
for $t\in \R^+$,
\be\label{lem52p30}
({\widetilde{u}}_n , {\widetilde{z}}_n )  
\rightharpoonup
(u(t+\tau, \tau, \theta_{-\tau} \omega, u_0), 
  z(t+\tau, \tau,  \theta_{-\tau} \omega, z_0))
  \ \mbox{ in } \ 
  \hone \times \ltwo.
 \ee
  By \eqref{lem51_1}-\eqref{lem51_2}  and \eqref{lem52p14} we
  have
  $(u_0,v_0) \in K(\tau, \omega)$,
  which along with \eqref{lem52p30}
  implies \eqref{lem52p1}
  with 
  $\widetilde{u}
  = u(t+\tau, \tau, \theta_{-\tau} \omega, u_0)$
  and 
  $\widetilde{z}
  = z(t+\tau, \tau, \theta_{-\tau} \omega, z_0)$,
  and thus completes the proof.
\end{proof}

Next, we prove the   $\cald$-pullback
asymptotic compactness of $\Phi$.

\begin{lem}
\label{lem53}
Under conditions 
 \eqref{f1}-\eqref{f3}  and \eqref{g1},
  there exists   
    $\varepsilon_0 = \varepsilon_0 (\alpha, \lambda, f)
    \in (0,1)$
       such that 
 for all  $\varepsilon \in (0, \varepsilon_0]$,
   $\Phi$ is $\cald$-pullback asymptotically compact
 in $\hone \times \ltwo$.  
 \end{lem}
 
 \begin{proof}
 Given $\tau \in \R$,
 $\omega \in \Omega$,
  $D \in \cald$
  and 
$({\widetilde{u}}_n,  {\widetilde{z}}_n)
\in \Phi (t_n, \tau-t_n, \theta_{-t_n} \omega,
D(\tau-t_n,  \theta_{-t_n} \omega))$
with $t_n \to \infty$.
 We want  to prove $({\widetilde{u}}_n,  {\widetilde{z}}_n)$
 has a convergent subsequence.
 By assumption,    
there  exist
$(u_{0,n}, z_{0,n})
\in D(\tau-t_n,  \theta_{-t_n}\omega) $
and  
$(u_n(\cdot, \tau-t_n, \theta_{-\tau} \omega, u_{0,n}),
z_n(\cdot, \tau- t_n, \theta_{-\tau}\omega, z_{0,n}))$
   with
initial condition $(u_{0,n}, z_{0,n})$ at $\tau - t_n$
such that
\be\label{lem53p2}
({\widetilde{u}}_n,  {\widetilde{z}}_n)
=(u_n(\tau, \tau-t_n, \theta_{-\tau} \omega, u_{0,n}),
z_n(\tau, \tau- t_n, \theta_{-\tau}\omega, z_{0,n})).
\ee
By \eqref{shift} and \eqref{lem53p2} we have
\be\label{lem53p4}
({\widetilde{u}}_n,  {\widetilde{z}}_n)
=(u_n(\tau, \tau-t_n, \theta_{-\tau} \omega, u_{0,n}),
v_n(\tau, \tau- t_n, \theta_{-\tau}\omega, v_{0,n})
+\varepsilon y(\omega)
u_n(\tau, \tau- t_n, \theta_{-\tau}\omega, u_{0,n})).
\ee
By Lemma \ref{asyv},
the sequence $\{
(u_n(\tau, \tau-t_n, \theta_{-\tau} \omega, u_{0,n}),
v_n(\tau, \tau- t_n, \theta_{-\tau}\omega, v_{0,n}) )\}_{n=1}^\infty$
has a convergent subsequence in $\hone \times \ltwo$,
and so is  the sequence $\{({\widetilde{u}}_n,  {\widetilde{z}}_n ) \}_{n=1}^\infty$
by \eqref{lem53p4}. This completes  the proof.
    \end{proof}

Finally,  we prove the existence and uniqueness of 
  random attractors of $\Phi$.

 \begin{thm}
Suppose 
 \eqref{f1}-\eqref{f3}  and \eqref{g1}
  hold. 
   Then there exists   
    $\varepsilon_0 = \varepsilon_0 (\alpha, \lambda, f)
    \in (0,1)$
       such that 
 for all  $\varepsilon \in (0, \varepsilon_0]$,
  system  \eqref{spde1}-\eqref{spde3}
  has    a unique $\cald$-pullback attractor 
  $\cala  \in \cald$
 in $\hone \times \ltwo$.
 Furthermore, if $g: \R \to \ltwo$ is $T$-periodic for some $T>0$,
 then  so is  the attractor $\cala$; that is,
 $\cala(\tau+T, \omega) =\cala(\tau, \omega)$
 for all $\tau \in \R$  and $\omega \in \Omega$.
\end{thm}

\begin{proof}
We have proved that $\Phi$ satisfies
all conditions (i)-(iv) of Theorem \ref{att}; more precisely,
$\Phi$ is $\cald$-pullback asymptotically compact
in $\hone \times \ltwo$ by Lemma \ref{lem53},
$\Phi(t,\tau, \omega, \cdot): \hone \times \ltwo
\to 2^{\hone \times \ltwo}$
is upper semicontinuous  by Lemma  \ref{uscphi},
$\Phi$ has a closed $\cald$-pullback absorbing set
$K\in \cald$ by Lemma \ref{lem51},
and $\Phi(t,\tau, \cdot, K(\tau,\cdot)):
\Omega_m \to 2^{\hone \times \ltwo}$
is weakly upper semicontinuous 
by Lemma \ref{lem52}.
As a result,  the existence and uniqueness of $\cald$-pullback
attractors of $\Phi$ follows immediately. 

If $g: \R \to \ltwo$ is $T$-periodic, then one can
check that $\Phi$  is also $T$-periodic:
$\Phi(t,\tau+T, \omega, \cdot)
= \Phi(t,\tau, \omega, \cdot)$
for all $t \in \R^+$,
$\tau \in \R$  and $\omega \in \Omega$.
On the other hand, by \eqref{lem51_1}-\eqref{lem51_2}
we see that
$K(\tau+T, \omega) =K(\tau, \omega)$
for all $\tau\in \R$ and
$\omega \in \Omega$
in the present case.
Therefore, the $T$-periodicity of $\cala$ follows from
Theorem \ref{att} again.
 \end{proof}

\end{document}